\documentclass[12pt,twoside,reqno,centertags,tikz,border=3mm]{amsart}
\usepackage{amsmath,amsthm,amsfonts,amssymb,amsrefs,
ascmac, % Allows to use screen command
mathrsfs, % Allows to use mathrsfs style fonts 
}
\usepackage{enumerate,color}

\usepackage{etoolbox} % Provides functions that seem to offer alternative ways of implementing some LATEX kernel commands
\usepackage{empheq}
\usepackage[pdftex]{graphicx}

\makeatletter

\@addtoreset{equation}{section}
\makeatother

\theoremstyle{plain} % Italic
\newtheorem{thm}{Theorem}[section]%[section] 
\newtheorem{lem}[thm]{Lemma}%[subsection]
\newtheorem{prop}[thm]{Proposition}%[subsection]
\theoremstyle{defi} % Roman
\newtheorem{defi}[thm]{Definition}%[section]
\newtheorem{rem}[thm]{Remark}%[section]
\newcommand{\supp}{\rm{supp}}

\def\supp{\, \hbox{\rm supp}\,  }

\let\tilde=\widetilde

\definecolor{green}{rgb}{0,0.6,0} % 緑
\definecolor{pink}{rgb}{1.0,0.5,0.5} % ピンク
\definecolor{cyan}{rgb}{0.3,1.0,1.0} % シアン
\definecolor{blue}{rgb}{0.0,0.0,1.0} % 青
\definecolor{titcolor}{rgb}{0,0.3242,0.5859} % 東工大カラー
\definecolor{background1}{rgb}{1,0.8,1}
\definecolor{background2}{rgb}{0.8,1,1}
\definecolor{background3}{rgb}{1,1,0.8}

\definecolor{orange}{cmyk}{0,0.6,0.76,0} % オレンジ
\definecolor{violet}{cmyk}{0.07,0.90,0,0.34} % バイオレット

\newcommand{\eqsp}[1]{{\begin{equation}\begin{aligned}#1\end{aligned}\end{
equation}}}

\title[Uniqueness for the Hardy-H\'enon parabolic equation]{Unconditional uniqueness and non-uniqueness for Hardy-H\'enon parabolic equations}
\date\today

\author[N. Chikami, M. Ikeda, K. Taniguchi and S. Tayachi]{Noboru Chikami, Masahiro Ikeda, Koichi Taniguchi and Slim Tayachi}
\address[N. Chikami]
{%Department of Computer Science and Engineering, 
Graduate School of Engineering, 
Nagoya Institute of Technology, 
Gokiso-cho, Showa-ku, Nagoya 
466-8555, Japan.}
\email{chikami.noboru@nitech.ac.jp}
\address[M. Ikeda]
{Faculty of Science and Technology,
Keio University, 
3-14-1 Hiyoshi, Kohoku-ku, Yokohama, 223-8522, Japan/ Center for Advanced Intelligence Project
RIKEN, Japan.}
\email{masahiro.ikeda@keio.jp/masahiro.ikeda@riken.jp}
\address[K. Taniguchi]
{Advanced Institute for Materials Research,
Tohoku University,
2-1-1 Katahira, Aoba-ku, Sendai, 980-8577, Japan.}
\email{koichi.taniguchi.b7@tohoku.ac.jp}
\address[S. Tayachi]
{Universit\'e de Tunis El Manar, Facult\'e des Sciences de Tunis, D\'epartement de Math\'ematiques, Laboratoire \'Equations aux D\'eriv\'ees Partielles LR03ES04, 2092 Tunis, Tunisia.}
\email{slim.tayachi@fst.rnu.tn}


\begin{document}


%%%%%%%%%%%%%%%%%%%%%%%%%%%%%%%%
%%%%%%%%%%%%% Abstract %%%%%%%%%%%%%
%%%%%%%%%%%%%%%%%%%%%%%%%%%%%%%%
\begin{abstract}
We study the problems of uniqueness for Hardy-H\'enon parabolic equations, 
which are semilinear heat equations with the singular potential (Hardy type) or the increasing potential (H\'enon type) in the nonlinear term. 
To deal with the Hardy-H\'enon type nonlinearities, we employ weighted Lorentz spaces as solution spaces. 
We prove unconditional uniqueness and non-uniqueness, and we establish uniqueness criterion for Hardy-H\'enon parabolic equations in the weighted Lorentz spaces.
The results extend the previous works on the Fujita equation and Hardy equations in Lebesgue spaces. 
\end{abstract}

\maketitle

%%%%%%%%%%%%
%%% Section 1 %%%
%%%%%%%%%%%%
\section{Introduction and main results}\label{sec:1}

%%% Subsection 1.1 %%%
\subsection{Introduction and our setting}\label{sub:1.1}
% Setting
We consider the Cauchy problem of the Hardy-H\'enon parabolic equation
\begin{equation}\label{HH}
	\begin{cases}
		\partial_t u - \Delta u = |x|^{\gamma} |u|^{\alpha-1} u,
			&(t,x)\in (0,T)\times \mathbb R^d, \\
		u(0) = u_0 \in L^{q,r}_{s}(\mathbb{R}^d),
	\end{cases}
\end{equation}
where   
$T>0,$ $d\in \mathbb{N}$, $\gamma \in \mathbb R$, $\alpha>1$, $q\in [1,\infty]$, $r \in (0,\infty]$ and $s\in \mathbb{R}$. 
Here, $\partial_t:=\frac{\partial}{\partial t}$ is the time derivative, 
$\Delta:=\sum_{j=1}^d\frac{\partial^2}{\partial x_j^2}$ is the Laplace operator on $\mathbb{R}^d$, 
$u=u(t,x)$ is an unknown complex-valued function on $(0,T)\times \mathbb R^d$, $u_0=u_0(x)$ is a prescribed complex-valued function on $\mathbb R^d$, 
and $L^{q,r}_s(\mathbb{R}^d)$ is the weighted Lorentz space (see Definition \ref{def:WLS}), which includes the Lebesgue space $L^q(\mathbb R^d)= L^{q,q}_0(\mathbb R^d)$ as a special case $r=q$ and $s=0$.  
The equation \eqref{HH} in the case $\gamma = 0$ is 
the {\sl Fujita equation}, 
which has been extensively studied in various directions. 
The equation \eqref{HH} with $\gamma<0$ is known as a 
{\sl Hardy parabolic equation}, while that with $\gamma>0$ 
is known as a {\sl H\'enon parabolic equation}. 
The corresponding stationary problem to \eqref{HH}, that is, 
\[
    -\Delta U=|x|^{\gamma}|U|^{\alpha-1}U,
\]
was proposed by H\'enon as a model to study the rotating 
stellar systems (see \cite{H-1973}), 
and has also been extensively studied in the mathematical context, 
especially in the fields of nonlinear analysis and variational methods 
(see \cite{GhoMor2013} for example). 

% Unconditional uniqueness and non-uniqueness
In this paper we study the problem on 
unconditional uniqueness and non-uniqueness for \eqref{HH} in weighted Lorentz spaces $L^{q,r}_s(\mathbb R^d)$.
Here, {\sl unconditional uniqueness} means uniqueness of the solution to \eqref{HH} 
for any initial data $u_0 \in L^{q,r}_s(\mathbb R^d)$
in the sense of the integral form 
\begin{equation}\label{integral-equation}
u(t)
= e^{t\Delta}u_0 + \int_{0}^t e^{(t-\tau)\Delta}(|\cdot|^{\gamma} |u(\tau)|^{\alpha-1} u(\tau))\, d\tau
\end{equation}
in $L^\infty(0,T ; L^{q,r}_s(\mathbb R^d))$ or $C([0,T] ; L^{q,r}_s(\mathbb R^d))$,
where $T>0$ and $\{e^{t\Delta}\}_{t>0}$ is the heat semigroup.
We say that {\sl non-uniqueness} holds for \eqref{HH}
if unconditional uniqueness fails. 
In contrast,
we say that {\sl conditional uniqueness} holds if 
uniqueness of the solution to \eqref{HH} holds in the entire space with some {\sl auxiliary function spaces}.
In addition, we also study {\sl uniqueness criterion} which is a necessary and sufficient condition on the Duhamel term (i.e. the second term in the right-hand side of \eqref{integral-equation}) for uniqueness to hold. 

% Purpose, Related work and positioning of our work
Let us here state previous works on uniqueness for \eqref{HH}. 
For \eqref{HH} with $\gamma\le 0$, 
the problem on uniqueness 
has been well studied (see \cite{Bal1977,Bar1983,BenTayWei2017,BreCaz1996,Chi2019,CIT2021,HarWei1982,Gig1986,MatosTerrane,NS1985,Tak2021,Tay2020,Ter2002,Wei1980,Wei1981} for example). 
In the study of unconditional uniqueness for \eqref{HH} in Lebesgue spaces $L^q(\mathbb R^d)$ or Lorentz spaces $L^{q,r}(\mathbb R^d)$, the following two critical exponents are known to be important. 
The first one is the so-called scale-critical exponent $q_c$ given by
\begin{equation}\label{q_c}
q_c = q_c (d,\gamma,\alpha) :=  \frac{d(\alpha-1)}{2+\gamma},
\end{equation}
and we say that the problem \eqref{HH} is {\sl scale-critical} if $q=q_c$, {\sl scale-subcritical} if $q>q_c$, and {\sl scale-supercritical} if $q<q_c$. 
The second one is the critical exponent $Q_c$ given by
\begin{equation}\label{Q_c}
Q_c =Q_c(d,\gamma,\alpha) := \frac{d\alpha}{d+\gamma},
\end{equation}
which is related to well-definedness of 
the Duhamel term in \eqref{integral-equation} in $L^{q,r}(\mathbb R^d)$. 
In fact, the nonlinear term $|x|^\gamma |u|^{\alpha-1}u \in L^1_{\mathrm{loc}}(\mathbb R^d)$ for any $u \in L^{q,r}(\mathbb R^d)$ 
if and only if ``$q>Q_c$" or ``$q=Q_c$ and $r \le \alpha$".
In the case $\gamma=0$, 
unconditional uniqueness for \eqref{HH} in $C([0,T] ; L^q(\mathbb R^d))$  was proved in the double subcritical case $q > \max\{q_c, Q_c\}$ by 
Weissler \cite{Wei1980} and in the single critical cases $q = Q_c > q_c$ and $q = q_c>Q_c$ by
Brezis and Cazenave \cite{BreCaz1996}. 
In the double critical case $q = q_c = Q_c$, non-uniqueness was proved for {\sl some} initial data $u_0 \in L^q(\mathbb R^d)$ by Terraneo \cite{Ter2002},
and then, 
for {\sl any} initial data $u_0 \in L^q(\mathbb R^d)$ by Matos and Terraneo \cite{MatosTerrane}. In \cite{Ter2002}, uniqueness criterion was also obtained in the double critical case. 
Recently, Takahashi \cite{Tak2021} proved the existence of an uncountably infinite number of solutions to \eqref{HH} with moving singularities for some initial data in the double critical case.
In the scale-supercritical case $q<q_c$, non-uniqueness for \eqref{HH} was proved for initial data $u_0=0$ by
Haraux and Weissler \cite{HarWei1982}. 
Uniqueness and non-uniqueness have also been studied for heat equations with exponential nonlinearities (see \cite{IKNW2021,IRT2022} and references therein).
In the Hardy case $-\min \{2,d\} < \gamma < 0$,
similar results were obtained by \cite{BenTayWei2017,Tay2020}, where
the Lorentz spaces $L^{q,r}(\mathbb R^d)$ is used to study unconditional uniqueness in the critical case $q=Q_c$ in \cite{Tay2020}.
The above previous works are summarized in Figure 1.
In contrast, the H\'enon case $\gamma >0$ has not been well studied. 
This is due to the difficulty of treating the increasing potential $|x|^\gamma$ in the nonlinear term at infinity. 
To overcome this difficulty, the weighted spaces are effective, and recently, conditional uniqueness was obtained in $L^{q}_s(\mathbb R^d)=L^{q,q}_s(\mathbb R^d)$ in \cite{CIT2022}; 
however, unconditional uniqueness and non-uniqueness are completely open.
The main purpose of this paper is 
to prove unconditional uniqueness, non-uniqueness and uniqueness criterion for \eqref{HH} with all $\gamma > -\min \{2,d\}$, including the H\'enon case, 
in $L^{q,r}_s(\mathbb R^d)$. 

% Figure 1
\begin{figure}[t]
\begin{center}\label{figure1}
\includegraphics[width=130mm]{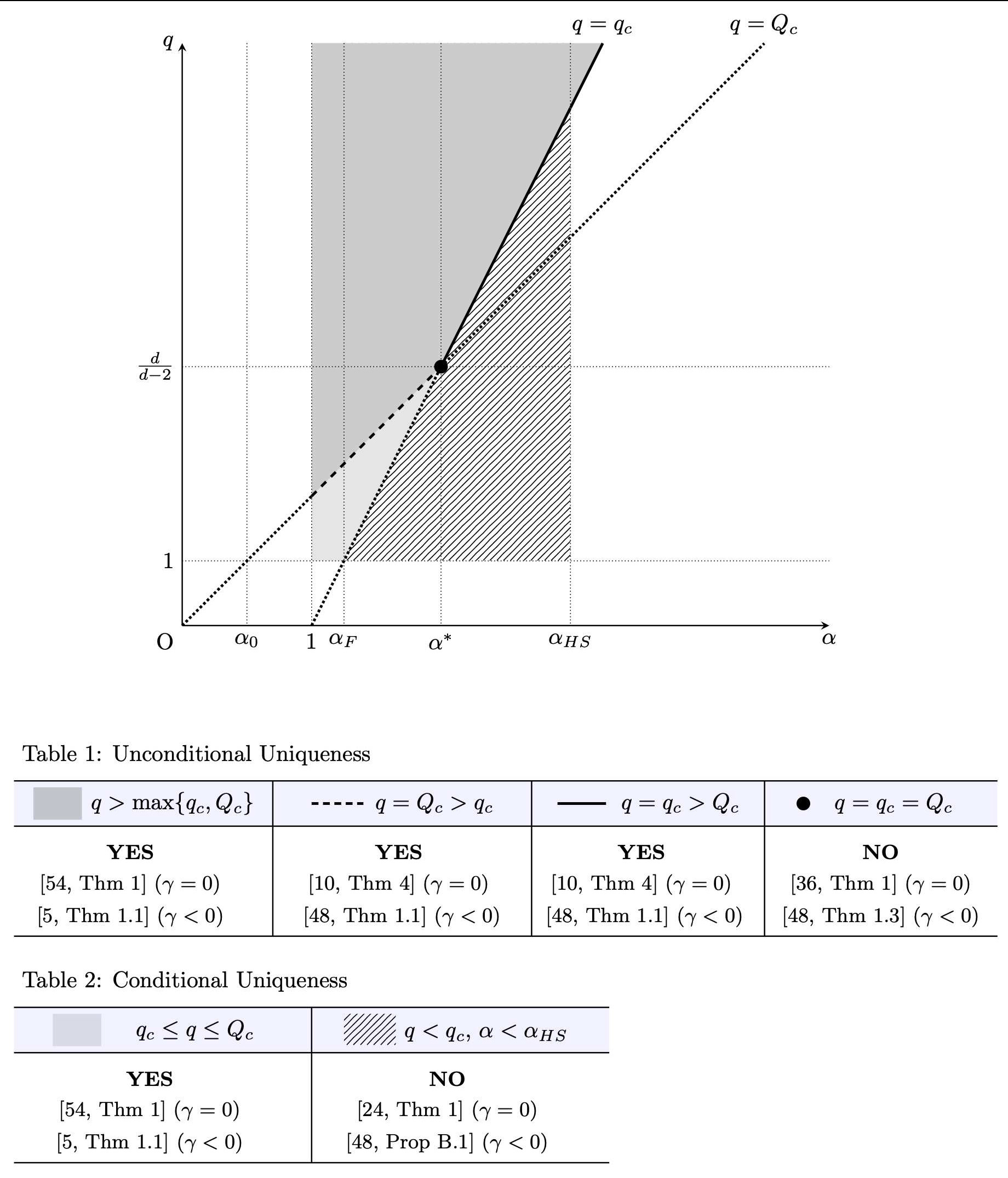}
\caption{The figure shows the domain of $(\alpha,q)$ for $d\ge 3$ and $\gamma \le 0$, where 
$\alpha_0 := 1 + \frac{\gamma}{d}$, $\alpha_F := 1 + \frac{2+\gamma}{d}$ is the Fujita exponent, $\alpha^*:= \frac{d+\gamma}{d-2}$ is the Serrin exponent and $\alpha_{HS} := \frac{d+2+2\gamma}{d-2}$ is the Hardy-Sobolev exponent.
Table 1 and Table~2 summarize the previous results on uniqueness for \eqref{HH} with $\gamma \le 0$. 
}
\end{center}
\end{figure}

%%% Subsection 1.2 %%%
\subsection{Statement of the results}\label{sub:1.2}
% Notion of solution, definitions, notation
To describe our results, let us give some definitions and notation. 
For $T \in (0,\infty]$ and a quasi-normed space $X$, 
we denote by $L^\infty(0,T ;X)$ the space of functions $u : (0,T) \to X$ such that 
\[
\|u\|_{L^\infty(0,T ;X)} :=
\underset{t \in (0,T)}{\mathrm{ess\, sup}}\, \|u(t)\|_{X} < \infty,
\]
and by $C([0,T] ; X)$ the space of continuous functions $u : [0,T] \to X$ with respect to the quasi-norm of $X$.
The space $\mathcal L^{q,r}_s (\mathbb R^d)$ is defined as the completion of $L^{q,r}_s (\mathbb R^d) \cap L^\infty_0(\mathbb R^d)$ with respect to $\|\cdot\|_{L^{q,r}_s}$,
where $L^\infty_0(\mathbb R^d)$ denotes the set of all functions in $L^\infty(\mathbb R^d)$ with compact support in $\mathbb R^d$ (see Definition \ref{def:WLS}). 

% Definition 1.1
\begin{defi}
Let $T>0$ and $X = L^{q,r}_s (\mathbb R^d)$ or $\mathcal L^{q,r}_s (\mathbb R^d)$. 
We say that a function $u=u(t,x)$ on $(0,T) \times \mathbb R^d$ is a {\sl mild solution} to \eqref{HH} with initial data $u_0 \in X$ in $C([0,T] ; X)$ ($L^\infty(0,T ; X)$ resp.) 
if $u$ belongs to $C([0,T] ; X)$ ($L^\infty(0,T ; X)$ resp.) and satisfies the integral equation \eqref{integral-equation} for almost everywhere $(t,x) \in (0,T)\times \mathbb R^d$.
\end{defi}

Note that the Duhamel term in \eqref{integral-equation} converges in $L^{q,r}_s(\mathbb R^d)$ under conditions on functions $u=u(t,x)$ and parameters $q,r,s$ given in Lemma \ref{lem:sec4_nonlinear1} or Lemma \ref{lem:sec4_nonlinear2}. \\

We define two critical cases in the framework of $L^{q,r}_s(\mathbb R^d)$ in a similar manner to $q_c$ and $Q_c$, respectively.
The equation \eqref{HH} is invariant under the following scale transformation:
\[
u_\lambda(t,x) := \lambda^\frac{2+\gamma}{\alpha-1}u(\lambda^2 t, \lambda x),\quad \lambda>0.
\]
More precisely, if $u$ is a solution to \eqref{HH}, then so is $u_\lambda$
with the rescaled initial data $\lambda^\frac{2+\gamma}{\alpha-1}u_0(\lambda x)$. 
Moreover, we calculate 
\[
\|u_\lambda(0)\|_{L^{q,r}_s} = \lambda^{-s + \frac{2+\gamma}{\alpha-1}- \frac{d}{q}}
\|u_0\|_{L^{q,r}_s}= \lambda^{-d (\frac{s}{d} + \frac{1}{q}- \frac{1}{q_c})}
\|u_0\|_{L^{q,r}_s},\quad \lambda>0.
\]
Hence, if $q$ and $s$ satisfy 
\[
\frac{s}{d} + \frac{1}{q} = \frac{1}{q_c},
\]
then
$
\|u_\lambda(0)\|_{L^{q,r}_{s}} 
=
\|u_0\|_{L^{q,r}_{s}}$ for any $\lambda>0$, 
i.e., the norm $\|u_\lambda(0)\|_{L^{q,r}_{s}} $ is invariant with respect to $\lambda$. 
Therefore, 
we say that the problem \eqref{HH} is {\sl scale-critical} if $\frac{s}{d} + \frac{1}{q} = \frac{1}{q_c}$, 
{\sl scale-subcritical} if $\frac{s}{d} + \frac{1}{q} < \frac{1}{q_c}$, and {\sl scale-supercritical} if $\frac{s}{d} + \frac{1}{q} > \frac{1}{q_c}$.
Another critical case is when the following holds:
\[
\frac{s}{d} + \frac{1}{q} = \frac{1}{Q_c}. 
\]
This is related to local integrability of the nonlinear term $|x|^\gamma |u|^{\alpha-1}u$. 
In fact, $|x|^\gamma |u|^{\alpha-1}u \in L^1_{\mathrm{loc}}(\mathbb R^d)$ for any $u \in L^{q,r}_s(\mathbb R^d)$ 
if and only if 
\begin{equation}\label{integral:S_c}
\frac{s}{d} + \frac{1}{q} < \frac{1}{Q_c} \quad \text{or} 
\quad \frac{s}{d} + \frac{1}{q} = \frac{1}{Q_c} \text{ and } r\le \alpha. 
\end{equation}
Then, it is ensured for the Duhamel term in \eqref{integral-equation} to be well-defined in $L^{q,r}_s(\mathbb R^d)$. 

In terms of the two critical cases, 
we divide the problem into the following four cases: {\sl Double subcritical case} ($\frac{s}{d} + \frac{1}{q} < \min\{\frac{1}{q_c}, \frac{1}{Q_c}\}$), 
{\sl single critical case I} ($\frac{s}{d} + \frac{1}{q} = \frac{1}{Q_c} < \frac{1}{q_c}$), {\sl single critical case II} ($\frac{s}{d} + \frac{1}{q} = \frac{1}{q_c} <\frac{1}{Q_c}$), and {\sl double critical case} ($\frac{s}{d} + \frac{1}{q} = \frac{1}{q_c} = \frac{1}{Q_c}$).
Moreover, we define the exponent
$\alpha^*$ by 
\[
\alpha^*
=\alpha^* (d,\gamma):=
\begin{dcases}
\frac{d+\gamma}{d-2}\quad &\text{if }d\ge3,\\
\infty&\text{if }d=1,2,\\
\end{dcases}
\]
which is often referred to the Serrin exponent
(see 
\cite{Ser1964,Ser1965} and also \cite{GidSpr1981}).
The exponents $\alpha^*$, $q_c$ and $Q_c$ are related as follows:
\[
\alpha \lesseqgtr \alpha^*
\quad \text{if and only if }\quad q_c \lesseqgtr Q_c.
\]

% Our results on unconditional uniqueness
In our results on unconditional uniqueness below, 
we assume that 
\begin{empheq}[left={\empheqlbrace}]{alignat=2} \label{assum:main}
\begin{split}
&d\in \mathbb N, 
\quad \gamma > -\min\{2,d\}, \quad \displaystyle \alpha >  \max\left\{1,1 + \frac{\gamma}{d}\right\},\\
& \alpha \le q \le \infty,\quad \frac{\gamma}{\alpha-1} \le s < d,\quad 0<r \le \infty.
\end{split}
\end{empheq}
Our results on unconditional uniqueness are the following:

% Theorem 1.2
\begin{thm}[Scale-subcritical case]\label{thm:unconditional1}
Let $T>0$, and let $d,\gamma,\alpha,q,r,s$ be as in \eqref{assum:main}. 
Assume either $(1)$ or $(2)$:
\begin{enumerate}[\rm (1)]
\item {\rm (Double subcritical case)} 
$r \le \alpha$ if $q=\alpha$, and 
$0 < \frac{s}{d} + \frac{1}{q} < \min\{\frac{1}{q_c}, \frac{1}{Q_c}\}$.

\item {\rm (Single critical case I)}
$\alpha < \alpha^*$, $q\not =\infty$, 
$r\le \alpha$ and $\frac{s}{d} + \frac{1}{q} = \frac{1}{Q_c} < \frac{1}{q_c}$. 

\end{enumerate}
Then unconditional uniqueness holds for \eqref{HH} in $L^\infty(0,T ; L^{q,r}_s(\mathbb R^d))$. 
\end{thm}

% Theorem 1.3
\begin{thm}[Scale-critical case]\label{thm:unconditional2}
Let $T>0$, and let $d,\gamma,\alpha,q,r,s$ be as in \eqref{assum:main}.  
Assume $d\ge 3$, $q\not =\infty$, and either $(1)$ or $(2)$:

\begin{enumerate}[\rm (1)]
\item {\rm (Single critical case II)} $\alpha > \alpha^*$ and $\frac{s}{d} + \frac{1}{q} = \frac{1}{q_c} < \frac{1}{Q_c}$  (replace $L^{q,\infty}_{s}(\mathbb R^d)$ by $\mathcal L^{q,\infty}_{s}(\mathbb R^d)$ if $r=\infty$).

\item {\rm (Double critical case)} $\alpha = \alpha^*$, $r\le \alpha^* -1$ and $\frac{s}{d} + \frac{1}{q} = \frac{1}{q_c} = \frac{1}{Q_c}$.
\end{enumerate}
Then unconditional uniqueness holds for \eqref{HH} in $C([0,T] ; L^{q,r}_s(\mathbb R^d))$.
\end{thm}

% Remark 1.4
\begin{rem}
In Theorem \ref{thm:unconditional1} {\rm (1)},
the condition ``$r \le \alpha$ if $q=\alpha$" comes from the restriction on parameters in linear estimates. 
More precisely, the condition is due to the restriction $r_1=1$ for linear estimates with $q_1=1$ in Proposition \ref{prop:linear-main} (see \eqref{linear-condi3} and also Lemma \ref{lem:sec4_nonlinear1} {\rm (ii)}).
\end{rem}

% Single critical case I and Double critical case
Next, we consider the following two cases where the unconditional uniqueness is not obtained in the above theorems: $r>\alpha$ in the single critical case I; 
$r>\alpha^*-1$ in the double critical case. 

In the single critical case I,
the condition $r \le \alpha$ naturally appears from the viewpoint of well-definedness of mild solutions to \eqref{HH} as seen in \eqref{integral:S_c}. 
On the other hand, when $r>\alpha$, we can define mild solutions to \eqref{HH} with the auxiliary condition and we know that conditional uniqueness holds (see \cite[Theorem 1.13]{CIT2022}). 
We are interested in the questions whether the conditional uniqueness can be improved.
In fact, we can give the following sufficient condition for uniqueness to hold which improves the conditional uniqueness \cite[Theorem 1.13]{CIT2022}. 

% Proposition 1.5
\begin{prop}\label{prop:uniqueness-sufficient}
Let $T>0$, and let $d,\gamma,\alpha,q,r,s$ be as in \eqref{assum:main}. 
Assume that $\alpha < \alpha^*$, $q\not =\infty$, $\alpha < r \le \infty$, and $\frac{s}{d} + \frac{1}{q} = \frac{1}{Q_c} < \frac{1}{q_c}$.
Let $u_0\in L^{q,r}_{s}(\mathbb R^d)$. Then,
if $u_1, u_2 \in L^\infty(0,T ; L^{q, r}_{s}(\mathbb R^d))$ 
are mild solutions to \eqref{HH} with $u_1(0)=u_2(0)=u_0$ such that 
\[
u_i(t) - e^{t\Delta} u_0 \in L^\infty(0,T ; L^{q, r'(\alpha-1)}_{s}(\mathbb R^d))
\quad \text{for $i=1,2$},
\]
then $u_1 = u_2$ on $[0,T]$. Here, $r'$ is the H\"older conjugate of $r$, i.e.,
$1=\frac{1}{r}+\frac{1}{r'}$.
\end{prop}

In the double critical case, we prove the result on non-uniqueness for \eqref{HH} if $\alpha^*-1< r \le \infty$.
More precisely, we have the following:

% Theorem 1.6 
\begin{thm}[Double critical case]\label{thm:nonuniqueness}
Let $d\ge3$, $\gamma>-2$, $\alpha = \alpha^*$, 
$\alpha^* \le q < \infty$, 
$\alpha^*-1< r \le \infty$, and $\frac{s}{d} + \frac{1}{q} = \frac{1}{q_c} = \frac{1}{Q_c}$. 
Then, for any initial data $u_0 \in L^{q,r}_{s}(\mathbb{R}^d)$, there exists $T=T(u_0)>0$ such that 
the problem \eqref{HH} has at least two different solutions in $C([0,T]; L^{q,r}_{s}(\mathbb{R}^d))$
(replace $L^{q,r}_{s}(\mathbb R^d)$ by $\mathcal L^{q,\infty}_{s}(\mathbb R^d)$ if $r=\infty$). 
\end{thm}

By Theorem \ref{thm:unconditional2} {\rm (2)} and Theorem~\ref{thm:nonuniqueness}, 
we clarify that the exponent $r=\alpha^*-1$ is a threshold of dividing unconditional uniqueness and non-uniqueness for \eqref{HH} in the double critical case. 
The importance of $r=\alpha^*-1$ was pointed out by \cite[Theorem 0.10 and Proposition 5.4]{Ter2002} in the Fujita case $\gamma=0$ 
(see \cite[Theorem 1.4 and Proposition 8.2]{Tay2020} for the Hardy case $\gamma<0$).
The idea of proof of Theorem \ref{thm:nonuniqueness} is based on the method by \cite{Ter2002,MatosTerrane}, i.e., we construct two different solutions which are regular and singular at $x=0$ to \eqref{HH} for any initial data $u_0$. 
The regular solution can be found in a similar way to \cite{CIT2022} and the singular solution can be constructed from the singular stationary solution to 
\[
	\Delta U + |x|^{\gamma} U^{\frac{d+\gamma}{d-2}} = 0\quad 
	\text{in } B\setminus \{0\}, \quad U>0,
\]
where $B := \{x \in \mathbb R^d \,;\, |x| < 1\}$.
The threshold $r=\alpha^*-1$ comes essentially from the logarithmic rate of the singularity at $x=0$ of the singular stationary solution (see \eqref{eq.singular1} and \eqref{eq.singular2} in Subsection \ref{sub:5.2}). 
The existence and behavior near the origin of singular stationary solutions have been studied in \cite{Avi1983,Avi1987,BidGar2001,BidRao1996,DZarxiv,GidSpr1981,GN2022,GueVer1988,Ser1964,Ser1965} for instance. 
See Subsection \ref{sub:5.2} for the details.

In addition, we give the following uniqueness criterion. 

% Theorem 1.7
\begin{thm}\label{thm:uniqueness-criterion}
Let $T>0$, and let $d,\gamma,\alpha,q,r,s$ be as in \eqref{assum:main}. 
Assume that $d\ge3$, $\gamma>-2$, $\alpha = \alpha^*$, 
$\alpha^* \le q < \infty$, 
$\alpha^*-1< r \le \infty$, and $\frac{s}{d} + \frac{1}{q} = \frac{1}{q_c} = \frac{1}{Q_c}$. 
Let $u_0\in L^{q, r}_{s}(\mathbb{R}^d)$. Then, 
if $u_1,u_2 \in C([0,T] ; L^{q, r}_{s}(\mathbb R^d))$ 
are mild solutions to \eqref{HH} with $u_1(0)=u_2(0)=u_0$ such that 
\begin{equation}\label{DC-criterion}
u_i(t) - e^{t\Delta} u_0 \in 
C([0,T] ; L^{q, \alpha^*-1}_{s}(\mathbb R^d))
\quad \text{for $i=1,2$},
\end{equation}
then $u_1 = u_2$ on $[0,T]$ 
(replace $L^{q, r}_{s}(\mathbb R^d)$ by $\mathcal{L}^{q, \infty}_{s}(\mathbb{R}^d)$ if $r=\infty$). 
\end{thm}

% Remark 1.8
\begin{rem} 
The exponent $r = \alpha^*-1$ in \eqref{DC-criterion} of Theorem \ref{thm:uniqueness-criterion} is optimal for the same reason as above (see Theorem \ref{thm:singular_sol}).
\end{rem}

% Scale-supercritical case
In the scale-supercritical case, we have the following result on non-uniqueness for \eqref{HH}. 
Here, we define the exponents $\alpha_F$ and $\alpha_{HS}$ by 
\[
\alpha_F = \alpha_F(d,\gamma) := 1+ \frac{2+\gamma}{d}
\quad \text{and}\quad 
\alpha_{HS} = \alpha_{HS}(d,\gamma) := \frac{d+2+2\gamma}{d-2},
\]
which are often referred to as the Fujita exponent (see \cite{Qi1998,Pin1997}) and the critical Hardy-Sobolev exponent (see \cite{Lie1983}). 

% Proposition 1.9
\begin{prop}[Scale-supercritical case]\label{prop:non-uniquness_super}
Let $d\ge 3$, $\gamma>-2$, $\alpha > 1$, $1<q\le \infty$, $1\le r\le \infty$ and $s\in \mathbb R$ be such that
\[
\gamma \le 
\begin{cases}
\sqrt{3}-1 &\text{if }d=3,\\
0 &\text{if }d\ge4,
\end{cases}
\quad 
\alpha_F <
\alpha<\alpha_{HS}
\quad \text{and}\quad 
\frac{1}{q_c} < \frac{s}{d} + \frac{1}{q} < 1.
\]
Then 
the equation \eqref{HH} has a global positive solution in $C([0,\infty); L^{q,r}_s(\mathbb{R}^d))$ with initial data $0$. 
\end{prop}

To visually understand our above results, we give Figure 2 for the case $\gamma<0$ and $\min\{{1\over q_c},{1\over Q_c}\}<\max\{{1\over q_c},{1\over Q_c}\}$.\\

% Figure 2
\begin{figure}[t]
\begin{center}\label{figure2}
\includegraphics[width=130mm]{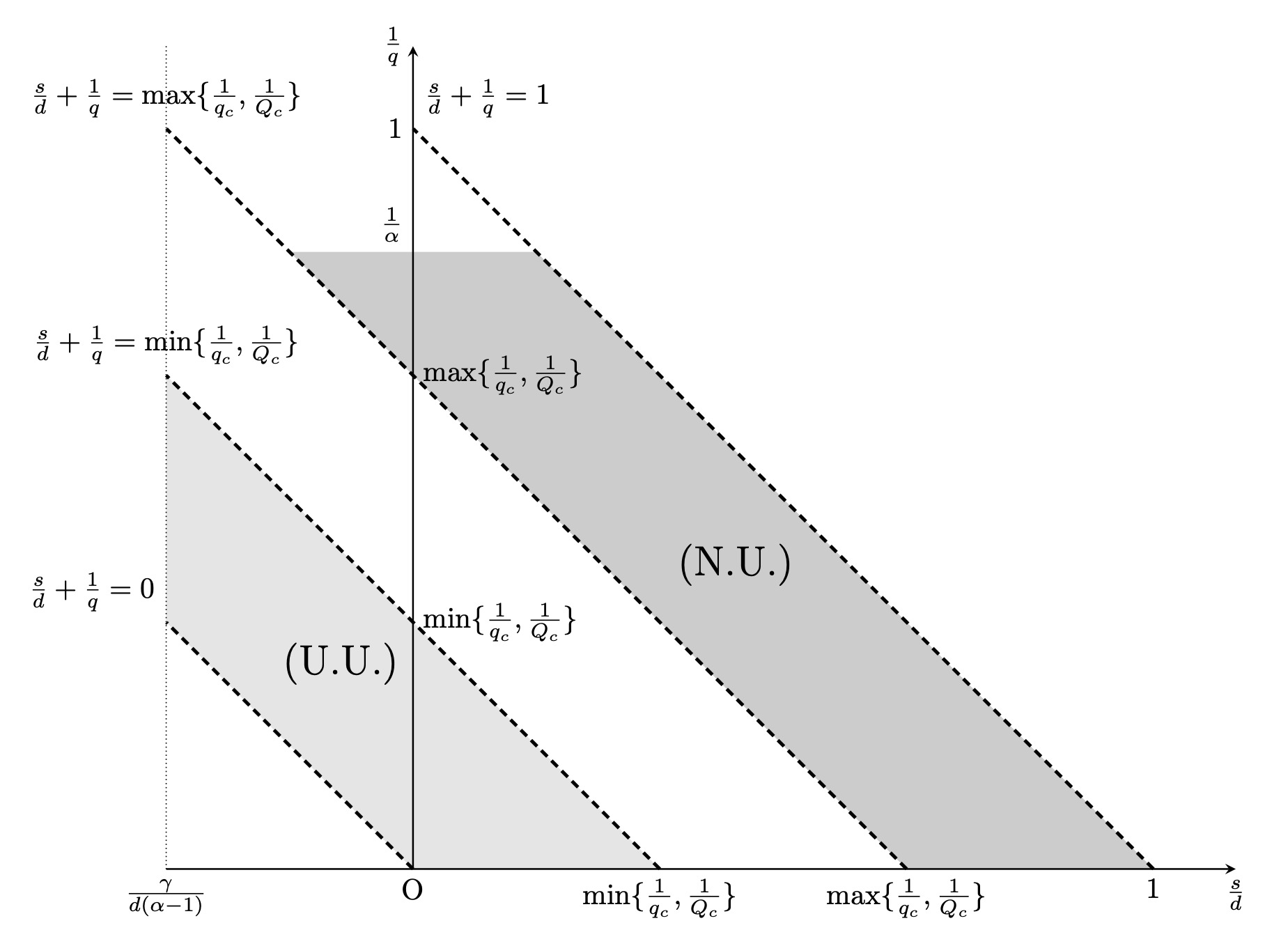}
\caption{
The figure shows the domain of $({s\over d},{1\over q})$ 
in the case $\gamma<0$ and $\min\{{1\over q_c},{1\over Q_c}\}<\max\{{1\over q_c},{1\over Q_c}\}$. 
 (U.U.) and (N.U.) mean unconditional uniqueness  and non-uniqueness, respectively.  
The cases $\gamma=0$ and $\gamma>0$ are deduced by moving the line ${s\over d}={\gamma\over d(\alpha-1)}$ to the right. 
}
\end{center}
\end{figure}

% Figure 3
\begin{figure}[h]
\begin{center}\label{figure3}
\includegraphics[width=130mm]{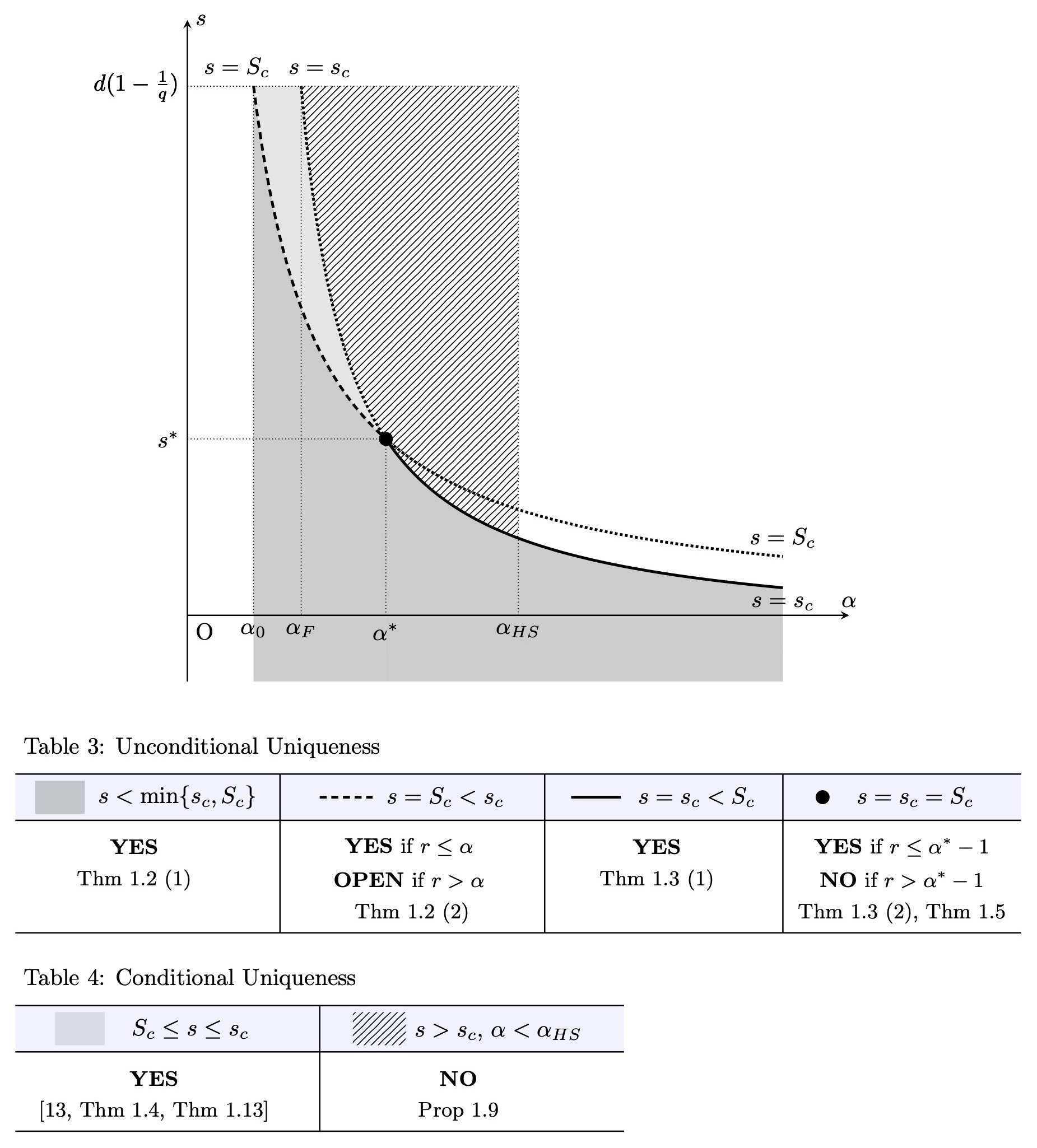}
\caption{The figure shows the domain of $(\alpha,s)$ for $d\ge 3$ and $q >1$. 
Here, $\alpha_0 := \min\{1,1 + \frac{\gamma}{d}\}$, $\alpha_F, \alpha^*, \alpha_{HS}$ are given in Figure 1,
$s_c, S_c$ are given in \eqref{def:s_c-S_c}, and 
$s^* := d-2-\frac{d}{q}$. 
Table 3 and Table~4 summarize our results on uniqueness for \eqref{HH}. 
}
\end{center}
\end{figure} 

% Comparison between our results and previous results
Herein, we compare our results with previous ones. 
Our results generalize the previous works \cite{BreCaz1996,HarWei1982,Wei1980,BenTayWei2017,MatosTerrane,Tay2020,Ter2002}, since $s$ can be taken as $s=0$ if $\gamma \le 0$ in our results. 
More precisely, our results on unconditional uniqueness (Theorem~\ref{thm:unconditional1} and Theorem \ref{thm:unconditional2}  {\rm (1)}) 
include the results in 
\cite[Theorem 4]{Wei1980} and \cite[Theorem 4]{BreCaz1996} ($\gamma=0$ and $s=0$) and 
\cite[Theorem~1.1]{BenTayWei2017} and \cite[Theorem~1.1]{Tay2020} ($\gamma <0$ and $s=0$), and 
our result on non-uniqueness (Theorem \ref{thm:nonuniqueness}) generalizes the previous works \cite[Theorem~1]{MatosTerrane} ($\gamma=0$ and $s=0$) and 
\cite[Theorem 1.3]{Tay2020} ($\gamma <0$ and $s=0$).
Moreover, 
our results on the double critical case (Theorem \ref{thm:unconditional2} {\rm (2)} and Theorem~\ref{thm:nonuniqueness}) also clarify the threshold $r=\alpha^*-1$ of dividing unconditional uniqueness and non-uniqueness. 
Regarding the uniqueness criterion, 
Theorem~\ref{thm:uniqueness-criterion} generalizes the previous works \cite[Theorem 0.10]{Ter2002} ($\gamma=0$ and $s=0$) and 
\cite[Theorem 1.4]{Tay2020} ($\gamma <0$ and $s=0$), and Proposition \ref{prop:uniqueness-sufficient} has not been mentioned in the previous works.
In the scale-supercritical case, Proposition~\ref{prop:non-uniquness_super} 
corresponds to \cite[Theorem 1]{HarWei1982} ($\gamma =0$ and $s=0$) and \cite[Proposition B.1]{Tay2020}  ($\gamma <0$ and $s=0$).

To easily compare our results with the previous work \cite{CIT2022} which includes the H\'enon case $\gamma>0$, 
we can rewrite our results by using the two critical exponents on $s$:
\begin{equation}\label{def:s_c-S_c}
	s_c=s_c(d,\gamma,\alpha, q) 
		:= \frac{2+\gamma}{\alpha-1} - \frac{d}{q}
	\quad \text{and}\quad S_c=S_c(d,\gamma,\alpha,q) := \frac{d+\gamma}{\alpha} - \frac{d}{q}. 
\end{equation}
The exponents $s_c$ and $S_c$ correspond to $q_c$ and $Q_c$ in the case without weights, respectively. 
In fact, we can see that 
\[
\text{$s_c=0$ if and only if $q=q_c$}
\quad \text{and}\quad 
\text{$S_c=0$ if and only if $q=Q_c$}.
\]
Hence, we can also say that the problem \eqref{HH} is {\sl scale-critical} if $s=s_c$, 
{\sl scale-subcritical} if $s<s_c,$ and {\sl scale-supercritical} if $s>s_c$.
Moreover, the four cases can be rewritten as follows: 
{\sl Double subcritical case} ($s< \min \{s_c, S_c \}$), 
{\sl single critical case I} ($s = S_c < s_c$), {\sl single critical case II} ($s = s_c < S_c$), and {\sl double critical case} ($s = s_c = S_c$).
The results in \cite{CIT2022} show local well-posedness, including the conditional uniqueness, for \eqref{HH} if $s\le s_c$ and non-existence of positive mild solution to \eqref{HH} for some initial data $u_0 \ge0$ if $s>s_c$. However, unconditional uniqueness and non-uniqueness are not mentioned in \cite{CIT2022}. 
Our results are summarized in Figure~3.\\

% Organization of the paper
This paper is organized as follows. In Section \ref{sec:2}, 
we summarize the definitions and fundamental lemmas on Lorentz spaces and weighted Lorentz spaces. 
In Section \ref{sec:3}, 
we establish the two kinds of weighted linear estimates.
In Subsection \ref{sub:3.1},
we extend the usual $L^{q_1}$-$L^{q_2}$ estimates to the weighted Lorentz spaces, which are fundamental tools in this paper. 
In Subsection \ref{sub:3.2},
we prove a certain space-time estimate in the weighted Lorentz spaces. We call it the weighted Meyer inequality. 
This inequality corresponds to a certain endpoint case of the weighted Strichartz estimates, and 
it is an important tool in studying the scale-critical case.
In Section~\ref{sec:4}, 
we prove our results on unconditional uniqueness and uniqueness criterion (Theorem~\ref{thm:unconditional1}, Theorem~\ref{thm:unconditional2}, Proposition \ref{prop:uniqueness-sufficient} and Theorem \ref{thm:uniqueness-criterion}), based on the weighted linear estimates.
In Section \ref{sec:5}, we prove our result on non-uniqueness (Theorem \ref{thm:nonuniqueness}). 
In Section \ref{sec:6}, we discuss the non-uniqueness in the scale-supercritical case and prove Proposition~\ref{prop:non-uniquness_super}. 
In Section \ref{sec:8}, we give a remark on the number of solutions in the double critical case, and 
additional results on the critical singular case $\gamma = -\min \{2,d\}$ and the exterior problem on domains not containing the origin.

\subsection*{Notation}
Throughout this paper, we use the notation $C$ for a positive constant which may change from line to line for convenience.
We use the symbols $a \lesssim b$ and $b \gtrsim a$ for $a, b \ge0$ which mean that there exists a constant $C>0$ such that $a \le C b$.
The symbol $a \sim b$ means that $a \lesssim b$ and $b \lesssim a$ happen simultaneously. 
We denote by $\overline{\Omega}$ the closure of a domain $\Omega$ in $\mathbb R^d$. 
For $a \in \mathbb R$ and a sequence $\{a_n\}_{n\in \mathbb N} \subset \mathbb R$, the symbol $a_n \nearrow a$ as $n\to \infty$ means that $a_n \le a_{n+1}$ for any $n \in \mathbb N$ and $a_n \to a$ as $n\to \infty$.
For functions $f$ and $g$, the symbol $f \ast g$ denotes the convolution of $f$ and $g$:
\[
(f \ast g)(x) := \int_{\mathbb R^d} f(x-y) g(y)\, dy,\quad x \in \mathbb R^d.
\]
For quasi-normed spaces $X$ and $Y$,
the notation 
$\|\cdot\|_{X \to Y}$ denotes the operator norm from $X$ to $Y$, i.e.,
\[
\|T\|_{X\to Y}
:=
\sup_{\|f\|_{X} = 1} \|Tf\|_{Y}
\]
for an operator $T$ from $X$ into $Y$, and 
the notation $X \hookrightarrow Y$ denotes that $X$ is continuously embedded in $Y$, i.e., $X$ is a subset of $Y$ and there exists a constant $C>0$ such that 
\[
\|f\|_{Y} \le C\|f\|_{X} \quad \text{for any }f\in X.
\]
For a domain $\Omega$ in $\mathbb R^d$, 
we denote by $C_0^\infty(\Omega)$ the set of all $C^\infty$-functions 
having compact support in $\Omega$, 
by 
$L^0(\Omega)$ the set of all Lebesgue measurable functions on $\Omega$, by $L^\infty_0(\Omega)$ the set of all functions in $L^\infty(\Omega)$ with compact support in $\Omega$,
and 
by $\mathcal S'(\mathbb R^d)$ the space of tempered distributions on $\mathbb R^d$.

%%%%%%%%%%%%
%%% Section 2 %%%
%%%%%%%%%%%%
\section{Weighted Lorentz spaces}\label{sec:2}
We define the distribution function $d_f$ of a function $f$ by
\[
	d_f (\lambda) := \left| \left\{ x \in \Omega \,;\, |f(x)| > \lambda \right\} \right|,
\]
where $|A|$ denotes the Lebesgue measure of a set $A$. 

% Definition 2.1
\begin{defi}
For $0<q,r \le \infty$, the Lorentz space $L^{q,r}(\Omega)$ is defined by 
\[
L^{q,r}(\Omega):=\left\{ f \in L^0(\Omega) \,;\, 
	\|f\|_{L^{q,r}(\Omega)}
	< \infty \right\}
\]
endowed with a quasi-norm 
\[
\|f\|_{L^{q,r}(\Omega)}
:=
\begin{cases}
\displaystyle \left( 
\int_0^\infty ( t^{\frac{1}{q}} f^*(t) )^r\frac{dt}{t}
\right)^{\frac{1}{r}}
\quad & \text{if } r<\infty,\\
\displaystyle \sup_{t>0} t^\frac{1}{q} f^*(t) & \text{if }r=\infty,
\end{cases}
\]
where $f^*$ is the decreasing rearrangement of $f$ given by 
\[
	f^*(t) := \inf \{ \lambda > 0 \,; \, d_f(\lambda) \le t \}.
\]
\end{defi}

We refer to \cite{Gr2008}
for the properties of the distribution function, the decreasing rearrangement and the Lorentz space. 

% Remark 2.2
\begin{rem}
For $0<q,r<\infty$, the quasi-norm of $L^{q,r}(\Omega)$ is equivalent to
\[
	\|f\|_{L^{q,r}(\Omega)} = q^{\frac1{r}} \left( \int_0^{\infty} (d_f (\lambda)^{\frac1{q}} \lambda )^r \frac{d\lambda}{\lambda} \right)^{\frac1{r}}.
\]
For $0<q<\infty$ and $r=\infty$, 
\[
\begin{split}
\|f\|_{L^{q,\infty}(\Omega)} 
& = \sup\left\{ \lambda d_f(\lambda)^\frac{1}{q} \,;\, \lambda>0\right\}  \\
& =\inf \left\{ C>0 \,;\, \lambda d_f(\lambda)^{\frac1{q}} \le C
			\quad\text{for all }\lambda>0 \right\}.
\end{split}
\]
\end{rem}

% Definition 2.3
\begin{defi}\label{def:WLS}
Let $0<q,r \le \infty$ and $s \in \mathbb R$. 

\begin{enumerate}[\rm (i)]

\item
The weighted Lebesgue space $L^{q}_s(\Omega)$ is defined by
\[
L^{q}_s(\Omega)
:= 
\left\{
f \in L^0(\Omega) \, ;\, 
\|f\|_{L^{q}_s} <\infty
\right\}
\]
endowed with a quasi-norm
\[
	\|f\|_{L^{q}_s(\Omega)}:= 
	\begin{cases}
	\displaystyle 
	\left(
	\int_{\Omega} (|x|^s|f(x)|)^q\, dx
	\right)^\frac{1}{q}\quad &\text{if }q<\infty,\\
	\underset{x \in \Omega}{\mathrm{ess\,sup}}\, |x|^s|f(x)|
	& \text{if }q=\infty.
	\end{cases}
\]
The space $\mathcal L^{q}_s (\Omega)$ is defined as the completion of $L^{q}_s (\Omega) \cap L^\infty_0(\Omega)$ with respect to $\|\cdot\|_{L^{q}_s(\Omega)}$.

\item
The weighted Lorentz space $L^{q,r}_s(\Omega)$ is defined by
\[
L^{q,r}_s(\Omega)
:= 
\left\{
f \in L^0(\Omega) \, ;\, 
\|f\|_{L^{q,r}_s(\Omega)} <\infty
\right\}
\]
endowed with a quasi-norm
\[
	\|f\|_{L^{q,r}_s(\Omega)}:= \| |\cdot |^s f\|_{L^{q,r}(\Omega)}.
\]
The space $\mathcal L^{q,r}_s (\Omega)$ is defined as the completion of $L^{q,r}_s (\Omega) \cap L^\infty_0(\Omega)$ with respect to $\|\cdot\|_{L^{q,r}_s(\Omega)}$.
\end{enumerate}
\end{defi}

Only when $\Omega = \mathbb R^d$, we omit $\Omega$ and we write $\|\cdot\|_{L^{q,r}_s} = \|\cdot\|_{L^{q,r}_s(\mathbb R^d)}$ for simplicity.

% Remark 2.4
\begin{rem}
There are several ways to define weighted Lorentz spaces. 
For example, the definitions in \cite{DenapoliDrelichman,Elona_D,Kerman} are different from ours. 
\end{rem}

% Remark 2.5
\begin{rem}\label{rem:w-Lorentz}
We give several properties and remarks on $L^{q,r}_s(\Omega)$. 
Let $0<q,r \le \infty$ and $s\in \mathbb R$. 
\begin{enumerate}[\rm (a)]

\item
$L^{q,q}_s(\Omega)=L^{q}_s(\Omega)$ and $\mathcal L^{q,q}_s(\Omega)=\mathcal L^{q}_s(\Omega)$.

\item
$L^{\infty,r}_s(\Omega) = \{0\}$ for any $r<\infty$. 
Hence, in this paper, we always take $r=\infty$ when $q=\infty$ in $L^{q,r}_s(\Omega)$ even if it is not explicitly stated.

\item $L^{q,r}_s(\Omega)$ is a quasi-Banach space (see Remark \ref{rem:quasi-Banach} below),
and it is normable 
if $1<q<\infty$ and $1\le r \le \infty$.

\item $L^{q,r}_s(\Omega) \cap L_0^\infty(\Omega)$ is dense in $L^{q,r}_s(\Omega)$ if $q<\infty$ and $r < \infty$, which implies that $\mathcal L^{q,r}_s (\Omega) = L^{q,r}_s (\Omega)$ (see Lemma \ref{lem:A-dense} below).
On the other hand, $\mathcal L^{q,r}_s (\Omega) \subsetneq L^{q,r}_s (\Omega)$ if $q=\infty$ or $r=\infty$.

\item $L^{q,r}_s (\Omega)$
has the following embedding:
\[
L^{q,r_1}_s (\Omega) \hookrightarrow L^{q,r_2}_s (\Omega)
\]
for $0<r_1\le r_2 \le \infty$ (see e.g. \cite[Proposition 1.4.10]{Gr2008}).

\item Let $0\in\overline{\Omega}$. Then 
$L^{q,r}_s(\Omega) \subset L^1_{\mathrm{loc}}(\Omega)$ if and only if either of {\rm (f-1)}--{\rm (f-3)} holds:
\begin{enumerate}
\item[\rm (f-1)] $q>1$ and $\frac{s}{d} + \frac{1}{q} < 1$;
\item[\rm (f-2)] $q>1$, $\frac{s}{d} + \frac{1}{q} = 1$ and $r\le 1$;
\item[\rm (f-3)] $q=1$, $\frac{s}{d} + \frac{1}{q} \le 1$ and $r\le 1$.
\end{enumerate}
\item Let $a, b \in \mathbb R$. Then 
\[
|x|^{-a} \left|\log |x|\right|^{-b} \in L^{q,r}_s(\{|x| \le e^{-1}\})
\]
if and only if either of {\rm (g-1)}--{\rm (g-3)} holds:
\begin{enumerate}
\item[\rm (g-1)] $a < s + \frac{d}{q}$;
\item[\rm (g-2)] $a = s + \frac{d}{q}$, $b > \frac{1}{r}$ and $r<\infty$;
\item[\rm (g-3)] $a = s + \frac{d}{q}$, $b \ge 0$ and $r=\infty$.
\end{enumerate}

\item 
Let $a, b \in \mathbb R$. Then 
\[
|x|^{-a} \left|\log |x|\right|^{-b} \in L^{q,r}_s(\{|x| \ge e\})
\]
if and only if either of {\rm (h-1)}--{\rm (h-3)} holds:
\begin{enumerate}
\item[\rm (h-1)] $a > s + \frac{d}{q}$;
\item[\rm (h-2)] $a = s + \frac{d}{q}$, $b > \frac{1}{r}$ and $r<\infty$;
\item[\rm (h-3)] $a = s + \frac{d}{q}$, $b \ge 0$ and $r=\infty$.
\end{enumerate}
\end{enumerate}
For {\rm (g)} and {\rm (h)}, see e.g. \cite[Exercise 1.4.8]{Gr2008} and also the calculations of proof of \cite[Proposition 8.4]{Tay2020}.
\end{rem}

We have the 
H\"older and Young inequalities in Lorentz spaces.

% Lemma 2.6
\begin{lem}[Generalized H\"older's inequality]
	\label{lem:Holder}
Let $0 < q, q_1,q_2 < \infty$ and $0<r,r_1,r_2 \le \infty$.
Then the following assertions hold:
\begin{enumerate}[\rm (i)]
\item If 
\[
\frac{1}{q} = \frac{1}{q_1} + \frac{1}{q_2}\quad \text{and}\quad \frac{1}{r} \le \frac{1}{r_1} + \frac{1}{r_2},
\]
then 
there exists a constant $C>0$ such that 
\[
\|f g\|_{L^{q,r}}
\le C \|f\|_{L^{q_1,r_1}}\|g\|_{L^{q_2,r_2}}
\]
for any $f \in L^{q_1,r_1}(\mathbb R^d)$ and $g \in L^{q_2,r_2}(\mathbb R^d)$.

\item
There exists a constant $C>0$ such that 
\[
\|f g\|_{L^{q,r}}
\le C \|f\|_{L^{q,r}}\|g\|_{L^{\infty}}
\]
for any $f \in L^{q,r}(\mathbb R^d)$ and $g \in L^{\infty}(\mathbb R^d)$.
\end{enumerate}
\end{lem}

% Lemma 2.7
\begin{lem}[Generalized Young's inequality]\label{lem:Young}
Let $1< q,q_1,q_2<\infty$ and $0 < r,r_1,r_2\le \infty$. Then the following assertions hold:
\begin{enumerate}[\rm (i)]
\item
If 
\[
\frac{1}{q} = \frac{1}{q_1} + \frac{1}{q_2} -1 \quad \text{and}\quad \frac{1}{r} \le \frac{1}{r_1} + \frac{1}{r_2},
\]
then there exists a constant $C>0$ such that
\[
\|f * g\|_{L^{q,r}}
\le C \|f\|_{L^{q_1,r_1}}\|g\|_{L^{q_2,r_2}}
\]
for any $f \in L^{q_1,r_1}(\mathbb R^d)$ and $g \in L^{q_2,r_2}(\mathbb R^d)$.

\item
If 
\[
1= \frac{1}{q_1} + \frac{1}{q_2}  \quad \text{and}\quad 1 \le \frac{1}{r_1} + \frac{1}{r_2},
\]
then there exists a constant $C>0$ such that
\[
\|f * g\|_{L^{\infty}}
\le C \|f\|_{L^{q_1,r_1}}\|g\|_{L^{q_2,r_2}}
\]
for any $f \in L^{q_1,r_1}(\mathbb R^d)$ and $g \in L^{q_2,r_2}(\mathbb R^d)$.

\item
There exists a constant $C>0$ such that 
\[
\|f * g\|_{L^{q,r}}
\le C \|f\|_{L^{q,r}}\|g\|_{L^{1}}
\]
for any $f \in L^{q,r}(\mathbb R^d)$ and $g \in L^{1}(\mathbb R^d)$.
\end{enumerate}
\end{lem}

Lemmas \ref{lem:Holder} and \ref{lem:Young} 
are originally proved by O'Neil \cite{ONe1963} (see also Yap \cite{Yap1969} for Lorentz spaces with second exponents less than one).
Lemma~\ref{lem:Young} (iii) is known in the abstract setting (cf. Lemari\'e-Rieusset \cite[Chapter 4, Proposition 4.1]{Le2002}). It is also recently proved by Wang, Wei and Ye \cite[Lemma 2.2]{WWYarxiv}.\\

We also have the interpolation inequality in Lorentz spaces (see e.g. \cite[(2.4) on page 8]{WWYarxiv}).

\begin{lem}\label{lem:interpolation_Lorentz}
Let $0<q_1<q<q_2\le \infty$, $0<r\le \infty$ and $0<\theta<1$ satisfy
\[
\frac{1}{q} = \frac{\theta}{q_1}+\frac{1-\theta}{q_2}.
\]
Then
\[
\|f\|_{L^{q,r}} \le 
\left(
\frac{(q_2-q_1)q^2}{(q_2-q)(q-q_1) r}
\right)^{\frac{1}{r}}
\|f\|_{L^{q_1,\infty}}^\theta \|f\|_{L^{q_2,\infty}}^{1-\theta}
\]
for any $f \in L^{q_1,\infty}(\mathbb R^d) \cap L^{q_2,\infty}(\mathbb R^d)$.
\end{lem}

%%%%%%%%%%%%
%%% Section 3 %%%
%%%%%%%%%%%%
\section{Linear estimates}\label{sec:3}

In this section, we summarize linear estimates 
for the heat semigroup in the weighted Lorentz spaces.

%%% Subsection 3.1 %%%
\subsection{Smoothing and time decay estimates in weighted spaces}\label{sub:3.1}
Let 
$\{e^{t\Delta}\}_{t>0}$ be the heat semigroup whose element is defined by 
\[
e^{t\Delta} f := G_t \ast f,\quad f\in \mathcal S'(\mathbb R^d)
\]
with the Gaussian kernel 
\[
G_t(x) := (4\pi t)^{-\frac{d}{2}} e^{-\frac{|x|^2}{4t}},\quad t>0,\ x \in \mathbb R^d.
\]
In this subsection, we prove the following:

% Proposition 3.1
\begin{prop}
	\label{prop:linear-main}
Let $d \in \mathbb N$, $1\le q_1\le \infty$, $1<q_2\le \infty$, $0< r_1,r_2 \le \infty$ and $s_1,s_2\in \mathbb R$. 
Then there exists a constant $C>0$ such that 
\begin{equation}\label{linear-weight1}
\|e^{t\Delta}\|_{L^{q_1,r_1}_{s_1} \to L^{q_2,r_2}_{s_2}} = C t ^{-\frac{d}{2} (\frac{1}{q_1} - \frac{1}{q_2}) - \frac{s_1 - s_2}{2}}
\end{equation}
for any $t>0$ if and only if $q_1,q_2,r_1,r_2,s_1,s_2$ satisfy 
\begin{empheq}[left={\empheqlbrace}]{alignat=2}
  & 0 \le \frac{s_2}{d} + \frac{1}{q_2} \le \frac{s_1}{d} + \frac{1}{q_1} \le 1, \label{linear-condi1}\\
  & s_2 \le s_1,\label{linear-condi2}
\end{empheq}
and
\begin{empheq}[left={\empheqlbrace}]{alignat=2}
  & r_1 \le 1\quad \text{if }\frac{s_1}{d} + \frac{1}{q_1}=1 \text{ or }q_1=1,\label{linear-condi3}\\
  & r_2 = \infty\quad \text{if } \frac{s_2}{d} + \frac{1}{q_2}=0,\label{linear-condi4}\\
  & r_1 \le r_2 \quad \text{if }\frac{s_1}{d} + \frac{1}{q_1}=\frac{s_2}{d} + \frac{1}{q_2},\label{linear-condi5}\\
  & r_i=\infty \quad \text{if } q_i=\infty \quad (i=1,2).\label{linear-condi7}
\end{empheq}
\end{prop}

% Remark 3.2
\begin{rem}\label{rem:linear-main}
The estimate \eqref{linear-weight1} can be also obtained for $0<q_2\le 1$. 
More precisely, let $d \in \mathbb N$, $1\le q_1\le \infty$, $0<q_2\le 1$, $0< r_1,r_2 \le \infty$ and $s_1,s_2\in \mathbb R$, and assume 
\eqref{linear-condi1}--\eqref{linear-condi7} with the additional condition 
\begin{equation}\label{linear-condi-add}
r_2\ge1
\quad \text{if }\frac{s_2}{d} + \frac{1}{q_2} = \frac{s_1}{d} + \frac{1}{q_1} = 1.
\end{equation}
Then we have \eqref{linear-weight1} for any $t>0$. 
The additional condition \eqref{linear-condi-add} is required due to use of 
the embedding $L^1(\mathbb R^d) \hookrightarrow L^{1,r_2}(\mathbb R^d)$ for $r_2\ge1$ and 
Young's inequality $\|f * g\|_{L^1} \le \|f\|_{L^1}\|g\|_{L^1}$ in the case $\frac{s_2}{d} + \frac{1}{q_2} = \frac{s_1}{d} + \frac{1}{q_1} = 1$. 
On the other hand, we can also prove the necessity of \eqref{linear-condi1}--\eqref{linear-condi7}, but we do not know if \eqref{linear-condi-add} is necessary.  
The proof is similar to that of Proposition \ref{prop:linear-main}, and we omit it. 
In the proofs of the nonlinear estimates (Lemmas \ref{lem:sec4_nonlinear1}, \ref{lem:sec4_nonlinear2}, \ref{lem:nonlinear_Kato} and \ref{lem:perturbed}), we do not use the case $0<q_2 \le 1$. 
\end{rem}

% Remark 3.3
\begin{rem}
The estimates \eqref{linear-weight1} are known in some particular cases, for example, 
the case $s_2=0$ in Lebesgue spaces in \cite{BenTayWei2017}, the case $s_2 \ge 0$ in Lorentz spaces in \cite{TW}, and 
the case $q_1 \le q_2$ in Lebesgue spaces in 
\cite{OkaTsu2016,Tsu2011} (see also \cite{CIT2022}).
Similar estimates are proved in Herz spaces and weak Herz spaces in \cite{OkaTsu2016,Tsu2011}. 
\end{rem}

% Remark 3.4
\begin{rem}
Proposition \ref{prop:linear-main} gives a precision of \cite[Proposition 3.3]{Tay2020} in the endpoint case \eqref{linear-condi5} with $s_2=0$ and $s_1>0$. This implies that \cite[Remark  3.4, (2)]{Tay2020} does not hold. 
However, this does not change the results in \cite{Tay2020} as this endpoint case is not used in \cite{Tay2020}. 
\end{rem}

To reduce \eqref{linear-weight1} for $e^{t\Delta}$ into that for $e^\Delta$, we give the following lemma.

% Lemma 3.5
\begin{lem}\label{lem:scaling}
Let $d \in \mathbb N$, $1\le q_1,q_2 \le \infty$, $0<r_1,r_2 \le \infty$ and $s_1,s_2\in \mathbb R$. Then 
$e^{\Delta}$ is bounded from $L^{q_1,r_1}_{s_1}(\mathbb R^d)$ into $L^{q_2,r_2}_{s_2}(\mathbb R^d)$ if and only if 
$e^{t\Delta}$ is bounded from $L^{q_1,r_1}_{s_1}(\mathbb R^d)$ into $L^{q_2,r_2}_{s_2}(\mathbb R^d)$ with 
\begin{equation}\label{linear-weight1'}
\|e^{t\Delta}\|_{L^{q_1,r_1}_{s_1} \to L^{q_2,r_2}_{s_2}} =  t ^{-\frac{d}{2} (\frac{1}{q_1} - \frac{1}{q_2}) - \frac{s_1 - s_2}{2}} \|e^{\Delta}\|_{L^{q_1,r_1}_{s_1} \to L^{q_2,r_2}_{s_2}}
\end{equation}
for any $t>0$. 
\end{lem}

% Proof of Lemma 3.5
\begin{proof} 
It is enough to show \eqref{linear-weight1'} if $e^{\Delta}$ is bounded from $L^{q_1,r_1}_{s_1}(\mathbb R^d)$ into $L^{q_2,r_2}_{s_2}(\mathbb R^d)$, since the converse is trivial. 
The proof is based on the scaling argument. 
Let  $f \in L^{q_1,r_1}_{s_1}(\mathbb R^d)$. 
Since
\[
(e^{t\Delta} f)(x) =
\left(e^{\Delta} (f(t^{\frac12} \cdot))\right)(t^{-\frac12} x),
\]
\[
(e^{\Delta} f)(x) =
\left(e^{t\Delta} (f(t^{-\frac12} \cdot))\right)(t^{\frac12} x),
\]
for $t>0$ and $x\in \mathbb R^d$, we have
\[
\|e^{t\Delta} f\|_{L^{q_2,r_2}_{s_2}}
\le  t ^{-\frac{d}{2} (\frac{1}{q_1} - \frac{1}{q_2}) - \frac{s_1 - s_2}{2}} \|e^{\Delta}\|_{L^{q_1,r_1}_{s_1} \to L^{q_2,r_2}_{s_2}}
\|f\|_{L^{q_1,r_1}_{s_1}},
\]
\[
\|e^{\Delta} f\|_{L^{q_2,r_2}_{s_2}}
\le  t ^{\frac{d}{2} (\frac{1}{q_1} - \frac{1}{q_2}) + \frac{s_1 - s_2}{2}}
\|e^{t\Delta}\|_{L^{q_1,r_1}_{s_1} \to L^{q_2,r_2}_{s_2}}
\|f\|_{L^{q_1,r_1}_{s_1}}.
\]
Hence, \eqref{linear-weight1'} is proved. 
\end{proof}

% Proof of the necessity part
\begin{proof}[Proof of the necessity part of Proposition \ref{prop:linear-main}]

For the condition \eqref{linear-condi7}, 
see Remark \ref{rem:w-Lorentz} (b).

\smallskip

\noindent{\it Step 1: Conditions $\frac{s_1}{d} + \frac{1}{q_1}\le 1$ in \eqref{linear-condi1} and \eqref{linear-condi3}}. 
If either of these fails, then 
$L^{q_1, r_1}_{s_1} (\mathbb R^d)$ is not included in $L^1_{\mathrm{loc}}(\mathbb R^d)$ (see Remark~\ref{rem:w-Lorentz} (f)),
which implies that $e^{t\Delta} : L^{q_1, r_1}_{s_1} (\mathbb R^d) \to L^{q_2,r_2}_{s_2}(\mathbb R^d)$ is not well-defined.

\smallskip

\noindent{\it Step 2: Conditions $\frac{s_2}{d}+\frac{1}{q_2}\ge0$ in \eqref{linear-condi1} and \eqref{linear-condi4}}. 
Suppose either of these fails, i.e., 
\[
\frac{s_2}{d}+\frac{1}{q_2} < 0\quad \text{or}\quad \frac{s_2}{d}+\frac{1}{q_2}=0 \text{ and }
 r_2 < \infty.
\]
We consider the case $\frac{s_2}{d}+\frac{1}{q_2}=0$ and $r_2 < \infty$. By Lemma \ref{lem:A-cha}, if $f \in L^{q_2, r_2}_{s_2}(\mathbb R^d)$, then 
\[
\liminf_{|x|\to 0} |f(x)| \le \liminf_{|x| \to 0} |x|^{s_2+\frac{d}{q_2}} |\log |x||^{\frac{1}{r_2}} |f(x)| = 0.
\]
However, there exists an $f_0 \in L^{q_1,r_1}_{s_1}(\mathbb R^d)$ such that 
\[
\liminf_{|x|\to 0} |e^{\Delta} f_0(x)| \not = 0,
\] 
which implies $e^{\Delta} f_0 \not\in L^{q_2, r_2}_{s_2}(\mathbb R^d)$. 
Hence, it is impossible to obtain \eqref{linear-weight1}. 
The case $\frac{s_2}{d} + \frac{1}{q_2} < 1$ is similarly proved.

\smallskip

\noindent{\it Step 3: Condition $\frac{s_1}{d}+\frac{1}{q_1}\le \frac{s_2}{d}+\frac{1}{q_2}$ in \eqref{linear-condi1}}. 
Suppose that $\frac{s_1}{d}+\frac{1}{q_1}\le \frac{s_2}{d}+\frac{1}{q_2}$ does not hold. 
Let $f \in C^\infty_0(\mathbb R^d)$ with $f\not =0$. Then 
we have 
\[
\|e^{t\Delta} f\|_{L^{q_2,r_2}_{s_2}} \le C 
t ^{\delta}
	\|f \|_{L^{q_1, r_1}_{s_1}} ,\quad t>0,
\]
where
\[
\delta := 
-\frac{d}{2} \left(\frac{1}{q_1} - \frac{1}{q_2}\right) - \frac{s_1 - s_2}{2}>0.
\]
Hence, $e^{t\Delta} f \to 0$ in $\mathcal S'(\mathbb R^d)$ as $t\to 0$. 
Combining this with the continuity $e^{t\Delta} f \to f$ in $\mathcal S'(\mathbb R^d)$ as $t\to 0$, 
we have $f=0$ by uniqueness of the limit. However, this is a contradiction to $f\not =0$. Thus, $\frac{s_1}{d}+\frac{1}{q_1}\le \frac{s_2}{d}+\frac{1}{q_2}$ is necessary.

\smallskip

\noindent{\it Step 4: Condition \eqref{linear-condi2}}. 
The proof is based on the translation argument as in \cite{DenapoliDrelichman,TW}.
In fact, take a non-negative function $f \in C^\infty_0(\mathbb R^d)$ with $\supp f \subset \{x = (x_1, x') \in \mathbb R \times \mathbb R^{d-1} \, ;\, x_1 \ge 0\}$, and let $x_0 =(1,0,\ldots,0) \in \mathbb R^d$ and $\tau>0$. 
Since $(e^{\Delta} f(\cdot - \tau x_0))(x) = (e^{\Delta} f)(x-\tau x_0)$, it follows from \eqref{linear-weight1} that 
\[
\left \||\cdot|^{s_2}(e^{\Delta} f)(\cdot-\tau x_0)\right\|_{L^{q_2,r_2}}
\le 
C \left \||\cdot|^{s_1}f(\cdot-\tau x_0)\right\|_{L^{q_1,r_1}}.
\]
By making the changes of variables, we have 
\begin{equation}\label{eq.inequality}
\tau^{-(s_1-s_2)}\left \|\left|{\cdot\over \tau}+ x_0\right|^{s_2}e^{\Delta} f\right\|_{L^{q_2,r_2}}\leq C \left \|\left|{\cdot\over \tau}+ x_0\right|^{s_1}f\right\|_{L^{q_1,r_1}}.
\end{equation}
The weight $|{\cdot \over \tau}+ x_0|^{s_2}$ has the uniform lower bounds with respect to sufficient large $\tau$:
\[
\left|{x\over \tau}+ x_0\right|^{s_2}
\ge 
\begin{cases}
\displaystyle \left(
1- {|x| \over \tau}
\right)^{s_2}
\ge 2^{-s_2}
\quad
& \text{for }|x|\le 1,\ \tau\ge 2 
\quad \text{if }s_2\ge0,\\
\displaystyle  \left({|x|\over \tau}+ 1\right)^{s_2} \ge (|x|+1)^{s_2}
&\text{for }\tau\ge 1 
\quad \text{if }s_2\le0.
\end{cases}
\]
Hence, once 
\begin{equation}\label{eq.limsup}
\limsup_{\tau \to \infty} \left \|\left|{\cdot\over \tau}+ x_0\right|^{s_1}f\right\|_{L^{q_1,r_1}} < \infty
\end{equation}
is obtained, we deduce $s_2 \le s_1$ from \eqref{eq.inequality}, \eqref{eq.limsup} and positivity of $e^{\Delta}f$. 
Therefore, it is enough to show \eqref{eq.limsup}. 
In the case $s_1 \ge 0$, we have the uniform upper bound
\[
\left|{x\over \tau}+ x_0\right|^{s_1} \le 
\left({|x|\over \tau}+ 1\right)^{s_1} \le (|x|+1)^{s_1}
\quad \text{for }\tau\ge 1, 
\]
which implies \eqref{eq.limsup}. 
In the other case $s_1 < 0$, 
the weight
\[
\left|{x\over \tau}+ x_0 \right|^{s_1} = \left[
\left(\frac{x_1}{\tau} + 1\right)^2 + \frac{|x'|^2}{\tau^2}
\right]^\frac{s_1}{2}
\]
has a singularity only at $x=x^*(\tau)=(-\tau,0,\ldots,0)$, and 
is increasing with respect to $\tau$ for each $x \in \{x_1\ge0\}$. 
Here, 
we note that $|{\cdot\over \tau}+ x_0|^{s_1}f \in L^{q_1,r_1}$ for any $\tau>0$, 
since the singular points $x^*(\tau)$ are not included in $\supp f$ for any $\tau>0$. 
Hence, 
\[
\left|{\cdot\over \tau}+ x_0\right|^{s_1}f
\nearrow f\quad 
\text{a.e.}\, x \in \{x_1\ge 0\} \quad \text{as }\tau \to\infty,
\]
and we can use Lemma \ref{lem:monotone}
to obtain 
\[
\lim_{\tau \to \infty} \left \|\left|{\cdot\over \tau}+ x_0\right|^{s_1}f\right\|_{L^{q_1,r_1}} 
 =\|f\|_{L^{q_1,r_1}}< \infty.
\]
This implies \eqref{eq.limsup}. 
Thus, the necessity of $s_2 \le s_1$ is proved.

\smallskip

\noindent{\it Final step: Condition \eqref{linear-condi5}}. 
Let 
\[
f (x) = (1+|x|)^{-\frac{d}{q_1}-s_1} ( \log (e+|x|))^{-b}
\] 
where 
$b > \frac{1}{r_1}$ if $r_1 < \infty$ and $b=0$ if $r_1=\infty$.
Then $f\in L^{q_1,r_1}_{s_1}(\mathbb R^d)$ (see Remark \eqref{rem:w-Lorentz} (h)). 
Since 
$f$ is a positive, radially symmetric and decreasing function, 
we have 
\begin{equation}\label{3.2:lower}
e^{\Delta}f(x)
\ge
\int_{|y|\le 1}
G_1(y) f(x-y)\, dy
\ge Cf(x)
\end{equation}
for $|x| \ge1$ sufficiently large. 
Now, if 
\eqref{linear-condi5} fails, 
i.e., 
$r_1 > r_2$ and $\frac{s_1}{d} + \frac{1}{q_1}=\frac{s_2}{d} + \frac{1}{q_2}$, 
then 
$f \not \in L^{q_2,r_2}_{s_2}(\mathbb R^d)$ since $b$ can be taken as $\frac{1}{r_1} < b < \frac{1}{r_2}$ if $r_1<\infty$ and $0= b < \frac{1}{r_2}$ if $r_1=\infty$.
Hence, by \eqref{3.2:lower}, we also have $e^{\Delta}f \not \in L^{q_2,r_2}_{s_2}(\mathbb R^d)$, which means $\|e^{\Delta}\|_{L^{q_1,r_1}_{s_1} \to L^{q_2,r_2}_{s_2}} = \infty$. 
Thus, the necessity of \eqref{linear-condi5} is shown by contraposition.
The proof of the necessity part is finished.
\end{proof}

% Proof of the sufficiency part
\begin{proof}[Proof of the sufficiency part of Proposition \ref{prop:linear-main}]
By Lemma \ref{lem:scaling}, it is enough to prove 
\eqref{linear-weight1} with $t=1$:
\begin{equation}\label{eq.goal:t=1}
\|e^{\Delta} f\|_{L^{q_2,r_2}_{s_2}} \le C\|f\|_{L^{q_1,r_1}_{s_1}}.
\end{equation}

We start the proof with the case 
$1 < q_1, q_2<\infty$. 
We first prove \eqref{eq.goal:t=1} 
with the non-endpoint case: 
\begin{equation}\label{para:non-end}
0 < \frac{s_2}{d} + \frac{1}{q_2} < \frac{s_1}{d} + \frac{1}{q_1} < 1
\quad \text{and}\quad  
s_2\le s_1.
\end{equation}
From Lemma \ref{lem:scaling} and the embedding $L^{q_1,r_1}(\mathbb R^d) \hookrightarrow L^{q_1,\infty}(\mathbb R^d)$ for any $0< r_1 \le \infty$, 
it is sufficient to show that 
$e^{\Delta}$ is bounded from $L^{q_1,\infty}_{s_1}(\mathbb R^d)$ into $L^{q_2,r_2}_{s_2}(\mathbb R^d)$. 
We divide the proof into three cases:
\[
s_2 \ge 0, \quad s_2 < 0 \le s_1 \quad \text{and}\quad s_1 < 0. 
\]

In the case $s_2\ge0$, we use the inequality $|x|^{s_2} \le C( |x-y|^{s_2} + |y|^{s_2})$ to obtain 
\[
\begin{split}
||x|^{s_2}e^{\Delta} f (x)|
& =
|x|^{s_2} \left|(G_1 * f) (x)\right|\\
& \le C \left\{
(G_1 * (|\cdot|^{s_2} |f|))(x)
+
((|\cdot|^{s_2} G_1) * |f|))(x)
\right\}.
\end{split}
\]
Then we use Lemma \ref{lem:Holder} (i) and 
Lemma \ref{lem:Young} (i) to estimate
\begin{equation}\label{eq2}
\begin{split}
\|G_1 * (|\cdot|^{s_2} |f|)\|_{L^{q_2,r_2}}
& \le C\|G_1\|_{L^{p_1,r_2}} \||\cdot|^{s_2} |f|\|_{L^{p_2,\infty}}\\
& \le C\|G_1\|_{L^{p_1,r_2}} \||\cdot|^{s_2-s_1}\|_{L^{\frac{d}{s_1-s_2}, \infty}}
\||\cdot|^{s_1}f\|_{L^{q_1,\infty}}\\
& \le C\|f\|_{L^{q_1,\infty}_{s_1}},
\end{split}
\end{equation}
where $p_1$ and $p_2$ satisfy $1<p_1<(\frac{s_2}{d} + \frac{1}{q_2})^{-1}$, $(1-\frac{s_2}{d})^{-1}<p_2 <q_2$, $\frac{1}{q_2} = \frac{1}{p_1} + \frac{1}{p_2} -1$ and $\frac{1}{p_2} = \frac{s_1-s_2}{d} + \frac{1}{q_1}$, and 
\begin{equation}\label{eq3}
\begin{split}
\||(|\cdot|^{s_2} G_1) * |f|\|_{L^{q_2,r_2}}
& \le 
C \||\cdot|^{s_2} G_1\|_{L^{p_3,r_2}} \|f\|_{L^{p_4,\infty}}\\
& \le C \||\cdot|^{s_2} G_1\|_{L^{p_3,r_2}}\||\cdot|^{-s_1}\|_{L^{\frac{d}{s_1},\infty}} 
\||\cdot|^{s_1}f\|_{L^{q_1,\infty}}\\
& \le C\|f\|_{L^{q_1,\infty}_{s_1}},
\end{split}
\end{equation}
where $p_3$ and $p_4$ satisfy $(1-\frac{s_2}{d})^{-1}< p_3 < q_2$, $1<p_4<(\frac{s_2}{d} + \frac{1}{q_2})^{-1}$, $\frac{1}{q_2} = \frac{1}{p_3} + \frac{1}{p_4} -1$ and $\frac{1}{p_4} = \frac{s_1}{d} + \frac{1}{q_1}$. 
Here, we note that such $p_1$, $p_2$, $p_3$ and $p_4$ exist if \eqref{para:non-end} and $s_2\ge0$ hold.
Hence, \eqref{eq.goal:t=1} is proved in this case.

In the case $s_2 < 0 \le s_1$,
we use Lemma \ref{lem:Holder} (i) to obtain 
\begin{equation}\label{eq4}
\| e^{\Delta} f \|_{L^{q_2,r_2}_{s_2}}
\le C
\| |\cdot|^{s_2} \|_{L^{-\frac{d}{s_2},\infty}}
\|e^{\Delta} f \|_{L^{p_5,r_2}},
\end{equation}
where $p_5$ satisfies $(\frac{s_1}{d} + \frac{1}{q_1})^{-1}<p_5<\infty$ and $\frac{1}{q_2} = -\frac{s_2}{d} + \frac{1}{p_5}$, 
and such a $p_5$ exists under the conditions \eqref{para:non-end} and $s_2 < 0 \le s_1$.
Now, noting $p_5$ satisfies $0 <  \frac{1}{p_5} < \frac{s_1}{d} + \frac{1}{q_1} < 1$ and $0\le s_1$, 
we can apply the estimate shown in the previous case with $s_2=0$ to obtain
\[
\|e^{\Delta} f \|_{L^{p_5,r_2}}
\le C
\|f\|_{L^{q_1,\infty}_{s_1}}.
\]
Thus, the case $s_2 < 0 \le s_1$ is also proved.

In the case $s_1< 0$, setting $g := |x|^{s_1}|f|$, and 
using the inequality $|y|^{-s_1} \le C( |x-y|^{-s_1} + |x|^{-s_1})$, we have  
\begin{equation}\label{eq6}
\begin{split}
\left| |x|^{s_2}e^{\Delta} f (x)\right|
& \le 
|x|^{s_2} \int_{\mathbb R^d}
G_1 (x-y) |y|^{-s_1}g(y)\, dy\\
& \le C\left( |x|^{s_2-s_1}
e^{\Delta} g(x)
+
|x|^{s_2}\left((|\cdot|^{-s_1} G_1) * g\right)(x)\right).
\end{split}
\end{equation}
Then we use Lemma \ref{lem:Holder} (i) and Lemma \ref{lem:Young} (i) to estimate
\begin{equation}\label{eq7}
\begin{split}
\| |\cdot|^{s_2-s_1}
e^{\Delta} g
\|_{L^{q_2,r_2}}
& \le C \| |\cdot|^{s_2-s_1} \|_{L^{\frac{d}{s_1-s_2},\infty}}
\| e^{\Delta} g\|_{L^{p_6,r_2}}\\
& \le C \| |\cdot|^{s_2-s_1} \|_{L^{\frac{d}{s_1-s_2},\infty}}
\|G_1\|_{L^{p_7,r_2}} \|g\|_{L^{{q_1, \infty}}}\\
& \le C\|f\|_{L^{q_1,\infty}_{s_1}},
\end{split}
\end{equation}
where $p_6$ and $p_7$ satisfy $q_1<p_6 < -\frac{d}{s_1}$, $1 < p_7 < (\frac{s_2}{d} + \frac{1}{q_2})^{-1}$, $\frac{1}{q_2} = \frac{s_1-s_2}{d} + \frac{1}{p_6}$ and $\frac{1}{p_6} = \frac{1}{p_7} + \frac{1}{q_1}-1$,
and 
\begin{equation}\label{eq8}
\begin{split}
\left\| |\cdot|^{s_2}\left((|\cdot|^{-s_1} G_1) * g\right)
\right\|_{L^{q_2,r_2}}
& \le C \| |\cdot|^{s_2}\|_{L^{-\frac{d}{s_2},\infty}}
\|(|\cdot|^{-s_1} G_1) * g\|_{L^{p_8, r_2}}\\
& \le C \| |\cdot|^{s_2}\|_{L^{-\frac{d}{s_2},\infty}}
\||\cdot|^{-s_1} G_1\|_{L^{p_9, r_2}}
\|g\|_{L^{q_1,\infty}}\\
&\le C\|f\|_{L^{q_1,\infty}_{s_1}},
\end{split}
\end{equation}
where $p_8$ and $p_9$ satisfy $(\frac{s_1}{d} + \frac{1}{q_1})^{-1} < p_8< \infty$, $(\frac{s_1}{d}+1)^{-1} < p_9 < (1-\frac{1}{q_1})^{-1}$, 
$\frac{1}{q_2} = - \frac{s_2}{d} + \frac{1}{p_8}$ and $\frac{1}{p_8} = \frac{1}{p_9} + \frac{1}{q_1}-1$. 
Here, we note that
such $p_6$, $p_7$, $p_8$ and $p_9$ exist if \eqref{para:non-end} and $s_1<0$ hold.
Thus, the case $s_1<0$ is also proved.

Next, we consider the endpoint cases \eqref{linear-condi3}, \eqref{linear-condi4} or \eqref{linear-condi5} with $1<q_1,q_2<\infty$.
Here, we give only sketch of proofs of single endpoint cases. If two or more endpoints overlap, simply combine them.

As to the case \eqref{linear-condi3}, i.e., $\frac{s_1}{d} + \frac{1}{q_1}=1$ and $r_1 \le 1$, 
we note that $s_1\ge0$, and the proof is almost the same as the non-endpoint case \eqref{para:non-end} with $s_1\ge0$. 
In fact, we can take $p_1=(\frac{s_2}{d} + \frac{1}{q_2})^{-1}$, $p_2=(1-\frac{s_2}{d})^{-1}$, $p_3=q_2$ and $p_4=1$, 
and use Lemma \ref{lem:Young} (iii) (instead of Lemma \ref{lem:Young} (i))
in \eqref{eq3}, where $\|f\|_{L^{p_4,\infty}}$ is replaced by $\|f\|_{L^1}$ and the restriction $r_1 \le 1$ appears.

As to the case \eqref{linear-condi4}, i.e., $\frac{s_2}{d} + \frac{1}{q_2}=0$ and $r_2 = \infty$, 
we note that $s_2< 0$, and the proof is similar to the non-endpoint case \eqref{para:non-end} with $s_2 < 0 \le s_1$ or $s_2 \le s_1 <0$. 
For $s_2 < 0 \le s_1$, we use 
Lemma \ref{lem:Holder} (ii) to obtain 
 \[
\| e^{\Delta} f \|_{L^{q_2,\infty}_{s_2}}
\le C
\| |\cdot|^{s_2} \|_{L^{-\frac{d}{s_2},\infty}}
\|e^{\Delta} f \|_{L^{\infty}}
\le C
\|e^{\Delta} f \|_{L^{\infty}}
\]
(this corresponds to taking $p_5 = \infty$ in \eqref{eq4}). 
The estimate $\|e^{\Delta} f \|_{L^{\infty}}
\le C
\|f\|_{L^{q_1,r_1}_{s_1}}$ will be given later (see the proof of the case $q_2=\infty$ below). 
For $s_2 \le s_1 <0$, we also have 
\[
\begin{split}
\| e^{\Delta} f \|_{L^{q_2,\infty}_{s_2}}
& \le 
\| |\cdot|^{s_2-s_1}
e^{\Delta} g
\|_{L^{q_2,\infty}}
+
\left\| |\cdot|^{s_2}\left((|\cdot|^{-s_1} G_1) * g\right)
\right\|_{L^{q_2,\infty}}\\
& \le C \Big(\| |\cdot|^{s_2-s_1} \|_{L^{\frac{d}{s_1-s_2},\infty}}
\| e^{\Delta} g\|_{L^{-\frac{d}{s_1},\infty}}\\
& \qquad \qquad \qquad \qquad \qquad  
+
 \| |\cdot|^{s_2}\|_{L^{-\frac{d}{s_2},\infty}}
\|(|\cdot|^{-s_1} G_1) * g\|_{L^{\infty}}
\Big)
\\
& \le C 
\Big(\| |\cdot|^{s_2-s_1} \|_{L^{\frac{d}{s_1-s_2},\infty}}
\|G_1\|_{L^{p_7,\infty}} \|g\|_{L^{{q_1, \infty}}}\\
& \qquad \qquad \qquad \qquad \qquad +
\| |\cdot|^{s_2}\|_{L^{-\frac{d}{s_2},\infty}}
\||\cdot|^{-s_1} G_1\|_{L^{p_9,1}}
\|g\|_{L^{q_1,\infty}}
\Big)\\
& \le C\|f\|_{L^{q_1,\infty}_{s_1}},
\end{split}
\]
where we take $p_6 = -\frac{d}{s_1}$, $p_7 = [1- (\frac{s_1}{d} + \frac{1}{q_1})]^{-1}$, $p_8 = \infty$ and $p_9 = (1-\frac{1}{q_1})^{-1}$.

As to the case \eqref{linear-condi5}, i.e., $\frac{s_1}{d} + \frac{1}{q_1}=\frac{s_2}{d} + \frac{1}{q_2}$ and $r_1 \le r_2$, 
we can use Lemma \ref{lem:Young} (iii) to make a similar argument to the non-endpoint case. 
In fact, when $s_2\ge0$, this case corresponds to taking $p_1=1$, $p_2=q_2$, $p_3=(1-\frac{s_2}{d})^{-1}$ and $p_4 = (\frac{s_2}{d} + \frac{1}{q_2})^{-1}$ in \eqref{eq2} and \eqref{eq3}. 
In particular, 
in \eqref{eq2}, Lemma \ref{lem:Young} (iii) is used and the restriction $r_1 \le r_2$ is required: 
\[
\begin{split}
\|G_1 * (|\cdot|^{s_2} |f|)\|_{L^{q_2,r_2}}
& \le C\|G_1\|_{L^1} \||\cdot|^{s_2} |f|\|_{L^{q_2,r_2}}\\
& \le C\|G_1\|_{L^1} \||\cdot|^{s_2-s_1}\|_{L^{\frac{d}{s_1-s_2}, \infty}}
\||\cdot|^{s_1}f\|_{L^{q_1,r_2}}\\
& \le C\|f\|_{L^{q_1,r_2}_{s_1}}\le C\|f\|_{L^{q_1,r_1}_{s_1}}.
\end{split}
\]
The case $s_2<0$ is similar, and we may omit it.

In the rest of the proof, we consider the cases $q_1=1$, $q_1=\infty$ or $q_2=\infty$.  
The case $q_1=1$ and $q_2 = \infty$ is just $L^1$-$L^\infty$ estimate. 
The case $q_1=q_2=\infty$ has been already proved (see, e.g., \cite[Lemma 2.1]{CIT2022}).

The case $1<q_1<\infty$ and $q_2=\infty$ is the estimate
\eqref{eq.goal:t=1} with 
\[
0\le s_2\le s_1,\quad 
\frac{s_1}{d} + \frac{1}{q_1} \le 1
\quad \text{and}\quad 
r_1\le 1\text{ if }\frac{s_1}{d} + \frac{1}{q_1} = 1.
\]
Since $s_2\ge0$, this case is proved in a similar way to \eqref{eq2} and \eqref{eq3}. 
In fact, we deduce from Lemma \ref{lem:Young} (ii) and Lemma \ref{lem:Holder} (i) that 
\begin{equation}\label{eq9}
\begin{split}
\|G_1 * (|\cdot|^{s_2} |f|)\|_{L^{\infty}}
& \le C\|G_1\|_{L^{p_{10},1}} \||\cdot|^{s_2} |f|\|_{L^{p_{11},\infty}}\\
& \le C\||\cdot|^{s_2-s_1}\|_{L^{\frac{d}{s_1-s_2}, \infty}}
\||\cdot|^{s_1}f\|_{L^{q_1,\infty}}\\
& \le C\|f\|_{L^{q_1,\infty}_{s_1}},
\end{split}
\end{equation}
where $p_{10}$ and $p_{11}$ satisfy $1\le p_{10}<\frac{d}{s_2}$, $\frac{d}{d-s_2}<p_{11}\le\infty$, $1 = \frac{1}{p_{10}} + \frac{1}{p_{11}}$ and $\frac{1}{p_{11}} = \frac{s_1-s_2}{d} + \frac{1}{q_1}$, 
and 
\[
\begin{split}
\||(|\cdot|^{s_2} G_1) * |f|\|_{L^{\infty}}
& \le 
C \||\cdot|^{s_2} G_1\|_{L^{p_{12},1}} \|f\|_{L^{p_{13},\infty}}\\
& \le C\||\cdot|^{-s_1}\|_{L^{\frac{d}{s_1},\infty}} 
\||\cdot|^{s_1}f\|_{L^{q_1,\infty}}\\
& \le C\|f\|_{L^{q_1,\infty}_{s_1}},
\end{split}
\]
where $p_{12}$ and $p_{13}$ satisfy $\frac{d}{d-s_2} \le p_{12} < \infty$, $1<p_{13}\le \frac{d}{s_2}$,  
$1 = \frac{1}{p_{12}} + \frac{1}{p_{13}}$ and $\frac{1}{p_{12}} = \frac{s_1}{d} + \frac{1}{q_1}$. 
Here, we note that such $p_{10}$, $p_{11}$, $p_{12}$ and $p_{13}$ exist if 
$0\le s_2\le s_1$ and $\frac{s_1}{d} + \frac{1}{q_1}<1$. 
For the case $\frac{s_1}{d} + \frac{1}{q_1}=1$, 
the first term can be estimated in the same way as \eqref{eq2} (where we take $p_8=\frac{d}{s_2}$ and $p_9=\frac{d}{d-s_2}$). 
For the second term, we take $p_{10}=\infty$ and $p_{11}=1$ and we use Lemma \ref{lem:Young} (ii) to obtain
\[
\begin{split}
\||(|\cdot|^{s_2} G_1) * |f|\|_{L^{\infty}}
& \le 
C \||\cdot|^{s_2} G_1\|_{L^{\infty}} \|f\|_{L^{1}}\\
& \le C\||\cdot|^{-s_1}\|_{L^{\frac{d}{s_1},\infty}} 
\||\cdot|^{s_1}f\|_{L^{q_1,1}}\\
& \le C\|f\|_{L^{q_1,1}_{s_1}}.
\end{split}
\]
Thus, the estimate \eqref{eq.goal:t=1} is proved in the case $q_2=\infty$.

The case $q_1=1$ and $1<q_2 <\infty$ is 
the estimate
\eqref{eq.goal:t=1} with 
\[
s_2\le s_1\le 0,\quad 
0\le \frac{s_2}{d} + \frac{1}{q_2} < \frac{s_1}{d} + 1,
\quad 
r_1\le 1 \quad \text{and}\quad 
 r_2 = \infty\text{ if } \frac{s_2}{d} + \frac{1}{q_2}=0.
\]
The proof is similar to \eqref{eq7} and \eqref{eq8}. 
Let $\frac{s_2}{d} + \frac{1}{q_2}>0$. 
As to the first term, it follows
from Lemma \ref{lem:Holder} (i) and Lemma \ref{lem:Young} (ii) that 
\begin{equation}\label{eq12}
\begin{split}
\| |\cdot|^{s_2-s_1}
e^{\Delta} g
\|_{L^{q_2,r_2}}
& \le C \| |\cdot|^{s_2-s_1} \|_{L^{\frac{d}{s_1-s_2},\infty}}
\| e^{\Delta} g\|_{L^{p_{14},r_2}}\\
& \le C \| |\cdot|^{s_2-s_1} \|_{L^{\frac{d}{s_1-s_2},\infty}}
\|G_1\|_{L^{p_{14},r_2}} \|g\|_{L^1}\\
& \le C\|f\|_{L^1_{s_1}},
\end{split}
\end{equation}
where $p_{14}$ satisfies $1<p_{14} < -\frac{d}{s_1}$ and $\frac{1}{q_2} = \frac{s_1-s_2}{d}+ \frac{1}{p_{14}}$. 
The second term can be estimated as 
\[
\begin{split}
\left\| |\cdot|^{s_2}\left((|\cdot|^{-s_1} G_1) * g\right)
\right\|_{L^{q_2,r_2}}
& \le C \| |\cdot|^{s_2}\|_{L^{-\frac{d}{s_2},\infty}}
\|(|\cdot|^{-s_1} G_1) * g\|_{L^{p_{15}, r_2}}\\
& \le C \| |\cdot|^{s_2}\|_{L^{-\frac{d}{s_2},\infty}}
\||\cdot|^{-s_1} G_1\|_{L^{p_{15}, r_2}}
\|g\|_{L^{1}}\\
&\le C\|f\|_{L^1_{s_1}},
\end{split}
\]
where $p_{15}$ satisfies $(\frac{s_1}{d} +1)^{-1} < p_{15}< \infty$ and $\frac{1}{q_2} = - \frac{s_2}{d} + \frac{1}{p_{15}}$. 
Here, we note that
such $p_{14}$ and $p_{15}$ exist if 
$s_2\le s_1\le 0$ and 
$0 < \frac{s_2}{d} + \frac{1}{q_2} < \frac{s_1}{d} + 1$. 
For the case $\frac{s_2}{d} + \frac{1}{q_2}=0$, 
the first term can be estimated in the same way as \eqref{eq12} (where we take $p_{14}=-\frac{d}{s_1}$). 
For the second term, 
we have only to take $p_{15}=\infty$ and $r_2=\infty$ and use Young's inequality $\|f * g\|_{L^\infty}\le \|f\|_{L^1}\|g\|_{L^\infty}$. 
Thus, the estimate \eqref{eq.goal:t=1} is proved in the case $q_1=1$ and $1<q_2 <\infty$.
The proof of Proposition \ref{prop:linear-main} is finished.
\end{proof}

%%% Subsection 3.2 %%%
\subsection{Weighted Meyer inequality}\label{sub:3.2}

In this subsection, we shall prove the following proposition, 
which is a key tool to study unconditional uniqueness and uniqueness criterion in the scale-critical case and 
the construction of a singular solution in the double critical case.

% Proposition 3.6
\begin{prop}\label{l:wMeyer.inq}
Let $T\in (0,\infty]$, and let $d\ge 3$, $1 \le q_1 \le\infty$, $1<q_2<\infty$, $0<r_1 \le \infty$ and $s_1,s_2 \in \mathbb R$ satisfy 
\begin{empheq}[left={\empheqlbrace}]{alignat=2}
  & 0 < \frac{s_2}{d} + \frac{1}{q_2} < \frac{s_1}{d} + \frac{1}{q_1} \le 1, \label{linear-condi1'}\\
  & s_2 \le s_1,\label{linear-condi2'}\\
  & \frac{d}{2} \left(\frac{1}{q_1} - \frac1{q_2} \right)+ \frac{s_1-s_2}{2} = 1,\label{linear-condi3'}
\end{empheq}
and
\begin{empheq}[left={\empheqlbrace}]{alignat=2}
  & r_1 \le 1\quad \text{if }\frac{s_1}{d} + \frac{1}{q_1}=1 \text{ or }q_1=1,\label{linear-condi4'}\\
  & r_1=\infty \quad \text{if } q_1=\infty.\label{linear-condi6'}
\end{empheq}
Then there exists a constant $C>0$ such that 
\begin{equation}\label{eq.wMi}
		\left\|\int_0^{t} e^{(t-\tau)\Delta} f(\tau) \,d\tau \right\|_{L^{q_2,\infty}_{s_2}} 
		\le C \sup_{0 <\tau <t} \|f(\tau)\|_{L^{q_1,r_1}_{s_1}}
\end{equation}
	for any $t\in (0,T)$ and $f \in L^\infty(0,T; L^{q_1,r_1}_{s_1}(\mathbb R^d))$.
\end{prop}

The case $s_1=s_2=0$ is known as Meyer's inequality and is proved by Meyer \cite{Mey1997} (see also \cite{Ter2002}). 

% Proof of Proposition 3.6
\begin{proof}
We shall prove only the case $q_1>1$ and $\frac{s_1}{d} + \frac{1}{q_1}<1$, since the proofs of the other cases are similar. 
By the argument in \cite{Mey1997}, it suffices to prove 
that 
\begin{equation}\label{Meyer-goal}
\|g\|_{L^{q_2,\infty}_{s_2}} \le C,
\end{equation}
where we define 
\[
	g(x) := \int_0^{\infty} e^{t\Delta} f(t,x) \,dt
\]
and we may assume that 
\[
	\sup_{t\ge0} \|f(t,\cdot) \|_{L^{q_1,r_1}_{s_1}} \le 1
\]
without loss of generality. 
Let $\lambda\in (0,\infty)$ be arbitrarily fixed. 
For $\tau \in (0,\infty),$ which is to be determined later, we divide $g$ into two parts:
\[
	g(x) = \int_{0}^{\tau} e^{t\Delta} f(t,x) \,dt 
			+\int_{\tau}^{\infty} e^{t\Delta} f(t,x) \,dt 
	=: h(x) + \ell(x). 
\]
Let 
 $p_0$ and $p_1$ be such that 
\[
1 < p_1 < q_2 < p_0 \le \infty
\quad \text{and}\quad 
0 \le \frac{s_2}{d} + \frac{1}{p_i} \le \frac{s_1}{d} + \frac{1}{q_1}\quad \text{for }i=0,1. %< 1
\]
Then, by Proposition \ref{prop:linear-main}, we have 
\[
\begin{aligned}
	\| \ell \|_{L^{p_0,\infty}_{s_2}}
	&\le \int_{\tau}^{\infty} \| e^{t\Delta} f(t) \|_{L^{p_0,\infty}_{s_2}} \,dt \\
	& \le C \int_{\tau}^{\infty} t^{-\frac{d}{2}(\frac{1}{q_1} - \frac1{p_0}) - \frac{s_1-s_2}{2} } 
		\|f(t) \|_{L^{q_1,\infty}_{s_1}} \,dt\\
	&\le  C \tau^{ - \frac{d}{2} (\frac1{q_2}-\frac1{p_0})} 
\end{aligned}
\]
and 
\[
\begin{aligned}
	\|h\|_{L^{p_1,\infty}_{s_2}}
	&  \le  \int_{0}^{\tau} \left\| e^{t\Delta} f(t) \right\|_{L^{p_1,\infty}_{s_2}} \,dt \\
	& \le \int_{0}^{\tau} t^{-\frac{d}{2} (\frac1{q_1} - \frac1{p_1}) - \frac{s_1 -s_2}{2}}  
	\left\| f(t) \right\|_{L^{q_1,\infty}_{s_1}} \,dt \\
	& \le C \tau^{ \frac{d}{2} (\frac1{p_1}-\frac1{q_2})}.
\end{aligned}
\]
Now, the definition of the 
Lorentz norms yields 
\[
	d_{|\cdot|^{s_2} \ell}\left(\frac{\lambda}2\right)
	\le \left(\frac{ \|\ell\|_{L^{p_0,\infty}_{s_2}} } { \lambda/2 } \right)^{p_0}
	\le \left(\frac{ C \tau^{ - \frac{d}{2} (\frac1{q_2}-\frac1{p_0})} } 
				{ \lambda } \right)^{p_0}
\]
and similarly,
\[
	d_{|\cdot|^{s_2} h}\left(\frac{\lambda}2\right)
	\le \left(\frac{ \|h \|_{L^{p_1,\infty}_{s_2}} } { \lambda/2 } \right)^{p_1}
	\le \left(\frac{ C \tau^{ \frac{d}{2} (\frac1{p_1}-\frac1{q_2})} } 
				{ \lambda } \right)^{p_1}. 
\]
Thus, choosing $\tau$ such that  
$\tau = \lambda^{-\frac{2q_2}{d}}$, 
we deduce
\begin{equation}\nonumber
	d_{|\cdot|^{s_2} g}(\lambda) 
	\le d_{|\cdot|^{s_2} h}\left(\frac{\lambda}2\right) 
		+ d_{|\cdot|^{s_2} \ell}\left(\frac{\lambda}2\right)
	\le  
\frac{C} {\lambda^{q_2} }, 
\end{equation}
which implies \eqref{Meyer-goal}. 
Thus, we conclude Proposition \ref{l:wMeyer.inq}.
\end{proof}

%%%%%%%%%%%%
%%% Section 4 %%%
%%%%%%%%%%%%
\section{Unconditional uniqueness and uniqueness criterion}\label{sec:4}
In this section, we prove Theorem \ref{thm:unconditional1}, Theorem \ref{thm:unconditional2}, Proposition \ref{prop:uniqueness-sufficient} and Theorem~\ref{thm:uniqueness-criterion}.

%%% Subsection 4.1 %%%
\subsection{Nonlinear estimates}\label{sub:4.1}

We define the Duhamel term $N(u)$ by 
\[
N(u) (t):= \int_{0}^t e^{(t-\tau)\Delta}(|\cdot|^{\gamma} |u(\tau)|^{\alpha-1} u(\tau))\, d\tau.
\]
Then we have the following nonlinear estimates, which are used to prove unconditional uniqueness in the double subcritical case and in the single critical case I.

% Lemma 4.1
\begin{lem}\label{lem:sec4_nonlinear1}
Let $d,\gamma,\alpha,q,s$ be as in \eqref{assum:main}. 
Let $T \in (0,\infty]$ and $\delta$ be given by 
\begin{equation}\label{def:delta}
\delta := 
\frac{d(\alpha-1)}{2}\left[
\frac{1}{q_c} - \left(
\frac{s}{d}+\frac{1}{q}
\right)
\right].
\end{equation}
Then the following assertions hold:
\begin{enumerate}[\rm (i)]
\item
If $0<\frac{s}{d} + \frac{1}{q} < \min\{\frac{1}{q_c}, \frac{1}{Q_c}\}$ and $q>\alpha$, then 
there exists a constant $C>0$ such that
\[
\begin{split}
& \|N(u_1)(t) - N(u_2)(t)\|_{L^{q,\infty}_{s}}  \\
& \le C t^{\delta}
 \max_{i=1,2} \|u_i\|_{L^\infty(0,t; L^{q,\infty}_{s})}^{\alpha-1}
		 \| u_1 - u_2\|_{L^\infty(0,t; L^{q,\infty}_{s})}
\end{split}
\]
for any $t\in (0,T)$ and $u_1, u_2 \in L^\infty(0,T ; L^{q,\infty}_{s}(\mathbb R^d))$.

\item
If either ``$0<\frac{s}{d} + \frac{1}{q} < \min\{\frac{1}{q_c}, \frac{1}{Q_c}\}$ and $q=\alpha$" or 
``$\frac{s}{d} + \frac{1}{q} = \frac{1}{Q_c} < \frac{1}{q_c}$", then 
there exists a constant $C>0$ such that
\[
\begin{split}
&  \|N(u_1)(t) - N(u_2)(t)\|_{L^{q,\alpha}_{s}} \\
& \le C t^{\delta}
 \max_{i=1,2} \|u_i\|_{L^\infty(0,t; L^{q,\alpha}_{s})}^{\alpha-1}
		 \| u_1 - u_2\|_{L^\infty(0,t; L^{q,\alpha}_{s})}
\end{split}
\]
for any $t\in (0,T)$ and $u_1, u_2 \in L^\infty(0,T ; L^{q,\alpha}_{s}(\mathbb R^d))$, provided that $q\not = \infty$.
\end{enumerate}
\end{lem}

% Remark
\begin{rem}
In (ii), the space of $u_1, u_2$ is restricted to $L^\infty(0,T ; L^{q,\alpha}_{s}(\mathbb R^d))$. Here, note that $L^\infty(0,T ; L^{q,\alpha}_{s}(\mathbb R^d)) \subsetneq L^\infty(0,T ; L^{q,\infty}_{s}(\mathbb R^d))$ (see Remark \ref{rem:w-Lorentz} (e)).
This restriction is due to the condition \eqref{linear-condi3} in Proposition \ref{prop:linear-main}.
\end{rem}

% Proof of Lemma 4.1
\begin{proof}
We define
$\sigma := \alpha s - \gamma$. 
First, we prove the assertion (i). 
Let $T \in (0,\infty]$ and $u_1, u_2 \in L^\infty(0,T ; L^{q,\infty}_{s}(\mathbb R^d))$. 
We assume \eqref{assum:main} and 
$0<\frac{s}{d} + \frac{1}{q} < \min\{\frac{1}{q_c}, \frac{1}{Q_c}\}$. 
Then the parameters $q,s,\sigma$ satisfy 
\[
1\le \frac{q}{\alpha} , q \le \infty,\quad 
0 < \frac{s}{d} + \frac{1}{q} < \frac{\sigma}{d} + \frac{\alpha}{q} < 1,\quad
s \le \sigma
\quad \text{and}\quad
d \left(\frac{\alpha}{q} - \frac{1}{q}\right) +\sigma - s < 2. 
\]
Hence, we use 
Proposition~\ref{prop:linear-main} with $(q_1,r_1,s_1) = (\frac{q}{\alpha}, \infty, \sigma)$ and $(q_2,r_2,s_2) = (q, \infty,s)$, and then,  
Lemma \ref{lem:Holder} with $(q,r)=(\frac{q}{\alpha},\infty)$, $(q_1,r_1) = (\frac{q}{\alpha-1},\infty)$ and 
$(q_2,r_2)=(q,\infty)$ 
 to obtain 
\begin{equation}\label{eq.DSC}
\begin{split}
& \|N(u_1)(t) - N(u_2)(t)\|_{L^{q,\infty}_{s}} \\
& \le C \int_0^t (t-\tau) ^{-\frac{d}{2} (\frac{\alpha}{q} - \frac{1}{q}) - \frac{\sigma - s}{2}}\\
& \qquad\qquad \times 
		\| |\cdot|^{\gamma} (|u_1(\tau)|^{\alpha-1}u_1(\tau) - |u_2(\tau)|^{\alpha-1}u_2(\tau)) \|_{L^{\frac{q}{\alpha},\infty}_{\sigma}} \, d\tau \\
& \le C
\int_0^t (t-\tau) ^{-\frac{d}{2} (\frac{\alpha}{q} - \frac{1}{q}) - \frac{\sigma - s}{2}}\\
&\qquad \qquad\times \||\cdot|^\gamma(|u_1(\tau)|^{\alpha-1} + |u_2(\tau)|^{\alpha-1})|u_1(\tau)-u_2(\tau)\|_{L^{\frac{q}{\alpha},\infty}_{\sigma}}\, d\tau \\
& \le 
C \int_0^t (t-\tau) ^{-\frac{d}{2} (\frac{\alpha}{q} - \frac{1}{q}) - \frac{\sigma - s}{2}} \, d\tau\\
& \qquad\qquad \times \max_{i=1,2} \|u_i\|_{L^\infty(0,t; L^{q,\infty}_{s})}^{\alpha-1}
		 \| u_1 - u_2\|_{L^\infty(0,t; L^{q,\infty}_{s})}\\
& \le C t^{\delta}
 \max_{i=1,2} \|u_i\|_{L^\infty(0,t; L^{q,\infty}_{s})}^{\alpha-1}
		 \| u_1 - u_2\|_{L^\infty(0,t; L^{q,\infty}_{s})}.
\end{split}
\end{equation}
Therefore, the assertion (i) is proved.

The assertion (ii) is also proved in the same way. In fact, when $\frac{s}{d} + \frac{1}{q} = \frac{1}{Q_c} < \frac{1}{q_c}$, 
we use Proposition~\ref{prop:linear-main} with the endpoint case \eqref{linear-condi3} to obtain 
\[
\begin{split}
& \| N(u_1) (t)- N(u_2)(t) \|_{L^{q,\alpha}_{s}}\\
& \le C
\int_0^t (t-\tau) ^{-\frac{d}{2} (\frac{\alpha}{q} - \frac{1}{q}) - \frac{\sigma - s}{2}}\\
& \qquad\qquad \times \||\cdot|^\gamma(|u_1(\tau)|^{\alpha-1} + |u_2(\tau)|^{\alpha-1})|u_1(\tau)-u_2(\tau)\|_{L^{\frac{q}{\alpha},1}_{\sigma}}\, d\tau \\
& \le C
t^\delta \max_{i=1,2} \|u_i\|_{L^\infty(0,t; L^{q,\alpha}_{s})}^{\alpha-1} \| u_1 - u_2\|_{L^\infty(0,t; L^{q,\alpha}_{s})}.
\end{split}
\]
Note that this case corresponds to taking the endpoint $\frac{\sigma}{d} + \frac{\alpha}{q} = 1$ in \eqref{eq.DSC}, which causes the restriction $r\le \alpha$. 
Here, 
the exponent $q=\infty$ is excluded (see Remark \ref{rem:w-Lorentz} (b)).
The proof in the case $0<\frac{s}{d} + \frac{1}{q} < \min\{\frac{1}{q_c}, \frac{1}{Q_c}\}$ and $q=\alpha$ is similar and also uses \eqref{linear-condi3}.
Thus, the proof of Lemma~\ref{lem:sec4_nonlinear1} is finished.
\end{proof}

In addition, we prepare the nonlinear estimates of the following type. These estimates are used to prove uniqueness criterion in the single critical case I, and unconditional uniqueness and uniqueness criterion in the scale-critical case. 

% Lemma 4.2
\begin{lem}\label{lem:sec4_nonlinear2}
Let $d,\gamma,\alpha,q,r,s$ be as in \eqref{assum:main}. Assume that $\tilde q \in (q, \infty)$ satisfies 
\begin{equation}\label{eq:lem_4.2_condi_s1}
\frac{s}{d} + \frac{1}{q} - \frac{2}{d(\alpha-1)} 
< \frac{s}{d} + \frac{1}{\tilde q}
< \min \left\{
\frac{1}{q_c}, \frac{1}{Q_c} - \frac{1}{\alpha-1} \left[
\frac{1}{Q_c} - \left(\frac{s}{d} + \frac{1}{q}\right)
\right]
\right\}.
\end{equation}
Let $T \in (0,\infty]$ and $\beta$ be defined by 
\begin{equation}\label{def:beta}
\beta = \beta (d, q,\tilde q) := \frac{d}{2}\left(\frac{1}{q} - \frac{1}{\tilde q} \right) .
\end{equation}
Then the following assertions hold:
\begin{enumerate}[\rm (i)]
\item
If $\frac{s}{d} + \frac{1}{q} = \frac{1}{Q_c} < \frac{1}{q_c}$ and $r>\alpha$, 
then
there exists a constant $C>0$ such that
\begin{equation}\label{lem_4.2:est1}
\begin{split}
\|N(u_1)(t) - N(u_2)(t)\|_{L^{q,r}_{s}}  & \le C t^{\delta}
\Big(
\max_{i=1,2}\|u_i-e^{\tau\Delta} u_0\|_{L^\infty(0,t; L^{q,r'(\alpha-1)}_{s})} \\
& \hspace{-2cm}
+
\sup_{0<\tau<t}\tau^{ \beta}\|e^{\tau\Delta} u_0\|_{L_{s}^{\tilde{q},\infty}}
\Big)^{\alpha-1}
\| u_1 - u_2\|_{L^\infty(0,t; L^{q,r}_{s})}
\end{split}
\end{equation}
for any $t\in (0,T)$ and $u_1, u_2 \in L^\infty(0,T ; L^{q,r}_{s}(\mathbb R^d))$ satisfying $u_i-e^{\tau\Delta} u_0 \in L^\infty(0,T; L^{q,r'(\alpha-1)}_{s}(\mathbb R^d))$ for $i=1,2$,
where $\delta$ is given by \eqref{def:delta}.

\item
If $\frac{s}{d} + \frac{1}{q} = \frac{1}{q_c} < \frac{1}{Q_c}$, then 
there exists a constant $C>0$ such that
\[
\begin{split}
\|N(u_1)(t) - N(u_2)(t)\|_{L^{q,\infty}_{s}}  & \le C 
\Big(
\max_{i=1,2}\|u_i-e^{\tau\Delta} u_0\|_{L^\infty(0,t; L^{q,\infty}_{s})} \\
&  \hspace{-2cm}
+
\sup_{0<\tau<t}\tau^{ \beta}\|e^{\tau\Delta} u_0\|_{L_{s}^{\tilde{q},\infty}}
\Big)^{\alpha-1}
\| u_1 - u_2\|_{L^\infty(0,t; L^{q,\infty}_{s})}
\end{split}
\]
for any $t\in (0,T)$ and $u_1, u_2 \in L^\infty(0,T ; L^{q,\infty}_{s}(\mathbb R^d))$.

\item
If $\frac{s}{d} + \frac{1}{q} = \frac{1}{q_c} = \frac{1}{Q_c}$, then 
there exists a constant $C>0$ such that
\[
\begin{split}
\|N(u_1)(t) - N(u_2)(t)\|_{L^{q,\infty}_{s}}  & \le C 
\Big(
\max_{i=1,2}\|u_i-e^{\tau\Delta} u_0\|_{L^\infty(0,t; L^{q,\alpha^*-1}_{s})} \\
&  \hspace{-2cm}
+ 
\sup_{0<\tau<t}\tau^{ \beta}\|e^{\tau\Delta} u_0\|_{L_{s}^{\tilde{q},\infty}}
\Big)^{\alpha^*-1}
\| u_1 - u_2\|_{L^\infty(0,t; L^{q,\infty}_{s})}
\end{split}
\]
for any $t\in (0,T)$ and $u_1, u_2 \in L^\infty(0,T ; L^{q,\infty}_{s}(\mathbb R^d))$
satisfying $u_i-e^{\tau\Delta} u_0 \in L^\infty(0,T; L^{q,\alpha^*-1}_{s}(\mathbb R^d))$ for $i=1,2$.
\end{enumerate}
\end{lem}

% Remark
\begin{rem}
In (ii) and (iii), the restriction on the second exponent $r=\infty$ in the left-hand side is due to use of the weighted Meyer inequality \eqref{eq.wMi} in Proposition \ref{l:wMeyer.inq}. 
In (iii), the reason why the space of $u_i-e^{\tau\Delta} u_0$ is restricted to be $L^\infty(0,T; L^{q,\alpha^*-1}_{s}(\mathbb R^d))$ is 
the endpoint condition \eqref{linear-condi4'} in Proposition \ref{l:wMeyer.inq}.
\end{rem}

% Proof of Lemma 4.2
\begin{proof}
Let $T \in (0,\infty]$. For two functions $u_1$ and $u_2$ on $(0,T)\times \mathbb R^d$, 
we estimate
\begin{equation}\label{firstcalcul1} 
\begin{split}
& |N(u_1)(t) - N(u_2)(t)| \\
	& \le  C\int_0^t e^{(t-\tau)\Delta}
		\left[|\cdot|^{\gamma}\left(|u_1(\tau)|^{\alpha-1}+|u_2(\tau)|^{\alpha-1}\right)|u_1(\tau) - u_2(\tau)|\right]
		\, d\tau \\ 
& \le C\int_0^t e^{(t-\tau)\Delta} 
		\left[|\cdot|^{\gamma}|u_1(\tau)-e^{\tau\Delta} u_0|^{\alpha-1}|u_1(\tau) - u_2(\tau)|\right]
		d\tau\\
	&+ C\int_0^t e^{(t-\tau)\Delta} 
		\left[|\cdot|^{\gamma}|u_2(\tau)-e^{\tau\Delta} u_0|^{\alpha-1}|u_1(\tau) - u_2(\tau)|\right]
		d\tau\\ 
	&+ C\int_0^t e^{(t-\tau)\Delta}
		\left[|\cdot|^{\gamma}|e^{\tau\Delta} u_0|^{\alpha-1}|u_1(\tau) - u_2(\tau)|\right]d\tau\\
	& =: I(t) + II(t) + III(t).
\end{split}
\end{equation}
First, we prove the assertion (i). Set $\sigma = \alpha s - \gamma$. 
In a similar way to the proof of  Lemma \ref{lem:sec4_nonlinear1} (ii), we estimate
\begin{equation}\label{5.3-est1}
\begin{split}
\|I(t)\|_{L^{q,r}_{s}}
& \le \int_0^t \|e^{(t-\tau)\Delta} (|\cdot|^\gamma |u_1(\tau)-e^{\tau\Delta} u_0|^{\alpha-1}|u_1(\tau) - u_2(\tau)|)|\|_{L^{q,r}_{s}} \, d\tau\\
& \le C
\int_0^t (t-\tau) ^{-\frac{d}{2} (\frac{\alpha}{q} - \frac{1}{q}) - \frac{\sigma - s}{2}}\, d\tau \\
&\qquad \qquad 
\times \sup_{0 <\tau< t}
 \||\cdot|^\gamma |u_1(\tau)-e^{\tau\Delta} u_0|^{\alpha-1} |u_1(\tau) - u_2(\tau)|\|_{L^{\frac{q}{\alpha},1}_{\sigma}}\\
 & 
 \le C t ^\delta \
\|u_1-e^{\tau\Delta} u_0\|_{L^\infty(0,t; L^{q,r'(\alpha-1)}_{s})}^{\alpha-1} 
\| u_1 - u_2\|_{L^\infty(0,t; L^{q,r}_{s})}
\end{split}
\end{equation}
for any $t \in (0,T)$,
where $1=\frac{1}{r}+\frac{1}{r'}$ and 
$\delta >0$ is given in \eqref{def:delta}.
Similarly, we have
\begin{equation}\label{5.3-est2}
\|II(t)\|_{L^{q,r}_{s}}
\le 
C t ^\delta \
\|u_2(\tau)-e^{\tau\Delta} u_0\|_{L^\infty(0,t; L^{q,r'(\alpha-1)}_{s})}^{\alpha-1} 
\| u_1 - u_2\|_{L^\infty(0,t; L^{q,r}_{s})}
\end{equation}
for any $t \in (0,T)$.
To estimate $III(t)$, 
we take auxiliary parameters $p$, $\tilde q$ and $\sigma$ satisfying
\begin{equation}\label{aux_condi2}
1<p<\infty,\quad 
q < \tilde q < \infty,\quad 
0 < \frac{s}{d} + \frac{1}{q} < \frac{\sigma}{d} + \frac{1}{p} < 1,
\quad s\le \sigma,
\end{equation}
\begin{equation}\label{aux_condi3}
\frac{1}{p} = \frac{\alpha-1}{\tilde q} + \frac{1}{q},
\end{equation}
\begin{equation}\label{aux_condi4}
-\frac{d}{2} \left(\frac{1}{p} - \frac{1}{q}\right) - \frac{\sigma - s}{2}>-1,
\quad -(\alpha-1)\beta >-1. 
\end{equation}
Here, the above $p$, $\tilde q$ and $\sigma$ exist if \eqref{assum:main} and \eqref{eq:lem_4.2_condi_s1} hold. 
We use Proposition~\ref{prop:linear-main} with 
$(q_1,r_1,s_1) = (p, \infty, \sigma)$ and $(q_2,r_2,s_2) = (q,r,s)$ to obtain 
\[
\|III(t)\|_{L^{q,r}_{s}}
 \le C
\int_0^t (t-\tau) ^{-\frac{d}{2} (\frac{1}{p} - \frac{1}{q}) - \frac{\sigma - s}{2}}
 \||\cdot|^\gamma |e^{\tau\Delta} u_0|^{\alpha-1} |u_1(\tau) - u_2(\tau)|\|_{L^{p,\infty}_{\sigma}}\, d\tau ,
\]
where \eqref{aux_condi2} is required.
Moreover, it follows from Lemma \ref{lem:Holder} with $(q,r) = (p,\infty)$, $(q_1,r_1)=(\frac{\tilde q}{\alpha-1}, \infty)$ and $(q_2,r_2)=(q,\infty)$ that 
\[
 \||\cdot|^\gamma |e^{\tau\Delta} u_0|^{\alpha-1} |u_1(\tau) - u_2(\tau)|\|_{L^{p,\infty}_{ \sigma}}
 \le C \|e^{\tau\Delta} u_0\|_{L^{\tilde{q},\infty}_{s}}^{\alpha-1}\|u_1(\tau) - u_2(\tau)\|_{L_{s}^{{q},\infty}},
\]
where \eqref{aux_condi3} is required.
Combining the above two estimates, and using the equality 
\[
-\frac{d}{2} \left(\frac{1}{p} - \frac{1}{q}\right) - \frac{\sigma - s}{2} -(\alpha-1)\beta + 1
= 
\delta
\]
which is a combination of \eqref{def:delta}, \eqref{def:beta} and \eqref{aux_condi3},
we have
\begin{equation}\label{5.3-est3}
\begin{split}
& \|III(t)\|_{L^{q,r}_{s}}\\
& \le 
Ct^{-\frac{d}{2} (\frac{1}{p} - \frac{1}{q}) - \frac{\sigma - s}{2} -(\alpha-1)\beta + 1}
	\left(\int_0^1 (1-\tau)^{-\frac{d}{2} (\frac{1}{p} - \frac{1}{q}) - \frac{\sigma - s}{2}}
	\tau^{-(\alpha-1)\beta}d\tau\right)\\ 
&\qquad 
	\times \left(\sup_{0<\tau<t}\tau^{ \beta}
	\|e^{\tau\Delta} u_0\|_{L_{\tilde s}^{\tilde{q},\infty}}^{\alpha-1}\right)
	\| u_1 - u_2\|_{L^\infty(0,t; L^{q,r}_{s})}\\
& \le
C t^{\delta} \left(\sup_{0<\tau<t}\tau^{ \beta}\|e^{\tau\Delta} u_0\|_{L_{s}^{\tilde{q},\infty}}^{\alpha-1}\right)
\| u_1 - u_2\|_{L^\infty(0,t; L^{q,r}_{s})},
\end{split}
\end{equation}
where \eqref{aux_condi4} is required.
Hence, summarizing \eqref{firstcalcul1}--\eqref{5.3-est2} and \eqref{5.3-est3}, 
we obtain \eqref{lem_4.2:est1}. 
Therefore, the assertion (i) is proved.

Next, we prove the assertion (ii).
Let $T \in (0,\infty]$ and $u_1, u_2 \in L^\infty(0,T ; L^{q,\infty}_{s}(\mathbb R^d))$.
In this case, the parameters $q,s,\sigma$ satisfy 
\begin{equation}\label{4-1_ii_condi}
0 < \frac{s}{d} + \frac{1}{q} < \frac{\sigma}{d} + \frac{\alpha}{q} < 1, 
\quad s \le \sigma
\quad \text{and}\quad 
\quad \frac{d}{2} \left(\frac{\alpha}{q} - \frac1{q} \right)+ \frac{\sigma-s}{2} = 1.
\end{equation}
We use 
Proposition \ref{l:wMeyer.inq} with the non-endpoint case as $(q_1,r_1,s_1) = (\frac{q}{\alpha}, \infty, \sigma)$ and $(q_2,s_2) = (q,s)$ to obtain 
\[
\begin{split}
\|I(t)\|_{L^{q,\infty}_{s}}
& \le C \sup_{0<\tau <t}
		\| |\cdot|^\gamma |u_1(\tau)-e^{\tau\Delta} u_0|^{\alpha-1}|u_1(\tau) - u_2(\tau)| \|_{L^{\frac{q}{\alpha},\infty}_{\sigma}} \, d\tau \\
& \le C
\|u_1-e^{\tau\Delta} u_0\|_{L^\infty(0,t; L^{q,\infty}_{s})}^{\alpha-1}
		 \| u_1 - u_2\|_{L^\infty(0,t; L^{q,\infty}_{s})}.
\end{split}
\]
Similarly, we have
\[
\|II(t)\|_{L^{q,\infty}_{s}}
\le C
\|u_2-e^{\tau\Delta} u_0\|_{L^\infty(0,t; L^{q,\infty}_{s})}^{\alpha-1}
		 \| u_1 - u_2\|_{L^\infty(0,t; L^{q,\infty}_{s})}.
\]
For the term $III(t)$, 
we can proceed as \eqref{5.3-est3} to obtain 
\[
\|III(t)\|_{L^{q,\infty}_{s}}
\le
C\left(\sup_{0<\tau<t}\tau^{ \beta}\|e^{\tau\Delta} u_0\|_{L_{s}^{\tilde{q},\infty}}^{\alpha-1}\right)
\| u_1 - u_2\|_{L^\infty(0,t; L^{q,\infty}_{s})}
\]
under the conditions \eqref{aux_condi2}--\eqref{aux_condi4}.
Therefore, the assertion (ii) is proved. 

For the assertion (iii), the proof can be done in the same way as the above (ii), but it corresponds to 
the endpoint case $\frac{\sigma}{d} + \frac{\alpha^*}{q} = 1$ in \eqref{4-1_ii_condi}.
For this, we use Proposition~\ref{l:wMeyer.inq} with the endpoint case \eqref{linear-condi4'}, which requires the stronger restriction on $r$, 
to obtain 
\[
\begin{split}
& \|I(t)\|_{L^{q,\infty}_{s}} + \|II(t)\|_{L^{q,\infty}_{s}}  \\
& \le C 
\max_{i=1,2}\|u_i-e^{\tau\Delta} u_0\|_{L^\infty(0,t; L^{q,\alpha^*-1}_{s})}^{\alpha^*-1}
\| u_1 - u_2\|_{L^\infty(0,t; L^{q,\infty}_{s})},
\end{split}
\]
where the condition $r = \alpha^*-1$ is required 
in the first norm of the right-hand side. 
The estimate for $III(t)$ is the same as in (ii).
Thus, (iii) is also proved.
\end{proof}

%%% Subsection 4.2 %%%
\subsection{Proofs of Theorems \ref{thm:unconditional1}, \ref{thm:unconditional2}, \ref{thm:uniqueness-criterion} and Proposition \ref{prop:uniqueness-sufficient}}

To begin with, we prepare the following lemma. 

% Lemma 4.3
\begin{lem}
\label{lem:density}
Let $d\in \mathbb N$, $1\le q, \tilde q\le \infty$,
$0<r\le \infty$ and $s\in \mathbb R$, and let $\beta$ be given by \eqref{def:beta}. 
Then, given a compact set $\mathcal{K}$ of $L^{q,r}_{s}(\mathbb R^d)$, 
there exists a function $\mu : (0,1)\to (0,\infty)$ such that 
\[
\lim_{t\to 0}\mu(t)=0
\]
and 
\[
t^{\beta}\|e^{t\Delta}f\|_{L_{s}^{\tilde q,\infty}}\leq \mu(t)
\] 
for any $t\in (0,1)$ and any $f\in \mathcal{K}$ (replace $L^{q, r}_{s}(\mathbb R^d)$ by $\mathcal{L}^{q, r}_{s}(\mathbb{R}^d)$ if $q=\infty$ or $r=\infty$). 
\end{lem}

The proof of this lemma can be done as in \cite[Lemma 8, page 283]{BreCaz1996} (see also \cite{MatosTerrane}) and uses the density of $L^{q, r}_{s}(\mathbb R^d) \cap L^\infty_0(\mathbb R^d)$ in $L^{q, r}_{s}(\mathbb R^d)$ or $\mathcal{L}^{q, r}_{s}(\mathbb{R}^d)$ (see Remark \ref{rem:w-Lorentz} (d) and Lemma \ref{lem:A-dense} in Appendix \ref{app:A}).\\

We are now in a position to prove the theorems. 

% Proof of Theorem 1.2
\begin{proof}[Proof of Theorem \ref{thm:unconditional1}]
We give the proof only for the case (2), since the proof of the case (1) is similar. 
Let $T>0$ and $u_1,u_2 \in L^\infty(0,T; L^{q,\alpha}_{s}(\mathbb R^d))$ be mild solutions to \eqref{HH} with initial data $u_1(0) = u_2(0)$. 
By Lemma \ref{lem:sec4_nonlinear1} (ii), we have 
\[
 \|u_1(t) - u_2(t)\|_{L^{q,\alpha}_{s}}  \le C_0 t^{\delta}
 \max_{i=1,2} \|u_i\|_{L^\infty(0,t; L^{q,\alpha}_{s})}^{\alpha-1}
		 \| u_1 - u_2\|_{L^\infty(0,t; L^{q,\alpha}_s)}
\]
for any $t \in (0,T)$, where $\delta>0$ is given in \eqref{def:delta}.
If we choose $t_0 \in (0,T)$ such that 
\[
C_0 t_0^{\delta}
 \max_{i=1,2} \|u_i\|_{L^\infty(0,T; L^{q,\alpha}_{s})}^{\alpha-1} < 1,
\]
then we can derive that $u_1=u_2$ on $[0,t_0]$.
We can repeat this argument until we reach $t=T$, and hence, we arrive at $u_1=u_2$ on $[0,T]$.
Thus, we conclude Theorem~\ref{thm:unconditional1}. 
\end{proof}

The proof of Proposition \ref{prop:uniqueness-sufficient} is similar to that of Theorem~\ref{thm:unconditional1}, and we have only to use  Lemma~\ref{lem:sec4_nonlinear2} (i) instead of Lemma \ref{lem:sec4_nonlinear1} (ii).

% Proof of Theorem 1.3
\begin{proof}[Proof of Theorem \ref{thm:unconditional2}]
We give the proof only for the case (2), since the proof of the case (1) is similar. 
Let $T>0$ and $u_1, u_2 \in C([0,T]; L^{q,\alpha^*-1}_{s}(\mathbb R^d))$ be mild solutions to \eqref{HH} with initial data $u_1(0) = u_2(0)=u_0 \in L^{q,\alpha^*-1}_{s}(\mathbb R^d)$.
By Lemma \ref{lem:sec4_nonlinear2} (iii), we have
\begin{equation}\label{4-2_1}
\begin{split}
\|u_1(t) - u_2(t)\|_{L^{q,\infty}_{s}}  & \le C 
\Big(
\max_{i=1,2}\|u_i-e^{\tau\Delta} u_0\|_{L^\infty(0,t; L^{q,\alpha^*-1}_{s})} \\
&
+ 
\sup_{0<\tau<t}\tau^{ \beta}\|e^{\tau\Delta} u_0\|_{L_{s}^{\tilde{q},\infty}}
\Big)^{\alpha-1}
\| u_1 - u_2\|_{L^\infty(0,t; L^{q,\infty}_{s})}
\end{split}
\end{equation}
for any $t\in (0,T)$, where $\tilde q \in (q,\infty)$.
Since $u_0 \in L^{q,\alpha^*-1}_{s}(\mathbb R^d)$, we see that 
\[
\begin{split}
& \|u_i-e^{\tau\Delta} u_0\|_{L^\infty(0,t; L^{q,\alpha^*-1}_{s})}\\
& \le 
\|u_i-u_0\|_{L^\infty(0,t; L^{q,\alpha^*-1}_{s})}
+
\|u_0-e^{\tau\Delta} u_0\|_{L^\infty(0,t; L^{q,\alpha^*-1}_{s})}
\end{split}
\]
for $i=1,2$. 
Since $u_1, u_2, e^{\tau\Delta} u_0 \in C([0,T]; L^{q,\alpha^*-1}_{s}(\mathbb R^d))$, the right-hand side converges to zero as $t\to 0$, and hence, 
\begin{equation}\label{4-2_2}
\lim_{t\to0}
\max_{i=1,2}\|u_i-e^{\tau\Delta} u_0\|_{L^\infty(0,t; L^{q,\alpha^*-1}_{s})}=0.
\end{equation}
On the other hand, we deduce from Lemma \ref{lem:density} that 
\begin{equation}\label{4-2_3}
\lim_{t\to0}\sup_{0<\tau<t}\tau^{ \beta}\|e^{\tau\Delta} u_0\|_{L_{s}^{\tilde{q},\infty}}=0.
\end{equation}
Hence, by \eqref{4-2_1}, \eqref{4-2_2} and \eqref{4-2_3}, there exists $t_0 \in (0,T]$ such that $u_1=u_2$ on $[0,t_0]$.
The extension of uniqueness to the whole interval $[0,T]$ can be done by the continuity argument as in \cite[Proof of Theorem 1.4]{Tay2020}. 
Thus, we conclude Theorem \ref{thm:unconditional2}.
\end{proof}

Theorem \ref{thm:uniqueness-criterion} is similarly proved to Theorem \ref{thm:unconditional2}, and so we omit the proof.

%%%%%%%%%%%%
%%% Section 5 %%%
%%%%%%%%%%%%
\section{Non-uniqueness}\label{sec:5}

In this section, 
we prove Theorem \ref{thm:nonuniqueness}, i.e.,
non-uniqueness for \eqref{HH} in the double critical case $\frac{s}{d} + \frac{1}{q} = \frac{1}{q_c} = \frac{1}{Q_c}$ (i.e. $\alpha=\alpha^*$).
For this purpose, we shall show the existence of two kind of mild solutions (regular and singular) 
to \eqref{HH} for arbitrary initial data $u_0 \in L^{q,r}_{s}(\mathbb R^d)$. 
For convenience, we define 
\[
q^*(\gamma):=\frac{d(\alpha^*-1)}{2} = \frac{d(2+\gamma)}{2(d-2)}.  
\]
Then we note that 
\[
q_c=Q_c = \frac{d}{d-2}=q^*(0).
\]

%%% Subsection 5.1 %%%
\subsection{Existence of the regular solution}\label{sub:5.1}

In this subsection, we prove the local in time existence of a mild solution $u$ to \eqref{HH} in $C([0,T] ; L^{q,r}_{s}(\mathbb R^d))$ 
with the auxiliary condition 
\begin{equation}\label{5.1:auxiliary-condi}
\|u\|_{\mathcal K^{\tilde q}(T)} := 
\sup_{0<t<T} t^{\beta} \|u(t)\|_{L^{\tilde q,\infty}_{s}} < \infty
\end{equation}
for $\tilde q > q$, 
where $\beta$ is given in \eqref{def:beta}.
The goal of this subsection is to prove the following: 

% Proposition 5.1
\begin{prop}\label{prop:regular1}
Let $d\ge3$, $\gamma>-2$, $\alpha = \alpha^*$, 
$\alpha^* \le q < \infty$, 
$0 < r \le \infty$, and $\frac{s}{d} + \frac{1}{q} = \frac{1}{q_c} = \frac{1}{Q_c}$. 
Assume that $\tilde q$ satisfies 
\begin{equation}\label{ass:regular1}
\max\left\{
0, \frac{1}{q}-\frac{1}{q^*(0)}, \frac{1}{q}-\frac{2}{d\alpha^*}
\right\}
< \frac{1}{\tilde q} < \frac{1}{q}.
\end{equation}
Then, 
for any $u_0 \in L^{q,r}_{s}(\mathbb R^d)$, there exist a time $T=T(u_0)>0$ 
and a unique mild solution $u \in C([0,T] ; L^{q,r}_{s}(\mathbb R^d))$ to \eqref{HH} with $u(0)=u_0$ satisfying \eqref{5.1:auxiliary-condi}
(replace $L^{q,r}_{s}(\mathbb R^d)$ by $\mathcal L^{q,\infty}_{s}(\mathbb R^d)$ if $r=\infty$). 
\end{prop}

The proof is based on the standard fixed point argument as in \cite[Subsection 3.1]{CIT2022}.
We prepare the following estimates on the Duhamel term $N(u)$.

% Lemma 5.2
\begin{lem}\label{lem:nonlinear_Kato}
Let $T>0$, and let $d\ge3$, $\gamma>-2$, $\alpha = \alpha^*$, 
$\alpha^* \le q < \infty$, 
$0 < r \le \infty$ and $\frac{s}{d} + \frac{1}{q} = \frac{1}{q_c} = \frac{1}{Q_c}$. 
\begin{enumerate}[\rm (i)]
\item Assume that $\tilde q$ satisfies \eqref{ass:regular1}. Then there exists a constant $C>0$ such that
\[
\|N(u_1) - N(u_2)\|_{\mathcal K^{\tilde q}(T)} 
\le C
\max_{i=1,2}
\|u_i\|_{\mathcal K^{\tilde q}(T)}^{\alpha-1}
\|u_1-u_2\|_{\mathcal K^{\tilde q}(T)}
\]
for any $t\in (0,T)$ and any functions $u_1,u_2$ satisfying 
\begin{equation}\label{eq:condi_u}
\|u_i\|_{\mathcal K^{\tilde q}(T)}<\infty,
\quad i=1,2.
\end{equation}

\item Assume that 
\begin{equation}\label{ass:regular2}
\max\left\{
0, \frac{1}{q}-\frac{2}{\alpha^*}
\right\}
< \frac{1}{\tilde q} < \frac{1}{q}.
\end{equation}
Then there exists a constant $C>0$ such that
\begin{equation}\label{eq.lem5.2_2}
\|N(u_1) - N(u_2)\|_{L^\infty(0,T ; L^{q,r}_s)} 
\le C
\max_{i=1,2}
\|u_i\|_{\mathcal K^{\tilde q}(T)}^{\alpha-1}
\|u_1-u_2\|_{\mathcal K^{\tilde q}(T)}
\end{equation}
for any $t\in (0,T)$ and any functions $u_1,u_2$ satisfying \eqref{eq:condi_u}.
\end{enumerate}
\end{lem}

% Remark 5.3
\begin{rem}
Note that \eqref{ass:regular1} implies \eqref{ass:regular2}. 
\end{rem}

% Proof of Lemma 5.2
\begin{proof}
We first prove the assertion (i). We set $\sigma := \alpha s - \gamma$ and take 
\begin{equation}\label{eq.cond_6-1}
1 < \tilde q, \frac{\tilde q}{\alpha} < \infty,
\quad 0< \frac{s}{d} + \frac{1}{\tilde q}
< \frac{\sigma}{d} + \frac{\alpha}{\tilde q} < 1,
\quad s \le \sigma,
\end{equation}
\begin{equation}\label{eq.cond_6-2}
-\frac{d}{2} \left(
\frac{\alpha}{\tilde q} - \frac{1}{\tilde q}
\right)
- \frac{\sigma -s}{2} >-1,\quad -\beta \alpha >-1. 
\end{equation}
Here, there exists a $\tilde q$ as above if \eqref{ass:regular1} holds.
In a similar way to \eqref{eq.DSC}, we estimate
\[
\begin{split}
& \|N(u_1)(t) - N(u_2)(t)\|_{L^{\tilde q,\infty}_s}\\
& \le C
\left(\int_0^t 
(t-\tau)^{-\frac{d}{2} \left(
\frac{\alpha}{\tilde q} - \frac{1}{\tilde q}
\right)
- \frac{\sigma -s}{2}}
\tau^{-\beta \alpha}\, d\tau
\right)
\max_{i=1,2}
\|u_i\|_{\mathcal K^{\tilde q}(T)}^{\alpha-1}
\|u_1-u_2\|_{\mathcal K^{\tilde q}(T)}\\
& \le C t^{-\beta} \max_{i=1,2}
\|u_i\|_{\mathcal K^{\tilde q}(T)}^{\alpha-1}
\|u_1-u_2\|_{\mathcal K^{\tilde q}(T)},
\end{split}
\]
where \eqref{eq.cond_6-1} and \eqref{eq.cond_6-2} are required in the first and second steps, respectively. 

Next, we prove the assertion (ii). 
By Lemma \ref{lem:interpolation_Lorentz}, we have 
\begin{equation}\label{eq.interpolation}
\begin{split}
& \|N(u_1)(t) - N(u_2)(t)\|_{L^{q,r}_s}\\
& \le C \|N(u_1)(t) - N(u_2)(t)\|_{L^{q_1,\infty}_s}^\theta
\|N(u_1)(t) - N(u_2)(t)\|_{L^{q_2,\infty}_s}^{1-\theta},
\end{split}
\end{equation}
where 
$1 < q_1<q<q_2\le \infty$, $0<\theta<1$ and 
$\frac{1}{q} = \frac{\theta}{q_1}+\frac{1-\theta}{q_2}$. 
For $j=1,2$, 
we take 
\begin{equation}\label{eq.cond_6-3}
1 < q_j, \frac{\tilde q}{\alpha} < \infty,
\quad 0< \frac{s}{d} + \frac{1}{q_j}
< \frac{\sigma}{d} + \frac{\alpha}{\tilde q} < 1,
\quad s \le \sigma,
\end{equation}
\begin{equation}\label{eq.cond_6-4}
-\frac{d}{2} \left(
\frac{\alpha}{\tilde q} - \frac{1}{q_j}
\right)
- \frac{\sigma -s}{2} >-1,\quad -\beta \alpha >-1,
\end{equation}
and 
we estimate
\[
\begin{split}
& \|N(u_1)(t) - N(u_2)(t)\|_{L^{q_j,\infty}_s}\\
& \le C
\left(\int_0^t 
(t-\tau)^{-\frac{d}{2} \left(
\frac{\alpha}{\tilde q} - \frac{1}{q_j}
\right)
- \frac{\sigma -s}{2}}
\tau^{-\beta \alpha}\, d\tau
\right)
\max_{i=1,2}
\|u_i\|_{\mathcal K^{\tilde q}(T)}^{\alpha-1}
\|u_1-u_2\|_{\mathcal K^{\tilde q}(T)}\\
& \le C \max_{i=1,2}
\|u_i\|_{\mathcal K^{\tilde q}(T)}^{\alpha-1}
\|u_1-u_2\|_{\mathcal K^{\tilde q}(T)},
\end{split}
\]
where \eqref{eq.cond_6-3} and \eqref{eq.cond_6-4} are required in the first and second steps, respectively. 
Here, there exists a $\tilde q$ as above if 
\[
\max\left\{
0, \frac{1}{q_j}-\frac{2}{\alpha^*}
\right\}
< \frac{1}{\tilde q} < \frac{1}{q_j}
\]
hold for $j=1,2$.
Therefore, we can obtain the required inequality \eqref{eq.lem5.2_2} for any $q,\tilde q$ satisfying \eqref{ass:regular2} 
if we take $q_1,q_2$ sufficiently close to $q$ so that $1 < q_1<q<q_2\le \infty$ and 
\[
\max\left\{
0, \frac{1}{q_1}-\frac{2}{\alpha^*}
\right\}
< \frac{1}{\tilde q} < \frac{1}{q_2}
\]
and we use \eqref{eq.interpolation} and perform the above argument.
Thus, the proof is finished. 
\end{proof}

% Proof of Proposition 5.1
\begin{proof}[Proof of Proposition \ref{prop:regular1}]
We give only a sketch of proof, as the proof is almost the same as in \cite[Subsection 3.1]{CIT2022}.
Let $u_0 \in L^{q,r}_{s}(\mathbb R^d)$, and let $\rho$ and $M$ be positive constants such that 
\[
\rho + C_0M^{\alpha^*} \le M
\quad \text{and}
\quad 
C_1 M^{\alpha^*-1} <\frac12,
\]
where
$C_0$ and $C_1$ are positive constants given in \eqref{5-1_1} and \eqref{5-1_2} below.
In addition, we take $T>0$ as 
\[
\|e^{t\Delta}u_0\|_{\mathcal K^{\tilde q}(T)} \le \rho.
\]
Now, we define a nonempty complete metric space $X_M$ by 
\[
X_{M} := \{
u \in \mathcal K^{\tilde q}(T)\,  ; \|u\|_{\mathcal K^{\tilde q}(T)} \le M \}
\]
with a metric $d(u_1,u_2):= \|u_1 -u_2\|_{\mathcal K^{\tilde q}(T)}$. Define a mapping $\Phi$ by 
\[
\Phi(u)(t)
:= e^{t\Delta}u_0 + N(u)(t)
\]
for $u \in X_M$. Then it follows from Lemma \ref{lem:nonlinear_Kato} (i) that
\begin{equation}\label{5-1_1}
\begin{split}
\|\Phi(u)\|_{\mathcal K^{\tilde q}(T)} 
& \le 
\|e^{t\Delta}u_0\|_{\mathcal K^{\tilde q}(T)} + \|N(u)\|_{\mathcal K^{\tilde q}(T)} \\
& \le 
\|e^{t\Delta}u_0\|_{\mathcal K^{\tilde q}(T)} + C_0\|u\|_{\mathcal K^{\tilde q}(T)}^{\alpha^*} \\
& \le \rho + C_0M^{\alpha^*} \le M,
\end{split}
\end{equation}
and
\begin{equation}\label{5-1_2}
\begin{split}
d(u_1,u_2)
& = \|N(u_1-u_2)\|_{\mathcal K^{\tilde q}(T)} \\
& \le 
C_1
\max_{i=1,2}
\|u_i\|_{\mathcal K^{\tilde q}(T)}^{\alpha^*-1}
\|u_1-u_2\|_{\mathcal K^{\tilde q}(T)}\\
& \le C_1 M^{\alpha^*-1} d(u_1,u_2) \le \frac12 d(u_1,u_2).
\end{split}
\end{equation}
for $u, u_1, u_2 \in X_M$. 
Hence, $\Phi$ is a contraction mapping from $X_M$ into itself. Thus, Banach's fixed point theorem ensures the existence of a unique fixed point $u \in X_M$ of $\Phi$. 
Finally, $u \in C([0,T] ; L^{q,r}_{s}(\mathbb R^d))$ follows from Lemma \ref{lem:nonlinear_Kato} (ii), Lemma \ref{lem:A_conti} and the well-known argument as in \cite{OkaTsu2016,Tsu2011} for instance.
The proof of Proposition \ref{prop:regular1} is finished.
\end{proof}

%%% Subsection 5.2 %%%
\subsection{Existence of singular solution}\label{sub:5.2}
The mild solution $u$ obtained in Subsection~\ref{sub:5.1} is a bounded solution (see \cite[Remark~1.1 and Proposition 3.2]{BenTayWei2017} and also \cite[the remark after Definition 2.1]{Wan1993}).
In this subsection, we find a singular mild solution $v$ to \eqref{HH} for any initial data $u_0 \in L^{q,r}_s(\mathbb R^d)$. 
Here, the singular mild solution means that $v(t) \not\in L^{\tilde q,\infty}_s(\mathbb R^d)$ for any $t \in [0,T]$ and for any $\tilde q$ satisfying \eqref{ass:regular1} 
(in particular, this solution has a singularity at $x=0$). 
The goal of this subsection is to prove the following: 

% Theorem 5.4
\begin{thm}\label{thm:singular_sol}
Let $d\ge3$, $\gamma>-2$, $\alpha = \alpha^*$, 
$\alpha^*\le q < \infty$, 
$\alpha^*-1< r \le \infty$, and $\frac{s}{d} + \frac{1}{q} = \frac{1}{q_c} = \frac{1}{Q_c}$. 
Then, for any $u_0 \in L^{q,r}_{s}(\mathbb R^d)$, 
there exist $T=T(u_0)>0$ and a mild solution $v \in C([0,T] ; L^{q, r}_{s}(\mathbb R^d))$ to \eqref{HH} with $v(0)=u_0$ such that 
$v \not\in L^{\tilde q,\infty}_s(\mathbb R^d)$ for any $\tilde q$ satisfying \eqref{ass:regular1} and 
\begin{equation}\label{aim-v}
v(t) - e^{t\Delta} u_0 \in 
L^{q, r}_{s}(\mathbb R^d) \setminus L^{q, \alpha^*-1}_{s}(\mathbb R^d) \text{ for any $r>\alpha^*-1$}
\end{equation}
for any $t\in (0,T]$ 
(replace $L^{q,r}_{s}(\mathbb R^d)$ by $\mathcal L^{q,\infty}_{s}(\mathbb R^d)$ if $r=\infty$). 
\end{thm}

The proof is based on the argument in \cite{MatosTerrane,Ter2002}. 
In order to construct the singular solution $v$, we use a positive, radially symmetric and singular stationary solution of 
\begin{equation}\label{Henon-ball}
		\Delta U + |x|^{\gamma} U^{\frac{d+\gamma}{d-2}} = 0\quad 
			\text{in } B\setminus \{0\}, \quad U>0,
\end{equation}
where $d\ge3$, $\gamma >-2$ and $B := \{x \in \mathbb R^d \,;\, |x| < 1\}$.
We have the results on the existence of the singular stationary solution and the sharp bound of its behavior at $x=0$.

% Theorem 5.5
\begin{thm}\label{thm:sss_sharp}
Let $d\ge3$ and $\gamma >-2$. The the following assertions hold:
\begin{enumerate}[\rm (i)]
\item
The equation \eqref{Henon-ball} has a positive, radial, and singular solution at $x=0$, where the singular solution means that it diverges at $x=0$.  

\item
Let $U \in C^2 (B\setminus \{0\})$ be a positive radial solution to \eqref{Henon-ball}. 
Then, $U$ has either a removable singularity at $|x|=0$ or a singularity at $|x|=0$ as 
\begin{equation}\label{B-behavior}
\lim_{x\to 0}
|x|^{d-2} |\log |x||^{\frac{d-2}{\gamma+2}}U(x) = 
\left(
\frac{(d-2)^2}{2+\gamma}
\right)^{\frac{d-2}{2+\gamma}}.
\end{equation}
\end{enumerate}
\end{thm}

% Remark 5.6
\begin{rem}
The constant \eqref{B-behavior} appears in \cite[Theorem 2.1]{BidRao1996} for $-2<\gamma<2$, and it gives the precise value to that in \cite[Theorem A]{Avi1987} and hence to that in \cite[Remark~6.2]{Tay2020}.
\end{rem}

The proofs of (i) and (ii) can be found in \cite[Example 1]{Avi1983} and \cite[Theorem 1.1 (ii)]{DZarxiv}, respectively. 
For completeness, we give the proof of (ii) in Appendix \ref{app:B}. 
Therefore, we denote by $U_0$ the singular stationary solution with 
\begin{equation}\label{eq.singular1}
U_0(x)  \sim |x|^{-(d-2)} |\log |x||^{-\frac{d-2}{\gamma+2}}  =|x|^{-(d-2)} |\log |x||^{-\frac{1}{\alpha^* -1}}
\end{equation}
near $x=0$. Then we note from Remark \ref{rem:w-Lorentz} {\rm (g)} that 
\begin{equation}\label{eq.singular2}
U_0 \in L^{q,r}_{s}(\mathbb R^d) \setminus L^{q,\alpha^*-1}_{s}(B)
\end{equation}
for any $r > \alpha^*-1$.\\

We extend $U_0$ to a function $V_0$ on $\mathbb R^d$ as follows.

 % Proposition 5.7
\begin{prop}
	\label{thm:singular-sol}
Let $d,\gamma,\alpha,q,s$ 
be as in Theorem \ref{thm:singular_sol}. 
Then there exists a function $V_0 \ge0$ on $\mathbb{R}^d\setminus \{0\}$ with compact support 
such that 
\[
	V_0(x) \sim 
	|x|^{-(d-2)} |\log |x||^{-\frac{1}{\alpha^* -1}}
\]
in a neighborhood of $x=0$, and 
\begin{equation}\label{phi-pro}
	R:= \Delta V_0 + |x|^{\gamma} V_0^{\alpha^*} \text{ is of $C^1$ with compact support}.
\end{equation}
Moreover, 
\begin{empheq}[left={\empheqlbrace}]{alignat=2}
  &V_0 \in L^{q,r}_{s}(\mathbb R^d) \setminus L^{q,\alpha^*-1}_{s}(\mathbb R^d)
\quad 
\text{for any $r > \alpha^* -1$}, \label{V_0:1}\\
  & V_0 \not \in 
L^{\tilde q,\infty}_{s}(\mathbb R^d) 
\quad \text{for any $\tilde q>q$}. \label{V_0:3}
\end{empheq} 
\end{prop}

The proof of Proposition \ref{thm:singular-sol} is the same as in \cite[Theorem 0.7]{Ter2002} (see also \cite[Proposition 6.1]{Tay2020}). \\

To prove Theorem \ref{thm:singular_sol}, 
we find a singular mild solution $v$ to \eqref{HH} of the form
\begin{equation}\label{form_v}
v(t) = w(t) + V_0.
\end{equation}
Here, $w=w(t)$ is a (regular) solution to the perturbed problem
\begin{equation}\label{perturbed-pro}
\begin{cases}
\displaystyle w(t) = e^{t\Delta}w_0 + \mathcal N(w)(t) + \int_0^t e^{(t-\tau)\Delta} 
R\, d\tau,\\
w(0) =w_0 := u_0 - V_0,
\end{cases}
\end{equation}
where
\[
\mathcal N(w)(t):= 
\int_0^t e^{(t-\tau)\Delta} 
\left(
|x|^\gamma |w(\tau) + V_0|^{\alpha^*-1} (w(\tau) + V_0)
-
|x|^\gamma V_0^{\alpha^*}
\right)\, d\tau.
\]
More precisely, we have the following:

% Lemma 5.8
\begin{lem}
\label{lem:pertubed}
Let $d\ge3$, $\gamma>-2$, $\alpha = \alpha^*$, 
$\alpha^* \le q < \infty$, 
$0 < r \le \infty$, and $\frac{s}{d} + \frac{1}{q} = \frac{1}{q_c} = \frac{1}{Q_c}$.
Then, for any $w_0 \in L^{q,r}_{s}(\mathbb R^d)$, there exist $T>0$ and a unique solution $w \in C([0,T]; L^{q,r}_{s}(\mathbb R^d))$ 
to \eqref{perturbed-pro} with $w(0)=w_0$ such that 
it satisfies 
\eqref{5.1:auxiliary-condi} for any $\tilde q$ satisfying \eqref{ass:regular1}
(replace $L^{q,r}_{s}(\mathbb R^d)$ by $\mathcal L^{q,\infty}_{s}(\mathbb R^d)$ if $r=\infty$).
\end{lem}

The proof of this lemma is based on the fixed point argument as in \cite{MatosTerrane}.
Hence, we need to show some estimates for the term $\mathcal N(w)$. 
To prove the estimates, we use the following decomposition of $V_0$. 
By the property \eqref{V_0:1} of $V_0$ and Lemma \ref{lem:A-dense} (i), 
for any $\varepsilon>0$, there exist functions $h\in L^{q,\infty}_{s}(\mathbb R^d) \cap L^\infty_0(\mathbb R^d)$
and $\overline{V}_0 \in L^{q,\infty}_{s}(\mathbb R^d)$
such that
\begin{equation}
\label{decompseV}
V_0=h+\overline{V}_0,\quad  \|\overline{V}_0\|_{L^{q,\infty}_{s}}<\varepsilon.
\end{equation}
Then we have the following estimates for $\mathcal N(w)$. 

% Lemma 5.9
\begin{lem}\label{lem:perturbed}
Let $d,\gamma,\alpha,q,r,s$ 
be as in Lemma \ref{lem:pertubed}, $\gamma_+ := \max\{0,\gamma\}$ and $\gamma_-:= -\min\{0,\gamma\}$. 
Assume $\tilde q$ satisfies \eqref{ass:regular1}. 
Then there exists a constant $C>0$ such that 
\begin{equation}\label{pertubed_nonlinear1}
\begin{split}
& \|\mathcal N(w_1) - \mathcal N(w_2)\|_{\mathcal K^{\tilde q}(t)}\\
&\le C
\left(
\max_{i=1,2}
\|w_i\|_{\mathcal K^{\tilde q}(t)}^{\alpha^*-1} + \|\overline{V}_0\|_{L^{q,\infty}_{s}}^{\alpha^*-1}
+ t^{1-\frac{\gamma_-}{2}}
 \||\cdot|^{\gamma_+}|h|^{\alpha^*-1}\|_{L^\infty}
\right)\\
& \qquad \times 
\|w_1-w_2\|_{\mathcal K^{\tilde q}(t)}
\end{split}
\end{equation}
and 
\begin{equation}\label{pertubed_nonlinear2}
\begin{split}
& \|\mathcal N(w_1) - \mathcal N(w_2)\|_{L^\infty(0,t; L^{q,r}_s)}\\
&\le C
\left(
\max_{i=1,2}
\|w_i\|_{\mathcal K^{\tilde q}(t)}^{\alpha^*-1} + \|\overline{V}_0\|_{L^{q,\infty}_{s}}^{\alpha^*-1}
+
 \|h\|_{L^{\tilde q,\infty}_s}^{\alpha^*-1}
\right)
\|w_1-w_2\|_{\mathcal K^{\tilde q}(t)}
\end{split}
\end{equation}
for any two functions $w_1,w_2$ satisfying \eqref{5.1:auxiliary-condi} 
and for any $t>0$. 
\end{lem}

% Proof of Lemma 5.9
\begin{proof}
We write
\[
\begin{split}
\mathcal{N}(w_1)(t)-\mathcal{N}(w_2)(t)
& = 
\int_0^t
e^{(t-\tau)\Delta}\big[|\cdot|^{\gamma}\left|w_1(\tau)+V_0\right|^{\alpha^*-1}\left(w_1(\tau)+V_0\right)\\
& 
\qquad -|\cdot|^{\gamma}\left| w_2(\tau)+V_0\right|^{\alpha^*-1}\left(w_2(\tau)+V_0\right)\big]d\tau.
\end{split}
\]
By the decomposition \eqref{decompseV} together with the inequality
\[
\begin{split}
&\left||x+y|^{\alpha^*-1}(x+y)  - |x'+y|^{\alpha^*-1}(x'+y)\right|\\
& \qquad \qquad \qquad \le  C|x-x'|\left(|x|^{\alpha^*-1}+|x'|^{\alpha^*-1}+|y|^{\alpha^*-1}\right)
\end{split}
\] 
for $x,x',y \in \mathbb{R}$, we have
\begin{equation} \label{eq0928}
\begin{split}
|\mathcal{N}(w_1)(t)-\mathcal{N}(w_2)(t)|
& \leq  C\int_0^te^{(t-\tau)\Delta}[|\cdot|^{\gamma}|w_1(\tau)|^{\alpha^*-1}|w_1(\tau)-w_2(\tau)|]\, d\tau\\
& +C\int_0^te^{(t-\tau)\Delta}[|\cdot|^{\gamma}|w_2(\tau)|^{\alpha^*-1}|w_1(\tau)-w_2(\tau)|]\, d\tau\\
& +C\int_0^te^{(t-\tau)\Delta}[|\cdot|^{\gamma}|h|^{\alpha^*-1}|w_1(\tau)-w_2(\tau)|]\, d\tau\\
& +C\int_0^te^{(t-\tau)\Delta}[|\cdot|^{\gamma}|\overline{V}_0|^{\alpha^*-1}|w_1(\tau)-w_2(\tau)|]\, d\tau\\
& =: I(t) + II(t) + III(t) + IV(t).
\end{split}
\end{equation}

First, we prove the estimate \eqref{pertubed_nonlinear1}. 
In the same way as in the proof of Lemma~\ref{lem:sec4_nonlinear2}, 
the norms of the terms $I(t)$ and $II(t)$ can be estimated as 
\begin{equation}\label{eq1-1}
\|I\|_{\mathcal K^{\tilde q}(t)}
+
\|II\|_{\mathcal K^{\tilde q}(t)}
\le C
 \max_{i=1,2}
\|w_i\|_{\mathcal K^{\tilde q}(t)}^{\alpha^*-1}
\|w_1-w_2\|_{\mathcal K^{\tilde q}(t)}.
\end{equation}

As to the term $III(t)$, 
we use Proposition \ref{prop:linear-main} with 
$(q_1,r_1,s_1) = (\tilde q, \infty, s + \gamma_-)$ 
and $(q_2,r_2,s_2) = (\tilde q,\infty,s)$ to obtain 
\begin{equation}\label{eq1-2}
\begin{split}
& \|III(t)\|_{L^{\tilde q,\infty}_{s}}\\
& \le C
\int_0^t 
(t-\tau)^{- \frac{\gamma_-}{2}}
\|
|\cdot|^{\gamma}|h|^{\alpha^*-1}|w_1(\tau)-w_2(\tau)|
\|_{L^{\tilde q}_{s + \gamma_-}}\, d\tau\\
& = C
\int_0^t 
(t-\tau)^{- \frac{\gamma_-}{2}}
\|
|\cdot|^{\gamma_+}|h|^{\alpha^*-1}|w_1(\tau)-w_2(\tau)|
\|_{L^{\tilde q,\infty}_{s}}\, d\tau\\
& \le C \||\cdot|^{\gamma_+}|h|^{\alpha^*-1}\|_{L^\infty}
\int_0^t 
(t-\tau)^{- \frac{\gamma_-}{2}}
\|w_1(\tau)-w_2(\tau)
\|_{L^{\tilde q,\infty}_{s}}\, d\tau\\
& \le C \||\cdot|^{\gamma_+}|h|^{\alpha^*-1}\|_{L^\infty}
\left(\int_0^t 
(t-\tau)^{- \frac{\gamma_-}{2}}
\tau^{-\beta}\, d\tau 
\right)
\|w_1-w_2\|_{\mathcal K^{\tilde q}(t)}\\
& \le C t^{1-\frac{\gamma_-}{2}-\beta} 
 \||\cdot|^{\gamma_+}|h|^{\alpha^*-1}\|_{L^\infty}
\|w_1-w_2\|_{\mathcal K^{\tilde q}(t)},
\end{split}
\end{equation}
where we required that 
\[
0< \frac{s}{d} + \frac{1}{\tilde q} \le \frac{s+\gamma_-}{d} + \frac{1}{\tilde q} <1\quad \text{and}\quad s\le s+\gamma_-. 
\]
Here, thanks to \eqref{ass:regular1} and $\gamma_-\in[0,2)$, 
the above conditions are satisfied. 

As to the term $IV(t)$, thanks to \eqref{ass:regular1}, we can take $\sigma := \alpha^* s - \gamma$ and 
\[
0< \frac{s}{d} + \frac{1}{\tilde q} \le \frac{\sigma}{d} + \frac{1}{p} < 1,
\quad s\le \sigma,\quad \frac{1}{p} = \frac{\alpha^*-1}{q}+\frac{1}{\tilde q}. 
\]
Then we 
use Proposition \ref{l:wMeyer.inq} with 
$(q_1,r_1,s_1) = (p, \infty, \sigma)$ 
and $(q_2,r_2,s_2) = (\tilde q,\infty,s)$ to obtain 
\begin{equation}\label{eq1-3}
\begin{split}
\|IV(t)\|
_{L^{\tilde q,\infty}_{s}}
& \le 
C \sup_{0 <\tau <t} \||\cdot|^{\gamma}|\overline{V}_0|^{\alpha^*-1}|w_1(\tau)-w_2(\tau)|\|_{L^{p,\infty}_{\sigma}}\\
& \le 
C \sup_{0 <\tau <t} \|(|\cdot|^{s}|\overline{V}_0|)^{\alpha^*-1}|\cdot|^{s}||w_1(\tau)-w_2(\tau)|\|_{L^{p,\infty}}\\
& \le C t^{-\beta}
\|\overline{V}_0\|_{L^{q,\infty}_{s}}^{\alpha^*-1}
\|w_1-w_2\|_{\mathcal K^{\tilde q}(t)}.
\end{split}
\end{equation}
By combining \eqref{eq0928}, \eqref{eq1-1}, \eqref{eq1-2} and \eqref{eq1-3}, we obtain 
\eqref{pertubed_nonlinear1}. 

For the estimate \eqref{pertubed_nonlinear2}, 
it is enough to use the interpolation argument with 
Lemma~\ref{lem:interpolation_Lorentz} such as the proof of Lemma~\ref{lem:nonlinear_Kato} (ii) 
to deal with all $r\in (0,\infty]$.
The terms 
$I(t)$, $II(t)$ and $IV(t)$ can be estimated in a similar way to \eqref{pertubed_nonlinear1}. 
The estimate for $III(t)$ is proved in a similar way to Lemma~\ref{lem:nonlinear_Kato} (ii). 
The proof of Lemma \ref{lem:perturbed} is finished. 
\end{proof}

\begin{proof}[Proof of Lemma \ref{lem:pertubed}]
Let $w_0 \in L^{q,r}_{s}(\mathbb R^d)$, and let $\rho, \delta, \varepsilon$ and $M$ be positive constants such that 
\[
\rho + C_1 (M^{\alpha^*-1} + \varepsilon^{\alpha^*-1} + \delta) M \le M
\quad \text{and}
\quad 
C_2
\left(
M^{\alpha^*-1} + \varepsilon^{\alpha^*-1}
+ \delta
\right)<\frac12,
\]
where
$C_1$ and $C_2$ are positive constants given in \eqref{5-2_1} and \eqref{5-2_2} below.
In addition, we take $T>0$ as 
\[
\|e^{t\Delta}w_0\|_{\mathcal K^{\tilde q}(T)} + C_0 T \|R\|_{L^{q,\infty}_{s}} \le \rho
\]
and 
\[
T^{1-\frac{\gamma_-}{2}}
 \||\cdot|^{\gamma_+}|h|^{\alpha^*-1}\|_{L^\infty} \le \delta,
\]
where $C_0$ is a positive constant given in \eqref{5-2_1} below.
Now, we define a nonempty complete metric space $X_M$ by 
\[
X_{M} := \{
w \in \mathcal K^{\tilde q}(T)\,  ; \|w\|_{\mathcal K^{\tilde q}(T)} \le M
 \}
\]
with a metric $d(u_1,u_2):= \|u_1 -u_2\|_{\mathcal K^{\tilde q}(T)}$. Define a mapping $\Phi$ by 
\[
\Phi(w)(t)
:= e^{t\Delta}w_0 +\mathcal N(w)(t) + \int_0^t e^{(t-\tau)\Delta} 
R\, d\tau
\]
for $w \in X_M$.
By \eqref{pertubed_nonlinear1} in Lemma \ref{lem:perturbed} and \eqref{decompseV}, it follows that 
\begin{equation}\label{5-2_1}
\begin{split}
& \|\Phi(w)\|_{\mathcal K^{\tilde q}(T)}\\
& \le \|e^{t\Delta}w_0\|_{\mathcal K^{\tilde q}(T)} +
C_0 T \|R\|_{L^{q,\infty}_{s}}\\
& + C_1
\left(
\|w\|_{\mathcal K^{\tilde q}(T)}^{\alpha^*-1} + \|\overline{V}_0\|_{L^{q,\infty}_{s}}^{\alpha^*-1}
+ T^{1-\frac{\gamma_-}{2}}
 \||\cdot|^{\gamma_+}|h|^{\alpha^*-1}\|_{L^\infty}
\right)
\|w\|_{\mathcal K^{\tilde q}(T)}\\
& \le \rho + C_1 (M^{\alpha^*-1} + \varepsilon^{\alpha^*-1} + \delta) M \le M
\end{split}
\end{equation}
and
\begin{equation}\label{5-2_2}
\begin{split}
& d(\Phi(w_1), \Phi(w_2)) \le \|\mathcal N(w_1)- \mathcal N(w_2)\|_{\mathcal K^{\tilde q}(T)}\\
& \le  C_2
\left(
\max_{i=1,2}
\|w_i\|_{\mathcal K^{\tilde q}(T)}^{\alpha^*-1} + \|\overline{V}_0\|_{L^{q,\infty}_{s}}^{\alpha^*-1}
+ T^{1-\frac{\gamma_-}{2}}
 \||\cdot|^{\gamma_+}|h|^{\alpha^*-1}\|_{L^\infty}
\right)\\
& \qquad \qquad \times 
\|w_1-w_2\|_{\mathcal K^{\tilde q}(T)}\\
& \le 
C_2
\left(
M^{\alpha^*-1} + \varepsilon^{\alpha^*-1}
+ \delta
\right)
d(w_1,w_2) \le \frac12d(w_1,w_2)
\end{split}
\end{equation}
for $w,w_1,w_2 \in X_M$.
Hence, $\Phi$ is a contraction mapping from $X_M$ into itself. Thus, Banach's fixed point theorem ensures the existence of a unique fixed point $w \in X_M$ of $\Phi$. 
Finally, $w \in C([0,T] ; L^{q,r}_{s}(\mathbb R^d))$ follows from \eqref{pertubed_nonlinear2} in Lemma \ref{lem:perturbed}, Lemma \ref{lem:A_conti} and the well-known argument as in \cite{OkaTsu2016,Tsu2011} for instance.
The proof of Lemma \ref{lem:pertubed} is finished.
\end{proof}

\begin{proof}[Proof of Theorem \ref{thm:singular_sol}]
The existence part of Theorem \ref{thm:singular_sol} immediately follows from a combination of Lemma \ref{lem:pertubed} with Proposition \ref{thm:singular-sol} and \eqref{form_v}. 
The remaining part, i.e., the properties \eqref{aim-v} of $v$, can be proved 
in a similar way to the proof of \cite[Proposition~8.2]{Tay2020}.
In fact, we decompose the Duhamel term $v(t) - e^{t\Delta}u_0$ into the following three terms:
\[
v(t) - e^{t\Delta}u_0
= 
(w(t) - e^{t\Delta}w_0) - e^{t\Delta}V_0 + V_0.
\]
The first term $w(t) - e^{t\Delta}w_0$ can be rewritten as 
\[
w(t) - e^{t\Delta}w_0 = \mathcal N(w)(t) + \int_0^t e^{(t-\tau)\Delta} R\, d\tau.
\]
We see from Lemma \ref{lem:perturbed} and the property \eqref{phi-pro} of $R$ that
both terms in the right-hand side belong to $L_{s}^{q,\tilde r}(\mathbb R^d)$ for any $\tilde r>0$ and any $t \in (0,T]$, 
and hence, $w(t) - e^{t\Delta}w_0$ also belongs to $L_{s}^{q,\tilde r}(\mathbb R^d)$. % for any $\tilde r>0$ and any $t \in (0,T]$.
As to the second term $e^{t\Delta} V_0$, we estimate  
\[
\| e^{t\Delta} V_0\|_{L_{s}^{q,\tilde r}}
\le 
\begin{cases}
\displaystyle C t^{-\frac{d}{2}(1- \frac{1}{q}) + \frac{s}{2}}
\|V_0\|_{L^1} = 
C t^{-1}
\|V_0\|_{L^1} \quad &\text{if }s\le 0,\\
\displaystyle  C t^{-\frac{d}{2}(1- \frac{1}{q})}
\|V_0\|_{L^1_{s}} = 
C_{V_0} t^{-\frac{d}{2}(1- \frac{1}{q})}
\|V_0\|_{L^1} \quad &\text{if }s > 0,
\end{cases}
\]
where we used Propositions~\ref{prop:linear-main} and $V_0 \in L^1(\mathbb R^d)$ with compact support in Proposition \ref{thm:singular-sol}.
Hence, $e^{t\Delta}V_0 \in L_{s}^{q,\tilde r}(\mathbb{R}^d)$ is also shown for any $\tilde r>0$ and $t \in (0,T]$. 
In contrast, the third term $V_0$ satisfies 
$V_0 \not \in L_{s}^{q,\alpha^*-1}(\mathbb{R}^d)$ and $V_0 \in L_{s}^{q,r}(\mathbb{R}^d)$ for any $r>\alpha^*-1$
by Proposition~\ref{thm:singular-sol}. 
Therefore, \eqref{aim-v} is proved for any $t \in (0,T]$. Thus, Theorem \ref{thm:singular_sol} is proved.
\end{proof}

% Proof of Theorem 1.6
\begin{proof}[Proof of Theorem \ref{thm:nonuniqueness}]
The proof is a combination of Proposition \ref{prop:regular1} and Theorem~\ref{thm:singular_sol}. 
In fact, by these results, 
there exist 
a regular mild solution $u$ and singular mild solution $v$ to \eqref{HH} with the same initial data $u_0$. 
When $r=\infty$, the above arguments are also valid if $L^{q,r}_{s}(\mathbb R^d)$ is replaced by $\mathcal L^{q,\infty}_{s}(\mathbb R^d)$.
\end{proof}

%%%%%%%%%%%%
%%% Section 6 %%%
%%%%%%%%%%%%
\section{Scale-supercritical case}\label{sec:6}
In this section we discuss the scale-supercritical case. 
We use the self-similar solution of \eqref{HH} to show the existence of a non-trivial mild solution of \eqref{HH} with initial data $0$. 
More precisely, we have the following:

% Proposition 6.1
\begin{prop}
\label{nonuniqhiros}
Let $d\ge 3$, $\gamma>-2$, $\alpha >\alpha_F$, $1< q \le \infty$, $1\le r\le \infty$ and $s\in \mathbb R$ be such that
\[
\frac{1}{q_c} < \frac{s}{d} + \frac{1}{q} < 1.
\]
Assume that there exists a solution $W$ of 
\begin{equation}
\label{equationprofil}
\Delta W+{1\over 2}x\cdot \nabla W+{2+\gamma\over 2(\alpha-1)}W+|x|^{\gamma}|W|^{\alpha-1} W=0,\quad 
x \in \mathbb{R}^d\setminus\{0\}
\end{equation}
such that 
\begin{enumerate}[\rm (i)]
\item $W>0$ and $W\in C(\mathbb{R}^d)\cap C^2(\mathbb{R}^d\setminus\{0\}),$
\item $\displaystyle \lim_{|x|\to 0}|x||\nabla W|=0,$
\item $\displaystyle \lim_{|x|\to \infty}|x|^mW(x)=0$ and $\displaystyle \lim_{|x|\to \infty}|x|^m|\nabla W(x)|=0$ for all $m>0$. 
\end{enumerate}
Let $\Psi(t,x)=t^{-{2+\gamma\over 2(\alpha-1)}}W(x/\sqrt{t})$ be the positive self-similar solution of \eqref{HH}.
Then $\Psi\in C([0,\infty); L^{q,r}_s(\mathbb{R}^d))$ satisfies the equation
\[
\Psi(t)=
\int_0^t
e^{(t-\tau)\Delta}\left(|\cdot|^{\gamma}|\Psi(\tau)|^\alpha
\Psi(\tau)\right)d\tau
\]
for any $t \in (0,\infty)$. 
In particular, $\Psi$ is a non-trivial mild solution to \eqref{HH} with initial data $0$ in 
$C([0,\infty); L^{q,r}_s(\mathbb{R}^d))$.
\end{prop}
\begin{proof}  By the assumptions (i)--(iii) on $W,$ it follows that
\[
{1\over 2}x\cdot \nabla W+{2+\gamma\over 2(\alpha-1)}W+|x|^{\gamma}|W|^{\alpha-1} W\in L^1(\mathbb{R}^d).
\]
Then $W$ satisfies the equation \eqref{equationprofil} in $ \mathcal{D}'(\mathbb{R}^d)$ and 
$\Psi(t,x)=t^{-{2+\gamma\over 2(\alpha-1)}}W(x/\sqrt{t})$ satisfies the equation \eqref{HH} in $ \mathcal{D}'((0,\infty)\times\mathbb{R}^d)$.
Here, $\mathcal{D}'(X)$ is the space of distributions on an open set $X$.
Hence,
\[
\Psi(t)=e^{(t-\varepsilon)\Delta}\Psi(\varepsilon)+\displaystyle\int_\varepsilon^te^{(t-\tau)\Delta}\left(|\cdot|^{\gamma}|\Psi(\tau)|^{\alpha-1}
\Psi(\tau)\right)d\tau
\]
for $0<\varepsilon<t$ in the sense of distributions. It is clear that
\[
\|\Psi(t)\|_{L^{q,r}_{s}}=t^{\frac{d}{2} ((\frac{s}{d} + \frac{1}{q}) - \frac{1}{q_c})}\|W\|_{L^{q,r}_{ s}}
<\infty,\quad t>0,
\]
where $0<\frac{s}{d}+\frac{1}{q} <1$. 
Then 
\begin{equation}\label{condi_7_1}
\lim_{t\to 0}\|\Psi(t)\|_{L^{q,r}_s}=0
\quad \text{for }\frac{1}{q_c} < \frac{s}{d} + \frac{1}{q} < 1.
\end{equation}
Finally, we prove that the integral 
\begin{equation}\label{integral-psi}
\int_0^t e^{(t-\tau)\Delta}(|\cdot|^{\gamma}|\Psi(\tau)|^\alpha
\Psi(\tau))\, d\tau
\end{equation}
converges absolutely in $L^{q,r}_s(\mathbb R^d)$. 
By Proposition \ref{prop:linear-main}, 
we have
\[
\begin{split}
& \left\|e^{(t-\tau)\Delta}\left(|\cdot|^{\gamma}|\Psi(\tau)|^{\alpha-1}
\Psi(\tau)\right)\right\|_{L^{q,r}_s}\\
& \le C(t-\tau)^{- {d\over 2}({\alpha\over \tilde{q}}-{1\over q})-{\alpha \tilde{s}-\gamma-s\over 2}}\|\Psi(\tau)\|_{L^{\tilde{q},r}_{\tilde{s}}}^{\alpha}\\
& = C(t-\tau)^{- {d\over 2}({\alpha\over \tilde{q}}-{1\over q})-{\alpha \tilde{s}-\gamma-s\over 2}}\tau^{\alpha({{d\over 2\tilde{q}}+{\tilde{s}\over 2}-{2+\gamma\over 2(\alpha-1)}})}\|W\|_{L^{\tilde{q},r}_{\tilde{s}}}^{\alpha},
\end{split}
\]
where we require that 
\begin{equation}\label{condi_7_2new}
\alpha < \tilde q <\infty,\quad 
0 < \frac{s}{d}+\frac{1}{q} < \frac{\alpha \tilde s-\gamma}{d} + \frac{\alpha}{\tilde q} < 1,
\quad s \le \alpha \tilde s-\gamma. 
\end{equation}
If $\alpha$, $q$, $s$, ${\tilde{q}}$ and ${\tilde{s}}$  satisfy
\begin{equation}\label{condi_7_3new}
{d\over 2}\left({\alpha\over \tilde{q}}-{1\over q}\right)+{\alpha \tilde{s}-\gamma-s\over 2}<1,\quad 
\alpha\left({2+\gamma\over 2(\alpha-1)}-{d\over 2}\left({{1\over \tilde{q}}+{\tilde{s}\over d}}\right)\right)<1,
\end{equation}
then \eqref{integral-psi} converges absolutely in $L^{q,r}_s(\mathbb R^d)$.
To check these conditions, let us choose $\tilde{q}$ and $\tilde{s}$ such that
\begin{equation}
\label{conditionsqstilde1}{s+\gamma\over \alpha}\leq  \tilde{s},\quad  0<{\alpha\over\tilde{q}}<1,\quad {\alpha\over \tilde{q}}+{\alpha \tilde{s}\over d}<{\gamma+d\over d},
\end{equation}
\begin{equation}
\label{conditionsqstilde2}{1\over q}+{\gamma+s\over d}< {\alpha\over \tilde{q}}+{\alpha \tilde{s}\over d}<{2\over d}+{1\over q}+{\gamma+s\over d}.
\end{equation}
It is obvious that under the assumptions in Proposition~\ref{nonuniqhiros}, it is possible to take $\tilde{q},$ $\tilde{s}$ satisfying \eqref{conditionsqstilde1} and \eqref{conditionsqstilde2}.
We now show that \eqref{condi_7_2new} and \eqref{condi_7_3new} hold if \eqref{condi_7_1}, \eqref{conditionsqstilde1} and \eqref{conditionsqstilde2} are satisfied. 
Indeed, \eqref{condi_7_2new} is already in \eqref{conditionsqstilde1} and the first inequality in \eqref{conditionsqstilde2}. 
For \eqref{condi_7_3new}, we have
\[
\begin{split}
{d\over 2}\left({\alpha\over \tilde{q}}-{1\over q}\right)+{\alpha \tilde{s}-\gamma-s\over 2} 
&=  {d\alpha\over 2\tilde{q}}-{d\over 2q}+{\alpha \tilde{s}-\gamma-s\over 2}\\
&= {d\alpha\over 2\tilde{q}}+{\alpha \tilde{s}\over 2}-{d\over 2q}-{\gamma+s\over 2}\\  
&<1+ {d\over 2q}+{\gamma+s\over 2}-{d\over 2q}-{\gamma+s\over 2}\\
&= 1,
\end{split}
\]
\[
\begin{split}
\alpha\left({2+\gamma\over 2(\alpha-1)}-{d\over 2}\left({{1\over \tilde{q}}+{\tilde{s}\over d}}\right)\right)
&=  {\alpha(2+\gamma)\over 2(\alpha-1)}-{d\over 2}\left({\alpha\over \tilde{q}}+{\alpha\tilde{s}\over d}\right)\\
&< {\alpha(2+\gamma)\over 2(\alpha-1)}-{d\over 2}\left({1\over q}+{\gamma+s\over d}\right)\\ 
&=1+ {\gamma\over 2}+{2+\gamma\over 2(\alpha-1)}-{d\over 2q}-{\gamma+s\over 2}\\
&< 1.
\end{split}
\]
Thus, we conclude Proposition \ref{nonuniqhiros}. 
\end{proof}

The existence of positive self-similar solutions $\Psi$ of \eqref{HH} with (i)--(iii) in Proposition \ref{nonuniqhiros} 
is proved for any $\alpha$ satisfying
\begin{equation}\label{Sobolev-sub}
\alpha_F <
\alpha<\alpha_{HS}
\end{equation}
with $\gamma =0$ by \cite[Propositions 3.1, 3.4 and 3.5]{HarWei1982} and 
with $\gamma$ 
satisfying 
\begin{equation}\label{restriction-gamma}
-2 < \gamma \le 
\begin{cases}
\sqrt{3}-1\quad &\text{if }d=3,\\
0 &\text{if }d\ge4
\end{cases}
\end{equation}
by Hirose \cite[Theorem 1.2 (ii)]{hi}.
Furthermore, 
\[
W(x)=C|x|^{\frac{2+\gamma}{\alpha-1}-d}e^{-{|x|^2\over 4}}\left(1+O\left(|x|^{-2}\right)\right)\quad  \text{as }|x|\to \infty.
\]
From Proposition \ref{nonuniqhiros} and this result, it immediately follows that 
the equation \eqref{HH} has three different solutions $0$ and $\pm \Psi$ with initial data $0$
in $C([0,\infty); L^{q,r}_s(\mathbb{R}^d))$ under the assumptions \eqref{Sobolev-sub} and \eqref{restriction-gamma} for 
$d,\gamma,\alpha,q,r,s$ as in Proposition \ref{nonuniqhiros}. 
Thus, Proposition \ref{prop:non-uniquness_super} is proved.

% Remark 6.3
\begin{rem}
When $\gamma$ does not satisfy \eqref{restriction-gamma}, 
the existence of self-similar solutions with {\rm (i)--(iii)} in Proposition \ref{nonuniqhiros} under the condition \eqref{Sobolev-sub} is an open problem. 
\end{rem}

The situation of the case $\alpha > \alpha_{HS}$ is different from the case \eqref{Sobolev-sub}. 
In this case, the nonexistence of positive self-similar solution $\Psi$ satisfying (i)--(iii) in Proposition~\ref{nonuniqhiros} 
is proved by 
the following result on uniqueness in the Sobolev space $H^1(\mathbb R^d)$:

% Lemma 6.4
\begin{lem}\label{lem:uniqueness-H1}
Let $T>0$ and $u=u(t,x)$ be a mild solution to \eqref{HH} satisfying 
\[
u \in C^1((0,T) ; L^2(\mathbb R^d)) \cap C^1((0,T) ; L^{\alpha+1}_{\frac{\gamma}{\alpha+1}}(\mathbb R^d)) \cap C((0,T) ; H^2(\mathbb R^d)).
\]
Assume that $u(t) \to 0$ in $H^1(\mathbb R^d)$ as $t\to 0$. Then $u\equiv 0$ on $[0,T]$. 
\end{lem}

The proof of Lemma \ref{lem:uniqueness-H1} is almost the same as that of \cite[Theorem 2]{HarWei1982}, and so we omit the proof. 
If $\alpha > \alpha_{HS}$ and there exists a positive self-similar solution $\Psi$ satisfying (i)--(iii) in Proposition \ref{nonuniqhiros}, 
then $\Psi$ satisfies all assumptions in Lemma~\ref{lem:uniqueness-H1}, and hence, $\Psi \equiv 0$. This contradicts $\Psi >0$. Thus, we see the nonexistence of such a $\Psi$.

%%%%%%%%%%%%
%%% Section 7 %%%
%%%%%%%%%%%%
\section{Additional results and remarks}\label{sec:8}

%%% Subsection 8.1 %%%

%\subsection{Single critical case II}

\subsection{Double critical case}
We give a remark on the number of solutions in the double critical case.
Theorem \ref{thm:nonuniqueness} shows that the problem \eqref{HH} has two different solutions, where 
one is regular and the other is singular (see Section \ref{sec:5}).
In fact, however, \eqref{HH} has an uncountably infinite number of different mild solutions
in $C([0,T]; L^{q,r}_{s}(\mathbb R^d))$ for any initial data $u_0 \in L^{q,r}_{s}(\mathbb R^d)$. 
This can be confirmed by constructing the family $\{u_{t_0}\}_{t_0\in (0,T)}$ of solutions to \eqref{HH} such that 
$u_{t_0}$ is a singular solution for $0<t \le t_0$ and a regular solution for $t_0<t<T$.

\subsection{Case $\gamma = -\min\{2,d\}$}
The problem on well-posedness for \eqref{HH} in 
the critical singular case $\gamma = -\min\{2,d\}$ has not been studied.
Establishing the weighted linear estimates \eqref{linear-weight1} with the double endpoint $\frac{s_1}{d}+\frac{1}{q_1}=1$ and $\frac{s_2}{d}+\frac{1}{q_2}=0$, we can present the following results on uniqueness for the case $d=1$ and $\gamma =-1$.

\begin{thm}\label{thm:unconditional3}
Let $T>0$, and let $d=1$, $\gamma=-1$, $\alpha>1$, $\alpha \le q < \infty$, and 
$\frac{s}{d} + \frac{1}{q} = 0$. Then the following assertions hold:
\begin{enumerate}[\rm (i)]
\item  Let $0<r \le \alpha-1$. Then unconditional uniqueness holds for \eqref{HH} in $L^\infty(0,T ; L^{q,r}_s(\mathbb R))$. 

\item Let $r> \alpha-1$ and $u_0\in L^{q,r}_{s}(\mathbb R)$. Then,
if $u_1,u_2 \in L^\infty(0,T ; L^{q, r}_{s}(\mathbb R))$ 
are mild solutions to \eqref{HH} with $u_1(0)=u_2(0)=u_0$ such that 
\[
u_i(t) - e^{t\Delta} u_0 \in L^\infty(0,T ; L^{q, \alpha-1}_{s}(\mathbb R))
\quad \text{for }i=1,2,
\]
then $u_1 = u_2$ on $[0,T]$.
\end{enumerate}
\end{thm}

\begin{proof}
The proofs of (i) and (ii) are similar to those of Theorem \ref{thm:unconditional1} (2) and Proposition \ref{prop:uniqueness-sufficient}, respectively. 
The only difference is use of Proposition \ref{prop:linear-main} with the double endpoint case \eqref{linear-condi3} and \eqref{linear-condi4}, where the restriction on $r$ is required. 
\end{proof}

In the case $d=1$ and $\gamma =-1$, the existence of a solution has not been proved, but Theorem \ref{thm:unconditional3} implies that only one solution exists at most.
It remains open whether unconditional uniqueness holds in the critical singular case $d\ge2$ and $\gamma =-2$. 
Once the weighted Meyer inequality \eqref{eq.wMi} with the endpoint case $\frac{s_2}{d}+\frac{1}{q_2}=0$ is proved, 
this problem is solved, but we do not know if the endpoint inequality holds.

\subsection{Case of the exterior problem}
It is also interesting to analyze in more detail the influence of the potential $|x|^\gamma$ at the origin or at infinity.
For this, we discuss unconditional uniqueness for the initial-boundary value problem of the Hardy-H\'enon parabolic equation on 
the exterior domain $\Omega :=\{x \in \mathbb R^d\,;\, |x| > 1\}$.
\begin{equation}\label{HH-D}
	\begin{cases}
		\partial_t u - \Delta u = |x|^{\gamma} |u|^{\alpha-1} u,
			&(t,x)\in (0,T)\times \Omega, \\
			u=0, &(t,x)\in (0,T)\times \partial\Omega,\\
		u(0) = u_0 \in L^{q,r}_{s}(\Omega),
	\end{cases}
\end{equation}
where $T>0,$ $d\in \mathbb{N}$, $\gamma \in \mathbb R$, $\alpha>1$, $q\in [1,\infty]$, $r \in (0,\infty]$ and $s\in \mathbb{R}$.
Here, $\partial \Omega$ denotes the boundary of $\Omega$.
In conclusion, the critical exponents \eqref{q_c} and \eqref{Q_c} with $\gamma=0$ (i.e. $q_c(0)=\frac{d(\alpha-1)}{2}$ and $Q_c(0) = \alpha$) appear in 
the results on unconditional uniqueness for \eqref{HH-D}, since the effect near the origin $x=0$ has been eliminated. 
The results of this subsection can be extended to more general situations such as the initial-boundary value problem on general domains $\Omega$ not containing the origin with the Robin boundary condition
(cf. \cite[Section 5]{ITW2021arxiv}). 
\medskip

In the following, we shall prove the result on unconditional uniqueness. 

\begin{prop}\label{prop:exterior-uniq}
Let $d\in \mathbb N$, $\gamma \in \mathbb R$, $\alpha>1$, $q \in [1,\infty]$ and $s\in \mathbb R$ be such that 
\begin{equation}\label{condi-exterior}
\alpha \le q \le \infty,\quad 
-\frac{d}{q} < s < d\left(1-\frac{\alpha}{q}\right)
\quad \text{and}\quad \frac{\gamma}{\alpha-1} \le s.
\end{equation}
Then the following assertions hold:
\begin{enumerate}[\rm (i)]
\item Assume either 
\begin{equation}\label{double-sub-ex}
q> \min \left\{q_c(0), Q_c(0)\right\}\quad \text{and}\quad r=\infty
\end{equation}
or 
\[
q = Q_c(0) > q_c(0)
\quad \text{and}\quad
r=\alpha.
\] 
Then unconditional uniqueness holds for \eqref{HH-D} in $L^\infty(0,T; L^{q,r}_s(\Omega))$.

\item Assume either 
\[
q = q_c(0) > Q_c(0)
\quad \text{and}\quad
r=\infty.
\] 
or 
\[
q = q_c(0) = Q_c(0)
\quad \text{and}\quad
r=\alpha-1.
\] 
Then unconditional uniqueness holds for \eqref{HH-D} in $C([0,T]; L^{q,r}_s(\Omega))$.
\end{enumerate}
\end{prop}

\begin{rem}
Since $L^{q,r}_{s_1}(\Omega) \subset L^{q,r}_{s_2}(\Omega)$ if $s_2\le s_1$, 
the exponent $s$ should be taken as close to $\max \{-\frac{d}{q}, \frac{\gamma}{\alpha-1}\}$ as possible in the above proposition 
from the point of view of unconditional uniqueness. 
\end{rem}

We denote by $-\Delta_D$ the Laplace operator with the homogeneous Dirichlet boundary condition on $\Omega$ and 
by $\{e^{t\Delta_D}\}_{t>0}$ the semigroup generated by $-\Delta_D$. The integral kernel $G_D(t,x,y)$ of $e^{t\Delta_D}$ satisfies the Gaussian upper bound 
\begin{equation}\label{Gaussian-upper}
0\le G_{D}(t, x,y) \le G_t(x-y)
\end{equation}
for any $t >0$ and almost everywhere $x,y \in \Omega$.
Then, we have the following linear estimates.

\begin{lem}
	\label{lem:linear-main-D}
Let $d \in \mathbb N$, $1 \le q_1\le \infty$, $1<q_2\le \infty$, $0< r_1,r_2 \le \infty$ and $s_1,s_2\in \mathbb R$. 
\begin{enumerate}[\rm (i)]
\item
Assume \eqref{linear-condi1}--\eqref{linear-condi7}. 
Then there exists a constant $C>0$ such that 
\[
\|e^{t\Delta_D} f\|_{L^{q_2,r_2}_{s_2}(\Omega)}
\le C t ^{-\frac{d}{2} (\frac{1}{q_1} - \frac{1}{q_2}) 
- \frac{s_1 - s_2}{2}}
%\min \{ 1, t^{- \frac{s_1 - s_2}{2}}\}
\|f\|_{L^{q_1,r_1}_{s_1} (\Omega)}
\]
for any $t >0$ and $f \in L^{q_1,r_1}_{s_1} (\Omega)$.

\item
Assume \eqref{linear-condi7} and 
%\begin{empheq}[left={\empheqlbrace}]{alignat=2}
\[
\begin{dcases}
   -\frac{d}{q_2} \le s_2 \le \min\left\{ s_1, d\left(1-\frac{1}{q_1}\right)\right\},\\
    q_1 \le q_2, \\
   r_1\le 1 \quad \text{if } s_2 = d\left(1 -\frac{1}{q_1}\right) \text{ or } q_1=1,\\
   r_2=\infty \quad \text{if } s_2 = -\frac{d}{q_2},\\
   r_1 \le r_2 \quad \text{if } q_1=q_2. 
%\end{empheq}
\end{dcases}
\]
Then there exists a constant $C>0$ such that 
\[
\|e^{t\Delta_D} f\|_{L^{q_2,r_2}_{s_2}(\Omega)}
\le C t ^{-\frac{d}{2} (\frac{1}{q_1} - \frac{1}{q_2})}
%\min \{ 1, t^{- \frac{s_1 - s_2}{2}}\}
\|f\|_{L^{q_1,r_1}_{s_1} (\Omega)}
\]
for any $t >0$ and $f \in L^{q_1,r_1}_{s_1} (\Omega)$.
\end{enumerate}
\end{lem}

\begin{proof}
The assertion (i) is obtained by combining the upper bound \eqref{Gaussian-upper} with the argument of proof of Propositions \ref{prop:linear-main}. 
The assertion (ii) is proved by combining the assertion (i) with $s_1=s_2$ and the inclusion $L^{q,r}_{s_1}(\Omega) \subset L^{q,r}_{s_2}(\Omega)$ if $s_2\le s_1$. 
\end{proof}

Similarly, we also have the following:

\begin{lem}\label{l:wMeyer.inq-D}
Let $T\in (0,\infty]$, and let $d\in \mathbb N$, $q_1 \in [1,\infty]$, $q_2\in (1,\infty)$, $r_1 \in (0,\infty]$ and $s_1,s_2 \in \mathbb R$.
\begin{enumerate}[\rm (i)]
\item
Assume \eqref{linear-condi1'}--\eqref{linear-condi6'}.
Then there exists a constant $C>0$ such that 
\begin{equation}\label{Meyer-exterior}
	\left\|\int_0^{t} e^{(t-\tau)\Delta_D} f(\tau) \,d\tau \right\|_{L^{q_2,\infty}_{s_2}(\Omega)} 
	\le C \sup_{0 <\tau <t} \|f(\tau)\|_{L^{q_1,r_1}_{s_1}(\Omega)}
\end{equation}
for any $t\in (0,T)$ and $f \in L^\infty(0,T; L^{q_1,r_1}_{s_1}(\Omega))$.
	
\item
Assume that 
\[
\begin{dcases}
  - \frac{d}{q_2} < s_2 \le \min\left\{ s_1, d\left(1-\frac{1}{q_1}\right)\right\},\\
   \frac{d}{2} \left(\frac{1}{q_1} - \frac1{q_2} \right) = 1,\\
   r_1 \le 1 \quad \text{if }s_2 = d\left( 1-\frac{1}{q_1} \right) \text{ or } q_1=1.
\end{dcases}
\]
Then the estimate \eqref{Meyer-exterior} holds. 
\end{enumerate}
\end{lem}

\begin{proof}[Proof of Proposition \ref{prop:exterior-uniq}]
The proof is simply the same argument as the proofs of Theorems \ref{thm:unconditional1}, \ref{thm:unconditional2} and \ref{thm:uniqueness-criterion}
with Propositions \ref{prop:linear-main} and \ref{l:wMeyer.inq} replaced by Lemmas~\ref{lem:linear-main-D} and \ref{l:wMeyer.inq-D}. 
In fact, in the double subcritical case \eqref{double-sub-ex}, 
we set
$\sigma := \alpha s - \gamma$ and use 
Lemma~\ref{lem:linear-main-D} (ii) with $(q_1,r_1,s_1) = (\frac{q}{\alpha}, \infty, \sigma)$ and $(q_2,r_2,s_2) = (q,\infty,s)$
and Lemma~\ref{lem:Holder} with $(q,r)=(\frac{q}{\alpha},\infty)$, $(q_1,r_1) = (\frac{q}{\alpha-1},\infty)$, 
$(q_2,r_2)=(q,\infty)$ to obtain 
\[
\begin{split}
& \|u_1(t) - u_2(t)\|_{L^{q,\infty}_{s}} \\
& \le C \int_0^t (t-\tau) ^{-\frac{d}{2} (\frac{\alpha}{q} - \frac{1}{q})}
		\| |\cdot|^{\gamma} (|u_1(\tau)|^{\alpha-1}u_1(\tau) - |u_2(\tau)|^{\alpha-1}u_2(\tau)) \|_{L^{\frac{q}{\alpha},\infty}_{\sigma}} \, d\tau \\
& \le 
C \int_0^t (t-\tau) ^{-\frac{d}{2} (\frac{\alpha}{q} - \frac{1}{q})} \, d\tau \times \max_{i=1,2} \|u_i\|_{L^\infty(0,t; L^{q,\infty}_{s})}^{\alpha-1}
		 \| u_1 - u_2\|_{L^\infty(0,t; L^{q,\infty}_{s})}\\
& \le C t^{\delta}
 \max_{i=1,2} \|u_i\|_{L^\infty(0,t; L^{q,\infty}_{s})}^{\alpha-1}
		 \| u_1 - u_2\|_{L^\infty(0,t; L^{q,\infty}_{s})}.
\end{split}
\]
Here, the conditions  
\[
1< \frac{q}{\alpha} \le q \le \infty,\quad 
-\frac{d}{q} < s < d\left(1-\frac{\alpha}{q}\right),
\quad s \le \sigma,
\quad 
\frac{d}{2} \left(\frac{\alpha}{q} - \frac{1}{q}\right)  < 1
\]
are required in order to use Proposition~\ref{prop:linear-main} and for the above integral in $\tau$ to be finite. 
Note that the conditions amount to \eqref{condi-exterior} and \eqref{double-sub-ex}. 
Similarly to the proof of Theorem \ref{thm:unconditional1}, we conclude the assertion (i) in the case \eqref{double-sub-ex}. 
The other cases can be also proved in a similar way, and so we may omit the details. 
The proof of Proposition \ref{prop:exterior-uniq} is finished.
\end{proof}

%%%%%%%%%%%%
%%% Appendix %%%
%%%%%%%%%%%%
\appendix

%%% Appendix A %%%
\section{Some lemmas on weighted Lorentz spaces}\label{app:A}
In this appendix
we provide several lemmas on weighted Lorentz spaces. 
First, we will show the Fatou property of $L^{q,r}_s(\Omega)$. A quasi-normed space $X\subset L^0(\Omega)$ is said to satisfy the Fatou property if the following holds:
Suppose that $f_n \in X$, $f_n\ge 0$ ($n\in \mathbb N$) and $f_n\nearrow f$ a.e. as $n\to \infty$. 
If $f \in X$, then $\|f_n\|_{L^{q,r}_s} \nearrow\|f\|_{L^{q,r}_s}$ as $n\to \infty$, whereas if $f \not \in X$, then $\|f_n\|_{L^{q,r}_s} \nearrow \infty$ as $n\to \infty$.
Then we have the following lemma.

\begin{lem}\label{lem:monotone}
Let $s\in \mathbb R$, $0<q,r\le \infty$ and $r=\infty$ if $q=\infty$. Then $L^{q,r}_s(\Omega)$ satisfies the Fatou property. 
\end{lem}

\begin{rem}\label{rem:quasi-Banach}
From Lemma~\ref{lem:monotone}, we can immediately see that $L^{q,r}_s(\Omega)$ is a quasi-Banach space by using the fact that 
a quasi-normed space $X \subset L^0(\Omega)$ is complete if it satisfies the Fatou property (see e.g. \cite[Remark 2.1 (ii) and Proposition 2.2]{LNarxiv}). 
\end{rem}

\begin{proof}
Suppose that $f_n \in X$, $f_n\ge 0$ ($n\in \mathbb N$) and $f_n\nearrow f$ a.e. as $n\to \infty$. 
Then $0 \le (|x|^s f_n)^* \nearrow (|x|^s f)^*$ as $n\to \infty$ (see \cite[Proposition 1.4.5 (8)]{Gr2008}).  
If $q,r < \infty$, then the monotone convergence theorem (over $(0,\infty)$ with Lebesgue measure) yields 
\[
\|f_n\|_{L^{q,r}_s} 
=
\left( 
\int_0^\infty ( t^{\frac{1}{q}} (|x|^s f_n)^*(t) )^r\frac{dt}{t}
\right)^{\frac{1}{r}}
 \nearrow
\left( 
\int_0^\infty ( t^{\frac{1}{q}} (|x|^s f)^*(t) )^r\frac{dt}{t}
\right)^{\frac{1}{r}}
= \|f\|_{L^{q,r}_s}
\]
as $n\to \infty$. If $q<\infty$ and $r=\infty$, then we have
\[
\lim_{n\to \infty}\|f_n\|_{L^{q,\infty}_s} 
= \sup_{n \in \mathbb N}\sup_{t>0} t^\frac{1}{q} (|x|^sf_n)^*(t)
= \sup_{t>0} \sup_{n \in \mathbb N} t^\frac{1}{q} (|x|^sf_n)^*(t)
=\|f\|_{L^{q,\infty}_s},
\]
since the limit as $n\to\infty$ is the supremum over $n\in \mathbb N$ by the monotone increase of $\{\|f_n\|_{L^{q,\infty}_s}\}_n$. 
The case $q=r=\infty$ is similarly proved as above. The proof of Lemma~\ref{lem:monotone} is finished.
\end{proof}

Next, we will prove the following density result, which is used in Remark~\ref{rem:w-Lorentz} (d) and Lemma~\ref{lem:density}.

 \begin{lem}\label{lem:A-dense}
 Let $\Omega$ be a domain in $\mathbb R^d$, and let $s \in \mathbb R$ and $0<q,r<\infty$. Then the following statements hold:
 \begin{enumerate}[\rm (i)] 
\item
$L^{q,r}_s(\Omega) \cap L^\infty_0(\Omega)$ is dense in $L^{q,r}_s(\Omega)$.

\item If $1<q<\infty$ and $1 \le r <\infty$, then 
$C^\infty_0(\Omega)$ is dense in $L^{q,r}_s(\Omega)$.
\end{enumerate}
 \end{lem}
 
\begin{rem}\label{rem:A_dense}
Let us give some remarks on Lemma \ref{lem:A-dense}.
\begin{enumerate}[\rm (a)] 
\item
There are many results on density for weighted spaces such as weighted Lebesgue space and weighted Banach function spaces (see e.g. \cite[Theorems 1.1 and 1.2]{NTY2004},
\cite[Lemmas 2.4, 2.10 and 2.12]{KS2014} and \cite[Proposition 3.1 and Remark 3.2]{ELM2021}). 
However, applying the previous results to our weighted Lorentz spaces requires restrictions on the parameters $s,q,r$ in the density result. 
In particular, the condition $0<\frac{s}{d}+\frac{1}{q}<1$ is imposed on $s$.
On the other hand, we note that Lemma \ref{lem:A-dense} requires no additional conditions on $s,q,r$.

\item The case $r=q$ of Lemma \ref{lem:A-dense} (i) can be found in 
\cite[Lemma 2.12]{LN2022}.

\item 
A similar density result to Lemma \ref{lem:A-dense} (i) is proved in the first paragraph of the proof of 
\cite[Proposition 3.13]{LNarxiv}.

\item For $0<q \le 1$ or $0<r<1$, the density of $C^\infty_0(\Omega)$ in $L^{q,r}_s(\Omega)$ is not proved. In fact, 
the approximate functions $f_n$ in \eqref{f_n_approx} (which approximate the target function $f$ for $0<q \le 1$) belong to $L^{q,r}_s(\Omega) \cap L^\infty_0(\Omega)$, but not to 
$C^\infty_0(\Omega)$. 
\end{enumerate}
\end{rem}

\begin{proof}
We divide the proof into three steps.

\smallskip 

\noindent{\it Step 1:} In this step 
we show that $L^{q,r}_s(\Omega)$ has an absolutely continuous quasi-norm, i.e., for any function $f\in L^{q,r}_s(\Omega)$, 
\begin{equation}\label{pf:lem_A3_step1}
\|f \chi_{E_n} \|_{L^{q,r}_s(\Omega)} \to 0
\quad \text{as }n\to \infty
\end{equation}
holds 
for any sequence $\{E_n\}_n$ of measurable subsets of $\Omega$ such that $\chi_{E_n} \to 0$ a.e. as $n\to \infty$.
See e.g. \cite[Chapter 1]{BS1988} , \cite[Section 2]{ELM2021} and \cite[Section 2]{KS2014} for the details of absolutely continuous (quasi-)norm.

If \eqref{pf:lem_A3_step1} is shown for any non-negative function in $L^{q,r}_s(\Omega)$, then 
\eqref{pf:lem_A3_step1} is also shown for all functions $f\in L^{q,r}_s(\Omega)$
by the decomposition $f = f_+ - f_-$ with the positive part $f_+\ge0$ and negative part $f_- \ge0$ of $f$.
Hence, we may assume that $f\in L^{q,r}_s(\Omega)$ is non-negative on $\Omega$ without loss of generality. 
Let $g_n := |x|^s f \chi_{E_n}$. Then
\[
\|f \chi_{E_n} \|_{L^{q,r}_s(\Omega)}
=
\|g_n\|_{L^{q,r}(\Omega)}
= q^{\frac1r}
\left(
\int_0^\infty \left(
d_{g_n} (\lambda)^\frac{1}{q} \lambda
\right)^r \frac{d\lambda}{\lambda}
\right)^\frac{1}{r},
\]
where we recall 
\[
d_{g_n} (\lambda) = \left|\{ x \in \Omega \, ;\, |g_n(x)| >\lambda\}\right|
=
\int_\Omega \chi_{\{|g_n| > \lambda\}}(x) \, dx.
\]
Since $\chi_{E_n} \to 0$ a.e. as $n\to \infty$, we see that 
$\chi_{\{|g_n| > \lambda\}} \to 0$ a.e. as $n\to \infty$. 
There is a dominating function of $\chi_{\{|g_n| > \lambda\}}$, i.e., $\chi_{\{|g_n| > \lambda\}} \le \chi_{\{|x|^s f > \lambda\}}$ and $ \chi_{\{|x|^s f > \lambda\}} \in L^1(\Omega)$ for any $n \in \mathbb N$. 
Hence, it follows from Lebesgue's dominated convergence theorem that $d_{g_n} (\lambda) \to 0$ for any $\lambda>0$. 
Furthermore, since $d_{g_n} (\lambda) \le d_{|x|^s f}(\lambda)$ for any $\lambda>0$, we can again apply Lebesgue's dominated convergence theorem to obtain 
\[
\|g_n\|_{L^{q,r}(\Omega)}
= q^{\frac1r}
\left(
\int_0^\infty \left(
d_{g_n} (\lambda)^\frac{1}{q} \lambda
\right)^r \frac{d\lambda}{\lambda}
\right)^\frac{1}{r} \to 0\quad \text{as }n\to \infty,
\]
which implies \eqref{pf:lem_A3_step1}.

\smallskip 

\noindent{\it Step 2:} In this step 
we prove the case $s=0$ of (i) and (ii).

If $1<q<\infty$ and $1\le r < \infty$, then $L^{q,r}(\Omega)$ is a Banach function space\footnote{There are several different definitions of Banach function space. In this part, we use it in the same sense as in Bennett and Sharpley's book \cite{BS1988}.}, and moreover, it is already shown in Step 1 that $L^{q,r}(\Omega)$ has an absolutely continuous norm.
Hence, it follows from \cite[Proposition 3.1 and Remark 3.2]{ELM2021} that $C_0^\infty(\Omega)$ is dense in $L^{q,r}(\Omega)$. 
Thus, the case $s=0$ of (ii) is proved.

The cases $0<q\le 1$ or $0<r<1$ can be obtained by the argument of proof of \cite[Lemma 2.12]{LN2022}. In fact, 
we take a constant $m \in \mathbb N$ such that $2^m q >1$ and $2^m r \ge 1$ and define $g := |f|^{2^{-m}}$. Then 
$g \in L^{2^mq,2^mr}(\Omega)$, since
\[
\|g\|_{L^{2^mq,2^mr}(\Omega)}^{2^m}
=
\left\||g|^{2^m}\right\|_{L^{q,r}(\Omega)}
=
\|f\|_{L^{q,r}(\Omega)} < \infty
\]
(see \cite[Remark 1.4.7]{Gr2008} for the first equality). 
Hence, there exists a sequence $\{g_n\}_n \subset C^\infty_0(\Omega)$ such that 
\begin{equation}\label{step:convergence}
\| g - g_n \|_{L^{2^mq,2^mr}(\Omega)} \to 0 \quad \text{as }n\to \infty.
\end{equation}
Define 
\begin{equation}\label{f_n_approx}
f_n := g_n^{2^m} \mathrm{sgn}\, (f)
\end{equation}
for $n \in \mathbb N$, 
where $\mathrm{sgn}\, (f)$ denotes the sign function of $f$.
Then $f_n  \in L^{q,r}(\Omega) \cap L^\infty_0(\Omega)$ and 
\[
|f - f_n| = \left| g^{2^m} - g_n^{2^m} \right|
=
| g - g_n| \prod_{\ell=0}^{m-1} \left| g^{2^\ell} + g_n^{2^\ell} \right|.
\]
By Lemma \ref{lem:Holder} (i), we estimate
\begin{equation}\label{step2:inequality}
\begin{split}
\|f - f_n\|_{L^{q,r}(\Omega)}  
& =
\left\|
| g - g_n| \prod_{\ell=0}^{m-1} \left| g^{2^\ell} + g_n^{2^\ell} \right|
\right\|_{L^{q,r}(\Omega)} \\
& \le C
\|g -g_n\|_{L^{2^mq,2^mr}(\Omega)} 
 \prod_{\ell=0}^{m-1}
 \|g^{2^\ell} + g_n^{2^\ell}\|_{L^{2^{m-\ell}q,2^{m-\ell}r}(\Omega)} .
\end{split}
\end{equation}
Here, $ \prod_{\ell=0}^{m-1} \|g^{2^\ell} + g_n^{2^\ell}\|_{L^{2^{m-\ell}q,2^{m-\ell}r}(\Omega)}$ is bounded in sufficiently large $n$ by combining the convergence \eqref{step:convergence} and the inequality 
\[
\begin{split}
 \|g^{2^\ell} + g_n^{2^\ell}\|_{L^{2^{m-\ell}q,2^{m-\ell}r}(\Omega)}
& \le \|g^{2^\ell} \|_{L^{2^{m-\ell}q,2^{m-\ell}r}(\Omega)} + \|g^{2^\ell}_n \|_{L^{2^{m-\ell}q,2^{m-\ell}r}(\Omega)}\\
& =  \|g \|_{L^{2^{m}q,2^{m}r}(\Omega)} + \|g_n \|_{L^{2^{m}q,2^{m}r}(\Omega)}.
 \end{split}
\]
Therefore, it follows from \eqref{step:convergence} and \eqref{step2:inequality} that $\|f - f_n\|_{L^{q,r}(\Omega)} \to 0$ as $n\to \infty$.
Thus, the density in the case $0<q\le 1$ is proved, and hence, the case $s=0$ of (i) is also proved.

\smallskip 

\noindent{\it Final step:} 
Finally, we prove the case $s\not =0$ of (i) and (ii). 
Let $f \in L^{q,r}_s(\Omega)$. We decompose $f$ into
\[
f = f \chi_{\{|x|\le \delta\}} + f \chi_{\{\delta <|x| < R\}} + f \chi_{\{|x|\ge R\}} 
\]
for $0 < \delta < R$. Then $\chi_{\{|x|\le \delta\}}$ and $\chi_{\{|x|\ge R\}}$ converge to $0$ almost everywhere in $\Omega$ as $\delta \to 0$ and $R\to \infty$, respectively.
Hence it follows from \eqref{pf:lem_A3_step1} that, for any $\varepsilon >0$, there exist $\delta, R>0$ such that 
\begin{equation}\label{A3:dense1}
\|f \chi_{\{|x|\le \delta\}}\|_{L^{q,r}_s(\Omega)} + \|f \chi_{\{|x|\ge R\}}\|_{L^{q,r}_s(\Omega)} < \varepsilon.
\end{equation}
Furthermore, 
since $L^{q,r}_s(\{\delta <|x| < R\}) = L^{q,r}(\{\delta <|x| < R\})$, it follows from 
the density of 
$L^{q,r}_s(\{\delta <|x| < R\}) \cap L^\infty_0(\{\delta <|x| < R\})$ in $L^{q,r}(\{\delta <|x| < R\})$ which is shown in Step 2, that 
there exists a function $g = g_{\varepsilon, \delta, R} \in L^{q,r}_s(\{\delta <|x| < R\}) \cap L^\infty_0(\{\delta <|x| < R\})$ satisfying
\[
\| f - g\|_{L^{q,r}(\{\delta <|x| < R\})} < \frac{\varepsilon}{\max\{\delta^s, R^s\}}.
\]
This implies that 
\begin{equation}\label{A3:dense2}
\| f - g\|_{L^{q,r}_s(\{\delta <|x| < R\})} \le \max\{\delta^s, R^s\} \| f - g\|_{L^{q,r}(\{\delta <|x| < R\})} < \varepsilon.
\end{equation}
Let $\tilde{g}$ be the zero extension of $g$ to $\Omega$. Then $\tilde{g} \in L^{q,r}_s(\Omega) \cap L^\infty_0(\Omega)$ and 
\[
\begin{split}
& \| f - \tilde{g}\|_{L^{q,r}_s(\Omega)} \\
& \le \| f \|_{L^{q,r}_s(\{|x| \le \delta\})} 
+
\| f - g\|_{L^{q,r}_s(\{\delta <|x| < R\})} 
+
\| f \|_{L^{q,r}_s(\{|x| \ge R\})}\\
&  < 2\varepsilon
\end{split}
\]
by \eqref{A3:dense1} and \eqref{A3:dense2}.
Therefore, the case $s\not=0$ of (i) is proved. For (ii), we also prove the density in the same way. Only difference is to be able to take an approximate function $g \in C^\infty_0(\{\delta <|x| < R\})$ and the zero extension $\tilde{g} \in C^\infty_0(\Omega)$. 
Thus the proof of Lemma \ref{lem:A-dense} is finished.
\end{proof}

Next, we will show the following lemma on behavior of $f \in L^{q,r}_s(\mathbb R^d)$ as $|x|\to0$ or $|x|\to \infty$, 
which is used in Step 2 of the proof of the necessity part of Proposition \ref{prop:linear-main}.

% Lemma A.3
\begin{lem}\label{lem:A-cha}
Let $0<q,r<\infty$ and $s\in \mathbb R$. 
If $f \in L^{q,r}_s(\mathbb R^d)$, then 
\[
\liminf_{|x| \to 0} |x|^{s+\frac{d}{q}} |\log |x||^{\frac{1}{r}} |f(x)| =
\liminf_{|x| \to \infty} |x|^{s+\frac{d}{q}} |\log |x||^{\frac{1}{r}} |f(x)|=0.
\]
\end{lem}

\begin{rem}
The case of weighted Lebesgue space $L^q_s(\mathbb R^d)$ can be found in \cite[Corollary A.4]{CIT2022} for instance. 
\end{rem}

\begin{proof}
For simplicity, we give the proof only for the case $d=1$ and $s=0$. 
Suppose that 
\[
\liminf_{|x| \to 0} |x|^{\frac{1}{q}} |\log |x||^{\frac{1}{r}} |f(x)|=c>0.
\]
Then there exists a positive constant $\delta$ such that $c/2 \le |x|^{\frac{1}{q}} |\log |x||^{\frac{1}{r}} |f(x)|$
for $|x|\le \delta$. Using \cite[Proposition 1.4.5 (4) and (5)]{Gr2008}, 
we have
\[
(f \chi_{|x|\le \delta})^*(t)
\ge \frac{c}{2} (|x|^{-\frac{1}{q}} |\log |x||^{-\frac{1}{r}} \chi_{\{|x|\le \delta\}})^*(t)
= \frac{c}{2} t^{-\frac{1}{q}} |\log t|^{-\frac{1}{r}} \chi_{\{t\le \delta'\}}
\]
for some $\delta'>0$. Hence,
\[
\begin{split}
\|f \chi_{|x|\le \delta} \|_{L^{q,r}}
& =
\left( 
\int_0^\infty ( t^{\frac{1}{q}} (f \chi_{|x|\le \delta})^*(t) )^r\frac{dt}{t}
\right)^{\frac{1}{r}}\\
& \ge \frac{c}{2}
\left( 
\int_0^{\delta'} 
 |\log t|^{-1} \frac{dt}{t}
\right)^{\frac{1}{r}} = +\infty,
\end{split}
\]
which implies $f \not \in L^{q,r}(\mathbb R)$. 
The second equality is similarly proved.
\end{proof}

Finally, we will prove
the continuity of heat semigroup at $t=0$,
which is used to prove the continuity of mild solutions at $t=0$ in Proposition \ref{prop:regular1} and 
Lemma \ref{lem:pertubed}. 

\begin{lem}\label{lem:A_conti}
Let $s\in \mathbb R$, $1<q \le \infty$ and $0<r \le \infty$ satisfy
\[
\begin{dcases}
   0 \le \frac{s}{d} + \frac{1}{q} \le 1, \\
  r \le 1\quad \text{if }\frac{s}{d} + \frac{1}{q}=1,\\
  r = \infty\quad \text{if } \frac{s}{d} + \frac{1}{q}=0 \text{ or } q=\infty.
\end{dcases}
\]
Then 
\[
\lim_{t\to 0}\|e^{t\Delta} f - f\|_{L^{q,r}_s} =0
\]
holds for any $f \in L^{q,r}_s(\mathbb R^d)$ 
(replace $L^{q, r}_{s}(\mathbb R^d)$ by $\mathcal{L}^{q, r}_{s}(\mathbb{R}^d)$ if $q=\infty$ or $r=\infty$).
\end{lem}

\begin{proof}
The case $r=q$ follows from the standard argument (see e.g. \cite[Theorem 5.5 on page 198]{SS2004}). 
The case $r\not =q$ can be proved by a real interpolation argument. 
In fact, it is known that $L^{q,r}(\mathbb R^d)$ coincides with the real interpolation space $(L^{q_0}(\mathbb R^d), L^{q_1}(\mathbb R^d))_{\theta,r}$, where $0<\theta<1$ and $1<q_0<q<q_1<\infty$ 
(see e.g. \cite[5.3.1 Theorem on page 113]{BL1976}). 
This implies that for any $g \in L^{q,r}(\mathbb R^d)$, there exist $g_i \in L^{q_i}(\mathbb R^d)$ ($i=0,1$) satisfying 
$g=g_0+g_1$ and 
\[
\|g\|_{L^{q,r}}\le 
\left(\int_0^\infty \left( \lambda^\theta (\|g_0\|_{L^{q_0}} + \lambda^{-1}\| g_1\|_{L^{q_1}})
\right)^r\frac{d\lambda}{\lambda}
\right)^\frac1r
\le 2 \|g\|_{L^{q,r}}.
\]
Let $f \in L^{q,r}_s$. Then, taking $g = |x|^s f$, and applying the above interpolation result to this $g$, we have two functions $g_i$ ($i=0,1$) as above.
Defining $f_i$ by $f_i := |x|^{-s} g_i$ for $i=0,1$, we see that $f_i \in L^{q_i}_s(\mathbb R^d)$ for $i=0,1$, 
the decomposition $f = f_0 + f_1$ and the inequalities 
\begin{equation}\label{A_4:dominating}
\|f\|_{L^{q,r}_s}\le 
\left(\int_0^\infty \left( \lambda^\theta (\|f_0\|_{L^{q_0}_s} + \lambda^{-1}\| f_1\|_{L^{q_1}_s})
\right)^r\frac{d\lambda}{\lambda}
\right)^\frac1r
\le 2 \|f\|_{L^{q,r}_s}.
\end{equation}
Now, we have 
\[
\|e^{t\Delta} f - f\|_{L^{q,r}_s}
\le 
\left(\int_0^\infty \left( \lambda^\theta (\|e^{t\Delta} f_0 - f_0\|_{L^{q_0}_s} + \lambda^{-1}\|e^{t\Delta} f_1 - f_1\|_{L^{q_1}_s})
\right)^r\frac{d\lambda}{\lambda}
\right)^\frac1r.
\]
The integrand in the right-hand side converges to $0$ almost everywhere in $\lambda \in (0,\infty)$ as $t\to 0$ 
by Lemma \ref{lem:A_conti} with $r=q$ (which is already proved). 
Moreover, we see that the integrand
has a dominating function by a combination of \eqref{A_4:dominating} and the inequality 
\[
\|e^{t\Delta} f_0 - f_0\|_{L^{q_0}_s} + \lambda^{-1}\|e^{t\Delta} f_1 - f_1\|_{L^{q_1}_s}
\le C ( \|f_0\|_{L^{q_0}_s} + \lambda^{-1}\| f_1\|_{L^{q_1}_s}),
\]
where we used the triangle inequality and boundedness of $e^{t\Delta}$ on $L^{q,r}_s(\mathbb R^d)$ (Proposition \ref{prop:linear-main}).
Therefore, we can use Lebesgue's dominated convergence theorem to obtain 
\[
\lim_{t\to 0}\|e^{t\Delta} f - f\|_{L^{q,r}_s} =0.
\]
Thus, the proof of Lemma \ref{lem:A_conti} is finished.
\end{proof}

%%% Appendix B %%%
\section{Proof of Theorem \ref{thm:sss_sharp} {\rm (ii)}}\label{app:B}

In this appendix, we give a proof of Theorem \ref{thm:sss_sharp} (ii) for completeness. 
The proof is based on the argument of the proof of \cite[Theorem~4.1]{GueVer1988}. 
For simplicity, we write $U = U(r)$, where $r=|x|$.
Then
$U$ satisfies the problem
\[
-(r^{d-1} U')' = r^{d-1+\gamma}U^{\frac{d + \gamma}{d-2}} ,\quad r \in (0,1).
\]
Then the upper bound of $U$ near $x=0$ was already obtained.

% Theorem 5.3
\begin{thm}[Theorem 1.1 (iv) in \cite{BidGar2001}]\label{thm:upperbound}
There exists a constant $C>0$ such that 
\[
U(r) \le C r^{-(d-2)} |\log r|^{-\frac{d-2}{\gamma+2}},\quad r \in (0,1).
\]
\end{thm}

We make the change of variable
\begin{equation}\label{def:mathfrak_u}
U(r) = r^{-(d-2)} \mathfrak{u}(t),\quad t =  -\log r. 
\end{equation}
The properties of $\mathfrak{u}$ are as follows.

% Lemma 5.4
\begin{lem}\label{lem:B1}
The function $\mathfrak{u}$ is of $C^2$ on $(0,\infty)$ and is a positive and strictly decreasing solution of the nonlinear ordinary differential equation
\begin{equation}\label{ODE-v}
\frac{d}{dt}\left( \frac{d\mathfrak{u}}{dt}(t)+ (d-2) \mathfrak{u}(t) \right) + \mathfrak{u}(t)^{\frac{d + \gamma}{d-2}} = 0,
\quad t \in (0,\infty)
\end{equation}
with
\[
\mathfrak{u}(0) = \lim_{r\to 1}U(r)\quad \text{and}\quad \frac{d\mathfrak{u}}{dt}(0) 
= - \lim_{r\to 1}\left(\frac{dU}{dr}(r) -(d-2) U(r)\right)
\]
(and hence, $\mathfrak{u}$ is a $C^1$-diffeomorphism from $(0,\infty)$ to $(0,\mathfrak{u}(0))$).
Moreover, $\mathfrak{u}^{\frac{d + \gamma}{d-2}} \in L^1((0, \infty))$ and 
\begin{equation}\label{positive-v}
\frac{d\mathfrak{u}}{dt}(t)+ (d-2) \mathfrak{u}(t) > 0,
\quad t \in (0,\infty). 
\end{equation}
\end{lem}

\begin{proof}
It is obvious that $\mathfrak{u}$ is positive and of $C^2$, and
a straightforward calculation gives that $\mathfrak{u}$ satisfies the nonlinear ordinary differential equation \eqref{ODE-v}. 
It is shown by Theorem \ref{thm:upperbound} that $\mathfrak{u}^{\frac{d + \gamma}{d-2}} \in L^1((0, \infty))$. 

We shall prove that $\mathfrak{u}$ is strictly decreasing on $(0,\infty)$ by contradiction.
Suppose that $\mathfrak{u}$ is not strictly decreasing on $(0,\infty)$. Then 
there exist $t_0, t_1$ such that $0<t_0<t_1$ and 
\begin{equation}\label{eq.con1}
\mathfrak{u}_t(t_0)=\mathfrak{u}_t(t_1) = 0\quad \text{and}\quad \mathfrak{u}_t \ge0 \text{ on } (t_0,t_1).
\end{equation}
Since $\mathfrak{u}$ is positive,  
we find from \eqref{eq.con1} that 
\[
\begin{split}
(d-2) \left\{\mathfrak{u}(t_1) - \mathfrak{u}(t_0)\right\}
& = \left[\mathfrak{u}_t(\tau)+ (d-2) \mathfrak{u}(\tau)\right]_{\tau=t_0}^{\tau = t_1}\\
& = - \int_{t_0}^{t_1} \mathfrak{u}(\tau)^{\frac{d + \gamma}{d-2}}\, d\tau < 0,
\end{split}
\]
which implies that $\mathfrak{u}(t_0) >\mathfrak{u}(t_1)$. This is a contradiction to $\mathfrak{u}_t \ge0$ on $(t_0,t_1)$. 
Therefore, $\mathfrak{u}$ is strictly decreasing on $(0,\infty)$. 
In addition, it is also shown by 
the inverse function theorem that $\mathfrak{u}$ is a $C^1$-diffeomorphism from $(0,\infty)$ to $(0,\mathfrak{u}(0))$. 

Lastly, 
since $\mathfrak{u} \in C^2((0,\infty))$, the fundamental theorem of calculus gives 
\[
\mathfrak{u}(t') - \mathfrak{u}(t) = \int_t^{t'} \mathfrak{u}_t(\tau)\, d\tau
\]
for $t'\ge t > 0$, and as $t' \to \infty$, 
\[
\mathfrak{u}(t) =  - \int_t^{\infty} \mathfrak{u}_t(\tau)\, d\tau
\]
for $t>0$. Since $\mathfrak{u}_t<0$, the convergence $\mathfrak{u}_t(t) \to 0$ as $t\to\infty$ must hold.
Noting $\mathfrak{u}(t), \mathfrak{u}_t(t) \to 0$ as $t\to\infty$, and integrating \eqref{ODE-v} over $[t,\infty)$, we have
\begin{equation}\label{IE-v}
\mathfrak{u}_t(t)+ (d-2) \mathfrak{u}(t)
=
 \int_{t}^{\infty} \mathfrak{u}(\tau)^{\frac{d + \gamma}{d-2}}\, d\tau
\end{equation}
for any $t>0$, which implies \eqref{positive-v}. 
The proof of Lemma \ref{lem:B1} is finished.
\end{proof}

% Lemma 5.5
\begin{lem}\label{lem:sec5_key}
Let $d\ge3$ and $\gamma >-2$. 
Assume that 
\begin{equation}\label{goal_sec5}
\lim_{t\to \infty}\frac{\mathfrak{u}_t(t)}{\mathfrak{u}(t)} = 0 \text{ or }-(d-2). 
\end{equation}
Then the assertion {\rm (ii)} in Theorem \ref{thm:sss_sharp} holds. 
\end{lem}

\begin{proof}
In the case where 
\[
\lim_{t\to \infty}\frac{\mathfrak{u}_t(t)}{\mathfrak{u}(t)} = -(d-2)
\quad \left(
\text{i.e. }\lim_{t\to \infty} (\log \mathfrak{u}(t))_t =-(d-2)
\right),
\]
then for any $\varepsilon\in (0,d-2)$, 
there exists $T=T(\varepsilon)>0$ such that 
\begin{equation}\label{eq.5-1}
-(d-2)-\varepsilon 
<(\log \mathfrak{u}(t))_t < 
-(d-2)+\varepsilon
\end{equation}
for any $t\ge T$. 
By integrating \eqref{eq.5-1} over $[T, t]$, we estimate
\[
 \mathfrak{u}(t) < \mathfrak{u}(T)e^{(-(d-2)+\varepsilon)(t-T)} < \mathfrak{u}(0)e^{(-(d-2)+\varepsilon)(t-T)},
\]
 and by recalling \eqref{def:mathfrak_u}, we find that 
\[
U(r) \le C
e^{((d-2)- \varepsilon)T}
 r^{-\varepsilon},\quad r\in (0,1). 
\]
Hence, $U$ can be extended as a $C^1$ function on $B$ (see \cite[Theorem 1]{Ser1965} and also \cite[Lemma 2.1 and Section 3]{DZarxiv}). 

Next, we consider the other case: 
\begin{equation}\label{eq.singular}
\lim_{t\to \infty}\frac{\mathfrak{u}_t(t)}{\mathfrak{u}(t)} = 0.
\end{equation}
Set 
\[
\psi (t) := \int_t^\infty \mathfrak{u}(\tau)^{\frac{d + \gamma}{d-2}} d\,\tau. 
\]
Then the following hold: 
\begin{equation}\label{eq.ss_3}
\lim_{t\to\infty} \frac{\psi(t)^{\frac{d + \gamma}{d-2}}}{\psi_t(t)} = - (d-2)^{\frac{d + \gamma}{d-2}},
\end{equation}
and 
\begin{equation}\label{eq.ss_4}
\lim_{t\to\infty} 
t^{\frac{d-2}{2+\gamma}}
\psi(t)
=
(2+\gamma)^{-\frac{d-2}{2+\gamma}}
(d-2)^{\frac{2d + \gamma-2}{2+\gamma}}.
\end{equation}
In fact, noting that $\lim_{t\to \infty} \mathfrak{u} (t) =\lim_{t\to \infty} \mathfrak{u}_t (t) =0$,  
we see from 
\eqref{ODE-v} that 
\[
\psi (t) = \mathfrak{u}_t(t)+ (d-2) \mathfrak{u}(t)
\quad \text{and}\quad 
\psi_t(t) = 
-\mathfrak{u}(t)^{\frac{d+\gamma}{d-2}}.
\]
Hence,
\[
\frac{\psi(t)^{\frac{d + \gamma}{d-2}}}{\psi_t(t)}
=
\frac{(\mathfrak{u}_t(t)+ (d-2) \mathfrak{u}(t))^{\frac{d + \gamma}{d-2}}}{-\mathfrak{u}(t)^{\frac{d+\gamma}{d-2}}}
 = -\left(\frac{\mathfrak{u}_t(t)}{\mathfrak{u}(t)}+ (d-2) \right)^{\frac{d + \gamma}{d-2}}.
\]
This and \eqref{eq.singular} imply \eqref{eq.ss_3}. 
Moreover, we see from \eqref{eq.ss_3} that 
\[
\lim_{t\to \infty}(\psi^{-\frac{2+\gamma}{d-2}}(t))_t
=
-\frac{2+\gamma}{d-2} 
\lim_{t\to \infty}\frac{\psi_t(t)}{\psi(t)^{\frac{d + \gamma}{d-2}}}
=
(2+\gamma)(d-2)^{-\frac{d + \gamma}{d-2}-1}.
\]
Integrating the above yields 
\[
\lim_{t\to\infty} 
t^{-1}\psi^{-\frac{2+\gamma}{d-2}}(t)
=
(2+\gamma)(d-2)^{-\frac{d + \gamma}{d-2}-1},
\]
which implies \eqref{eq.ss_4}.
By using \eqref{eq.ss_3} and \eqref{eq.ss_4} and noting that $\mathfrak{u}(t) = (-\psi_t(t))^{\frac{d-2}{d+\gamma}}$, 
we obtain
\[
\begin{split}
\lim_{t\to\infty} t^{\frac{d-2}{2+\gamma}}\mathfrak{u}(t)
= \lim_{t\to\infty} t^{\frac{d-2}{2+\gamma}}(-\psi_t(t))^{\frac{d-2}{d+\gamma}}
 =
\lim_{t\to\infty} t^{\frac{d-2}{2+\gamma}}
\left(
\frac{\psi(t)}{d-2}
\right)
 =
\left(
\frac{(d-2)^2}{2+\gamma}
\right)^{\frac{d-2}{2+\gamma}}.
\end{split}
\]
Thus, we conclude Lemma \ref{lem:sec5_key}.
\end{proof}

Finally, we conclude the proof of (ii) of Theorem \ref{thm:sss_sharp} by showing the following.

% Lemma 5.6
\begin{lem}\label{lem:key_sec5_2}
Let $d\ge3$ and $\gamma >-2$. Then 
\eqref{goal_sec5} holds.
\end{lem}

\begin{proof}
Since $\mathfrak{u}$ is a $C^1$-diffeomorphism from $(0,\infty)$ to $(0,\mathfrak{u}(0))$ by Lemma \ref{lem:B1}, 
we can define
\[
\rho = \mathfrak{u}(t)\quad \text{and}\quad \mathfrak{v}(\rho) = \mathfrak{u}_t (t)
\quad \left(\text{i.e. }\mathfrak{v}(\rho) = \mathfrak{u}_t(\mathfrak{u}^{-1}(\rho))\right).
\]
For convenience, we set 
\[
\mathfrak{w}(\rho) := \frac{\mathfrak{v}(\rho)}{\rho}.
\]
Then our goal is to prove that 
\begin{equation}\label{eq.goal_lem5.6}
\lim_{\rho \to +0}\mathfrak{w}(\rho) = 0 \text{ or } - (d-2).
\end{equation}

First, we will show there exists a limit of $\mathfrak{w}$ as $\rho \to +0$ such that 
\begin{equation}\label{eq.goal_lem5.6_1}
\lim_{\rho\to +0} \mathfrak{w}(\rho) = m \in [-(d-2),0].
\end{equation}
Since $\mathfrak{w}$ is continuous and $-(d-2) < \mathfrak{w} < 0$, \eqref{eq.goal_lem5.6_1} is obvious if $\mathfrak{w}$ is monotone in $(0,\mathfrak{u}(0))$.
Moreover, even if it is not, we can prove that 
\begin{equation}\label{sign1}
\mathfrak{w}_{\rho \rho}(a) >0 \quad \text{if there is }a \in (0,\mathfrak{u}(0)) \text{ such that }\mathfrak{w}_\rho(a)=0. 
\end{equation}
In fact, 
by Lemma \ref{lem:B1}, $\mathfrak{w}$ satisfies $-(d-2) < \mathfrak{w} < 0$ and 
\[
\begin{split}
0 & = \rho \mathfrak{w}(\rho)\frac{d}{d\rho}\left( \rho \mathfrak{w}(\rho)+ (d-2) \rho \right) + \rho^{\frac{d + \gamma}{d-2}}\\
& =\rho \mathfrak{w}(\rho) \left\{\left( \mathfrak{w}(\rho)+ (d-2) \right) + \rho \mathfrak{w}_\rho (\rho)\right\}+ \rho^{\frac{d + \gamma}{d-2}},
\end{split}
\]
that is,
\[
\mathfrak{w}_\rho = - \frac{1}{\rho} \left\{\frac{\rho^{\frac{\gamma+2}{d-2}}}{\mathfrak{w}} + ( \mathfrak{w}+ (d-2) )\right\}
= - \frac{\mathfrak{w}^2+(d-2) \mathfrak{w} +\rho^{\frac{\gamma+2}{d-2}}}{\rho \mathfrak{w}}
=: F(\rho, \mathfrak{w}).
\]
Then, $\mathfrak{w}_\rho(a)=0$ implies that 
\[
\mathfrak{w}_{\rho \rho}(a)  = F_\rho (a, \mathfrak{w}(a))
 = - \frac{\gamma +2}{d-2} \frac{a^\frac{\gamma+2}{d-2}}{a^2 \mathfrak{w}(a)}>0.
\]
Hence, \eqref{sign1} is proved. 
Since \eqref{sign1} implies that the sign of $\mathfrak{w}_\rho$ is constant near $\rho=+0$, 
$\mathfrak{w}$ is monotone near $\rho=+0$. Hence, since $-(d-2) < \mathfrak{w} < 0$, there exists a limit of $\mathfrak{w}$ as $\rho \to +0$ satisfying \eqref{eq.goal_lem5.6_1}.

Next, we will show that $m=0$ or $-(d-2)$. 
Suppose that $-(d-2)<m<0$ for contradiction. We calculate
\[
m= \lim_{\rho\to +0} \mathfrak{w}(\rho)
= \lim_{\rho\to +0} \frac{\mathfrak{v}(\rho)}{\rho}
= \lim_{t\to +\infty} \frac{\mathfrak{u}_t(t)}{\mathfrak{u}(t)}
= \lim_{t\to +\infty} (\log \mathfrak{u}(t))_t.
\]
Then, for any $\varepsilon\in (0,-m)$, there exists $T>0$ such that 
\begin{equation}\label{eq.111}
(m-\varepsilon) \mathfrak{u}(t) < 
\mathfrak{u}_t(t) < (m+\varepsilon) \mathfrak{u}(t) 
\end{equation}
and 
\begin{equation}\label{eq.112}
m-\varepsilon <(\log \mathfrak{u}(t))_t < m+\varepsilon
\end{equation}
for any $t> T$. Integrating \eqref{eq.112} over $[t,T]$ gives 
\begin{equation}\label{eq.113}
\mathfrak{u}(T)e^{(m-\varepsilon)(t-T)} < \mathfrak{u}(t) < \mathfrak{u}(T)e^{(m+\varepsilon)(t-T)} < \mathfrak{u}(0)e^{(m+\varepsilon)(t-T)} 
\end{equation}
for any $t> T$, and hence,
\begin{equation}\label{eq.ss_2}
\mathfrak{u}(t)^{\frac{d + \gamma}{d-2}} < \mathfrak{u}(0)^{\frac{d + \gamma}{d-2}} e^{\frac{(m+\varepsilon)(d + \gamma)}{d-2}(t-T)} 
\end{equation}
for any $t> T$. 
By \eqref{eq.111} and \eqref{eq.113},
we also have 
\begin{equation}\label{eq.ss_1}
\begin{split}
\mathfrak{u}_t(t) + (d-2)\mathfrak{u}(t) 
& > \{ (d-2) + m - \varepsilon\} \mathfrak{u}(t)\\
& >\{ (d-2) + m - \varepsilon\} \mathfrak{u}(T) e^{(m-\varepsilon)(t-T)}. 
\end{split}
\end{equation}
By combining \eqref{IE-v}, \eqref{eq.ss_2} and \eqref{eq.ss_1}, 
we have
\[
\begin{split}
\{ (d-2) + m - \varepsilon\} \mathfrak{u}(T) e^{(m-\varepsilon)(t-T)}
& < 
\mathfrak{u}_t(t)+ (d-2) \mathfrak{u}(t)\\
& = \int_t^\infty \mathfrak{u}(\tau)^{\frac{d + \gamma}{d-2}} d\,\tau \\
& < 
\int_t^\infty
\mathfrak{u}(0)^{\frac{d + \gamma}{d-2}} e^{\frac{(m+\varepsilon)(d + \gamma)}{d-2}(\tau-T)}d\,\tau\\
& = C e^{\frac{(m+\varepsilon)(d + \gamma)}{d-2}(t-T)}
\end{split}
\]
for any $t> T$. Then, if we further assume $(d-2) + m - \varepsilon>0$, 
we have
\[
\mathfrak{u}(T) <  C e^{\{\frac{(m+\varepsilon)(d + \gamma)}{d-2} - (m-\varepsilon)\}(t-T)}
\]
for any $t> T$. 
However, as $t \to \infty$, this contradicts that 
\[
\mathfrak{u}(T) \ge C_\varepsilon \quad \text{for some constant $C_\varepsilon >0$},
\] 
if we fix $\varepsilon$ sufficiently small so that
\[
\frac{(m+\varepsilon)(d + \gamma)}{d-2} - (m-\varepsilon) < 0\quad \text{i.e.}\quad 
0 < \varepsilon < \frac{(2+ \gamma)|m|}{2d + \gamma-2}. 
\]
Therefore, $m$ must be $0$ or $-(d-2)$, which means \eqref{eq.goal_lem5.6}. Thus, we prove Lemma~\ref{lem:key_sec5_2}.  
\end{proof}

%%% Acknowledgement %%%
\section*{Acknowledgement}
%\par
%The first author is supported by Grant-in-Aid for Young Scientists (B) 
%(No. 17K14216) and Challenging Research (Pioneering) (No.17H06199), 
%Japan Society for the Promotion of Science. 

The first author is supported by Grant-in-Aid for Early-Career Scientists (No. 21K13821), Japan
Society for the Promotion of Science.
The second author is supported by JST CREST (No. JPMJCR1913), Japan and 
the Grant-in-Aid for Scientific Research (B) (No.18H01132) and 
Young Scientists Research (No.19K14581), 
JSPS.
The third author is supported by 
Grant for Basic Science Research Projects from The Sumitomo Foundation (No.210788). 
The fourth author is supported by  the Laboratoire \'Equations aux D\'eriv\'ees Partielles LR03ES04.
We thank the anonymous reviewers for the careful reading of our manuscript and the insightful comments and suggestions.

%%%%%%%%%%%%%%%%%%%%%%%%%%%%%%%%%
%%%%%%%%%%%%% References %%%%%%%%%%%%%
%%%%%%%%%%%%%%%%%%%%%%%%%%%%%%%%%


\begin{bibdiv}
 \begin{biblist}[\normalsize]
 
 \bib{Avi1983}{article}{
   author={P., Aviles\text{,}},
   title={On isolated singularities in some nonlinear partial differential equations},
   journal={Indiana Univ. Math. J.},
   volume={32},
   date={1983},
   number={5},
   pages={773--791},
%   issn={0362-546X},
}
 

 \bib{Avi1987}{article}{
   author={P., Aviles\text{,}},
   title={Local behavior of solutions of some elliptic equations},
   journal={Comm. Math. Phys.},
   volume={108},
   date={1987},
   number={2},
   pages={177--192},
%   issn={0362-546X},
}

%\bib{BCD2011}{book}{
%   author={H., Bahouri\text{,}},
%   author={J.-Y., Chemin\text{,}},
%   author={R., Danchin\text{,}},
%   title={Fourier analysis and nonlinear partial differential equations},
%   series={Grundlehren der Mathematischen Wissenschaften},
%   volume={343},
%   publisher={Springer},
%   place={Heidelberg},
%   date={2011},
%   pages={xvi+523},
%%   isbn={978-3-642-16829-1},
%%   review={\MR{2768550 (2011m:35004)}},
%}

\bib{Bal1977}{article}{
   author={M., Ball\text{,}},
   title={Remarks on blow-up and nonexistence theorems for nonlinear evolution equations},
   journal={Quart. J. Math.},
   volume={28},
   date={1977},
   pages={473--486},
%   issn={0362-546X},
%   review={\MR{3606306}},
%   doi={10.1016/j.na.2016.12.008},
}

\bib{Bar1983}{article}{
   author={P., Baras\text{,}},
   title={Non-unicit\'e des solutions d'une \'equation d'\'evolution non-lin\'eaire},
   journal={Annales de la Facult\'e des sciences de Toulouse: Math\'ematiques},
   volume = {5}
   number = {3-4},
   date={1983},
   pages={287--302},
%   issn={0362-546X},
%   review={\MR{3606306}},
%   doi={10.1016/j.na.2016.12.008},
}


%\bib{BL1989}{article}{
%   author={C., Bandle\text{,}},
%   author={H. A., Levine\text{,}},
%   title={On the existence and nonexistence of global solutions of reaction-diffusion equations in sectorial domains},
%   journal={Trans. Amer. Math. Soc.},
%   volume={316},
%   date={1989},
%   number={2},
%   pages={595--622},
%%   issn={0362-546X},
%%   review={\MR{3606306}},
%%   doi={10.1016/j.na.2016.12.008},
%}

%\bib{BarGol1984}{article}{
%   author={P., Baras\text{,}},
%   author={J. A., Goldstein\text{,}},
%   title={The heat equation with a singular potential},
%   journal={Trans. Am. Math. Soc.},
%   volume={284},
%   date={1984},
%   number={1},
%   pages={121--139},
%%   issn={0362-546X},
%%   review={\MR{3606306}},
%%   doi={10.1016/j.na.2016.12.008},
%}


\bib{BenTayWei2017}{article}{
   author={B., Ben Slimene\text{,}},
   author={S., Tayachi\text{,}},
   author={F. B., Weissler\text{,}},
   title={Well-posedness, global existence and large time behavior for
   Hardy-H\'enon parabolic equations},
   journal={Nonlinear Anal.},
   volume={152},
   date={2017},
   pages={116--148},
%   issn={0362-546X},
%   review={\MR{3606306}},
%   doi={10.1016/j.na.2016.12.008},
}

\bib{BS1988}{book}{
   author={C., Bennett\text{,}},
   author={R. C., Sharpley\text{,}},
   title={Interpolation of operators},
%   series={Mathematical Surveys and Monographs},
%   volume={187},
   publisher={Academic press},
   date={1988},
%   pages={xxiv+299},
%   isbn={978-0-8218-9152-0},
%   review={\MR{3052352}},
%   doi={10.1090/surv/187},
}

\bib{BL1976}{book}{
   author={J., Bergh\text{,}},
   author={J., L\"ofstr\"om\text{,}},
   title={Interpolation Spaces. An Introduction},
   series={Grundlehren der Mathematishe Wissenschaften},
   volume={223},
   publisher={Springer, Berlin-Heidelberg-New York},
   date={1976},
   pages={pp. 1--202},
%   isbn={978-0-8218-9152-0},
%   review={\MR{3052352}},
%   doi={10.1090/surv/187},
}



%\bib{Bic2019}{article}{
%   author={U., Biccari\text{,}},
%   title={Boundary controllability for a one-dimensional heat equation with a singular inverse-square potential},
%   journal={Math. Control Relat. Fields},
%   volume={9},
%   date={2019},
%   number={1},
%   pages={191--219},
%%   issn={0021-7670},
%%   review={\MR{1403259}},
%%   doi={10.1007/BF02790212},
%}

\bib{BidGar2001}{article}{
   author={M.-F., Bidaut-V\'eron\text{,}},
   author={M., Garc\'ia-Huidobro\text{,}},
   title={Regular and singular solutions of a quasilinear equation with weights},
   journal={Asymptotic Analysis},
   volume={28},
   date={2001},
   pages={115--150},
%   issn={0021-7670},
%   review={\MR{1403259}},
%   doi={10.1007/BF02790212},
}

\bib{BidRao1996}{article}{
   author={M.-F., Bidaut-V\'eron\text{,}},
   author={T., Raoux\text{,}},
   title={Asymptotics of solutions of some nonlinear elliptic systems},
   journal={Commun. Partial. Differ. Equ.},
   volume={21},
   date={1996},
   number={7-8},
   pages={1035--1086},
%   issn={0021-7670},
%   review={\MR{1403259}},
%   doi={10.1007/BF02790212},
}

%\bib{Br2011}{book}{
%   author={H., Brezis\text{,}},
%   title={Functional analysis, Sobolev spaces and partial differential
%   equations},
%   series={Universitext},
%   publisher={Springer, New York},
%   date={2011},
%   pages={xiv+599},
%%   isbn={978-0-387-70913-0},
%%   review={\MR{2759829 (2012a:35002)}},
%}

\bib{BreCaz1996}{article}{
   author={H., Brezis\text{,}},
   author={T., Cazenave\text{,}},
   title={A nonlinear heat equation with singular initial data},
   journal={J. Anal. Math.},
   volume={68},
   date={1996},
   pages={277--304},
%   issn={0021-7670},
%   review={\MR{1403259}},
%   doi={10.1007/BF02790212},
}

%\bib{CGL2022}{article}{
%   author={R., Castillo\text{,}},
%   author={O., Guzm\'an-Rea\text{,}},
%   author={M., Loayza\text{,}},
%   title={On the local existence for Hardy parabolic equations with singular initial data},
%   journal={J. Math. Anal. Appl.},
%   volume={510},
%   date={2022},
%   number={2},
%   pages={126022},
%%   issn={0273-0979},
%%   review={\MR{1824891}},
%%   doi={10.1090/S0273-0979-01-00903-X},
%}

\bib{Chi2019}{article}{
   author={N., Chikami\text{,}},
   title={Composition estimates and well-posedness for Hardy-H\'{e}non parabolic
   equations in Besov spaces},
   journal={J. Elliptic Parabol. Equ.},
   volume={5},
   date={2019},
   number={2},
   pages={215--250},
%   issn={2296-9020},
%   review={\MR{4031955}},
%   doi={10.1007/s41808-019-00039-8},
}

\bib{CIT2021}{article}{
   author={N., Chikami\text{,}},
   author={M., Ikeda\text{,}},
   author={K., Taniguchi\text{,}},
   title={Well-posedness and global dynamics for the critical Hardy-Sobolev parabolic equation},
   journal={Nonlinearity},
   volume={34},
   date={2021},
   number={11},
   pages={8094--8142},
%   issn={0273-0979},
%   review={\MR{1824891}},
%   doi={10.1090/S0273-0979-01-00903-X},
}

\bib{CIT2022}{article}{
   author={N., Chikami\text{,}},
   author={M., Ikeda\text{,}},
   author={K., Taniguchi\text{,}},
   title={Optimal well-posedness and forward self-similar solution for the Hardy-H\'enon parabolic equation in critical weighted Lebesgue spaces},
   journal={Nonlinear Anal.},
   volume={222},
   date={2022},
%   number={3},
   pages={112931},
%   issn={0273-0979},
%   review={\MR{1824891}},
%   doi={10.1090/S0273-0979-01-00903-X},
}

\bib{DenapoliDrelichman}{article}{
   author={P. L., De N\'{a}poli\text{,}},
   author={I., Drelichman\text{,}},
   title={Weighted convolution inequalities for radial functions},
   journal={Ann. Mat. Pura Appl. (4)},
   volume={194},
   date={2015},
%   number={4},
   pages={167--181},
%   issn={0273-0979},
%   review={\MR{1824891}},
%   doi={10.1090/S0273-0979-01-00903-X},
}

\bib{DZarxiv}{article}{
   author={G., Du\text{,}},
   author={S., Zhou\text{,}},
   title={Local and global properties of $p$-Laplace H\'enon equations},
   journal={J. Math. Anal. Appl.},
   volume={535},
   date={2024},
   number={2},
   pages={128148},
%   issn={0273-0979},
%   review={\MR{1824891}},
%   doi={10.1090/S0273-0979-01-00903-X},
}

\bib{ELM2021}{article}{
   author={D. E., Edmunds\text{,}},
   author={J., Lang\text{,}},
   author={Z., Mihula\text{,}},
   title={Measure of noncompactness of Sobolev embeddings on strip-like domains},
   journal={J. Approx. Theory},
   volume={269},
   date={2021},
%   number={2},
   pages={105608},
%   issn={0273-0979},
%   review={\MR{1824891}},
%   doi={10.1090/S0273-0979-01-00903-X},
}

\bib{Elona_D}{article}{
   author={A., Elona\text{,}},
   title={Boundedness of the Hilbert transform on weighted Lorentz spaces},
   journal={PhD thesis, Universitat de Barcelona},
%   volume={194},
   date={2012},
%   number={4},
%   pages={167--181},
%   issn={0273-0979},
%   review={\MR{1824891}},
%   doi={10.1090/S0273-0979-01-00903-X},
}

\bib{GhoMor2013}{book}{
   author={N., Ghoussoub\text{,}},
   author={A., Moradifam\text{,}},
   title={Functional inequalities: new perspectives and new applications},
   series={Mathematical Surveys and Monographs},
   volume={187},
   publisher={American Mathematical Society, Providence, RI},
   date={2013},
   pages={xxiv+299},
%   isbn={978-0-8218-9152-0},
%   review={\MR{3052352}},
%   doi={10.1090/surv/187},
}

\bib{GidSpr1981}{article}{
   author={B., Gidas\text{,}},
   author={J., Spruck\text{,}},
   title={Global and local behavior of positive solutions of nonlinear elliptic equations},
   journal={Commun. Pure Appl. Anal.},
   volume={34},
   date={1981},
   number={4},
   pages={525--598},
%   issn={0022-0396},
%   review={\MR{964617}},
%   doi={10.1016/0022-0396(88)90068-X},
}

\bib{Gig1986}{article}{
   author={M., Giga\text{,}},
   title={Solutions for semilinear parabolic equations in $L^p$ and regularity of weak solutions of the Navier-Stokes system},
   journal={J. Differential Equations},
   volume={62},
   date={1986},
   number={2},
   pages={186--212},
%   issn={0022-0396},
%   review={\MR{964617}},
%   doi={10.1016/0022-0396(88)90068-X},
}

\bib{GN2022}{article}{
   author={M., Giga\text{,}},
   author={Q. A., Ng{\^o}\text{,}},
   title={Exhaustive existence and non-existence results for Hardy-H\'enon equations in $\mathbf{R}^n$},
   journal={Partial Differential Equations and Applications},
   volume={3},
   date={2022},
   number={6},
   pages={81},
%   issn={0022-0396},
%   review={\MR{964617}},
%   doi={10.1016/0022-0396(88)90068-X},
}

\bib{Gr2008}{book}{
   author={L., Grafakos\text{,}},
   title={Classical Fourier analysis},
   series={Graduate Texts in Mathematics},
   volume={249},
   edition={2},
   publisher={Springer, New York},
   date={2008},
   pages={xvi+489},
%   isbn={978-0-387-09431-1},
%   review={\MR{2445437 (2011c:42001)}},
}

%\bib{Gr2009}{book}{
%   author={L., Grafakos\text{,}},
%   title={Modern Fourier analysis},
%   series={Graduate Texts in Mathematics},
%   volume={250},
%   edition={2},
%   publisher={Springer, New York},
%   date={2009},
%   pages={xvi+504},
%%   isbn={978-0-387-09433-5},
%%   review={\MR{2463316 (2011d:42001)}},
%%   doi={10.1007/978-0-387-09434-2},
%}



\bib{GueVer1988}{article}{
   author={M., Guedda\text{,}},
   author={L., V\'{e}ron\text{,}},
   title={Local and global properties of solutions of quasilinear elliptic
   equations},
   journal={J. Differential Equations},
   volume={76},
   date={1988},
   number={1},
   pages={159--189},
%   issn={0022-0396},
%   review={\MR{964617}},
%   doi={10.1016/0022-0396(88)90068-X},
}

\bib{HarWei1982}{article}{
   author={A., Haraux\text{,}},
   author={F. B., Weissler\text{,}},
   title={Non-uniqueness for a semilinear initial value problem},
   journal={Indiana Univ. Math. J.},
   volume={31},
   date={1982},
%   number={1},
   pages={167--189},
%   issn={0022-0396},
%   review={\MR{964617}},
%   doi={10.1016/0022-0396(88)90068-X},
}

\bib{H-1973}{article}{
   author={M., H\'enon\text{,}},
   title={Numerical experiments on the stability of spherical stellar systems},
   journal={Astron. Astrophys},
   volume={24},
   date={1973},
%   number={9},
   pages={229--238},
%   issn={0022-0396},
%   review={\MR{3765769}},
%   doi={10.1016/j.jde.2018.01.023},
}

\bib{hi}{article}{
author={M., Hirose\text{,}},
title={Existence of global solutions for a semilinear parabolic Cauchy problem},
journal={Differential Integral Equations},
volume={21},
 date={2008},
pages={623--652}
}


\bib{IKNW2021}{article}{
author={S., Ibrahim\text{,}},
author={H., Kikuchi\text{,}},
author={K., Nakanishi\text{,}},
author={J., Wei\text{,}},
title={Non-uniqueness for an energy-critical heat equation on $\mathbb R^2$},
journal={Math. Ann.},
volume={380},
 date={2021},
pages={317--348}
}

\bib{ITW2021arxiv}{article}{
   author={M., Ikeda\text{,}},
   author={K., Taniguchi\text{,}},
   author={Y., Wakasugi\text{,}},
   title={Global existence and asymptotic behavior for semilinear damped wave equations on measure spaces},
   journal={to appear in Evol. Equ. Control Theory},
%%   volume={38},
   date={2024},
%%   number={3},
%%   pages={273--291},
%%   issn={0273-0979},
%%   review={\MR{1824891}},
%%   doi={10.1090/S0273-0979-01-00903-X},
}

\bib{IRT2022}{article}{
author={N., Ioku\text{,}},
author={B., Ruf\text{,}},
author={E., Terraneo\text{,}},
title={Non-uniqueness for a critical heat equation in two dimensions with singular data},
journal={Ann. Inst. H. Poincar\'e Anal. Non Lin\'eaire},
volume={36},
   number={7},
 date={2022},
pages={2027--2051}
}

%\bib{Kat76}{book}{
%   author={T., Kato\text{,}},
%   title={Perturbation theory for linear operators},
%   edition={2},
%   note={Grundlehren der Mathematischen Wissenschaften, Band 132},
%   publisher={Springer-Verlag, Berlin-New York},
%   date={1976},
%   pages={xxi+619},
%%   review={\MR{0407617}},
%}

\bib{KS2014}{article}{
author={A. Y., Karlovich\text{,}},
author={I. M., Spitkovsky\text{,}},
title={The Cauchy singular integral operator on weighted variable Lebesgue spaces},
journal={In: Concrete Operators, Spectral Theory, Operators in Harmonic Analysis and Approximation: 22nd International Workshop in Operator Theory and its Applications, Sevilla, July 2011, Springer Basel, 2014},
%volume={36},
%   number={7},
% date={2014},
pages={275--291}
}

\bib{Kerman}{article}{
   author={R. A., Kerman\text{,}},
   title={Convolution theorems with weights},
   journal={Trans. Amer. Math. Soc.},
   volume={280},
   date={1983},
%   number={4},
   pages={207--219},
%   issn={0273-0979},
%   review={\MR{1824891}},
%   doi={10.1090/S0273-0979-01-00903-X},
}

%\bib{KS2016}{article}{
%   author={P. S. Kumari\text{,}},
%   author={M. Sarwar\text{,}},
%   title={Some fixed point theorems in generating space of $b$-quasi-metric family},
%   journal={SpringerPlus},
%   volume={5},
%   date={2016},
%   number={268},
%%   pages={207--219},
%%   issn={0273-0979},
%%   review={\MR{1824891}},
%%   doi={10.1090/S0273-0979-01-00903-X},
%}


\bib{Le2002}{book}{
   author={P. G., Lemari\'e-Rieusset\text{,}},
   title={Recent developments in the Navier-Stokes problem},
%   series={Graduate Texts in Mathematics},
%   volume={249},
   edition={1},
   publisher={CRC Press},
   date={2002},
   pages={xvi+408},
   doi={https://doi.org/10.1201/9780367801656}
%   isbn={978-0-387-09431-1},
%   review={\MR{2445437 (2011c:42001)}},
}

\bib{Lie1983}{article}{
   author={E. H., Lieb\text{,}},
   title={Sharp constants in the Hardy-Littlewood-Sobolev and related inequalities},
   journal={Ann. Math.},
   volume={118},
   date={1983},
%   number={4},
   pages={349--374},
%   issn={0273-0979},
%   review={\MR{1824891}},
%   doi={10.1090/S0273-0979-01-00903-X},
}

\bib{LN2022}{article}{
   author={E., Lorist\text{,}},
   author={Z., Nieraeth\text{,}},
   title={Sparse domination implies vector-valued sparse domination},
   journal={Math. Z.},
   volume={301},
   date={2022},
%   number={4},
   pages={1107--1141.},
%   issn={0273-0979},
%   review={\MR{1824891}},
%   doi={10.1090/S0273-0979-01-00903-X},
}

\bib{LNarxiv}{article}{
   author={E., Lorist\text{,}},
   author={Z., Nieraeth\text{,}},
   title={Banach function spaces done right},
   journal={to appear in Indag. Math.},
%   volume={301},
   date={2024},
%   number={4},
%   pages={1107--1141.},
%   issn={0273-0979},
%   review={\MR{1824891}},
%   doi={10.1090/S0273-0979-01-00903-X},
}

\bib{MatosTerrane}{article}{
AUTHOR = {J., Matos\text{,}}
AUTHOR = {E., Terraneo\text{,}},
TITLE = {Nonuniqueness for a critical nonlinear heat equation with any
              initial data},
   JOURNAL = {Nonlinear Anal.},
    VOLUME = {55},
      YEAR = {2003},
  %  NUMBER = {7-8},
     PAGES = {927--936},
  }

\bib{Mey1997}{article}{
   author={Y., Meyer\text{,}},
   title={Wavelets, paraproducts, and Navier-Stokes equations},
   conference={
      title={Current developments in mathematics, 1996 (Cambridge, MA)},
   },
   book={
      publisher={Int. Press, Boston, MA},
   },
   date={1997},
   pages={105--212},
%   review={\MR{1724946}},
}

\bib{NTY2004}{article}{
   author={E., Nakai\text{,}},
   author={N., Tomita\text{,}},
   author={K., Yabuta\text{,}},
   title={Density of the set of all infinitely differentiable functions with compact support in weighted Sobolev spaces},
   journal={Sci. Math. Jpn.},
   volume={60},
   date={2004},
   number={1},
   pages={121--127},
}

\bib{NS1985}{article}{
   author={W. -M., Ni\text{,}},
   author={P. E., Sacks\text{,}},
   title={Singular behaviour in nonlinear parabolic equations},
   journal={Trans. Amer. Math. Soc.},
   volume={287},
   date={1985},
   pages={657--671},
}


\bib{OkaTsu2016}{article}{
   author={T., Okabe\text{,}},
   author={Y., Tsutsui\text{,}},
   title={Navier-Stokes flow in the weighted Hardy space with
applications to time decay problem},
   journal={J. Differential Equations},
   volume={216},
   date={2016},
   number={3},
   pages={1712--1755},
%   issn={0022-247X},
%   review={\MR{4080035}},
%   doi={10.1016/j.jmaa.2020.123976},
}

\bib{ONe1963}{article}{
   author={R., O'Neil\text{,}},
   title={Convolution operators and $L(p,\,q)$ spaces},
   journal={Duke Math. J.},
   volume={30},
   date={1963},
   pages={129--142},
   issn={0012-7094},
%   review={\MR{146673}},
}

\bib{Pin1997}{article}{
   author={R. G., Pinsky\text{,}},
   title={Existence and nonexistence of global solutions for $u_t = \Delta u + a(x)u^p$ in $\mathbf{R}^d$},
   journal={J. Differential Equations},
   volume={133},
   date={1997},
   pages={152--177},
}

\bib{Qi1998}{article}{
   author={Y., Qi\text{,}},
   title={The critical exponents of parabolic equations and blow-up in $\mathbf{R}^n$},
   journal={Proc. Roy. Soc. Edinburgh Sect. A},
   volume={128},
   date={1998},
   pages={123--136},
}

\bib{SS2004}{book}{
   author={T., Senba\text{,}},
   author={T., Suzuki\text{,}},
   title={Applied Analysis: Mathematical Methods in Natural Science},
%   series={Graduate Texts in Mathematics},
%   volume={249},
%   edition={1},
   publisher={Imperial College Press, London},
   date={2004},
%   pages={xvi+408},
%   doi={https://doi.org/10.1201/9780367801656}
%   isbn={978-0-387-09431-1},
%   review={\MR{2445437 (2011c:42001)}},
}

\bib{Ser1964}{article}{
   author={J., Serrin\text{,}},
   title={Local behavior of solutions of quasi-linear equations},
   journal={Acta Math.},
   volume={111},
   date={1964},
%   number={},
   pages={247--302},
%   issn={0022-247X},
%   review={\MR{4080035}},
%   doi={10.1016/j.jmaa.2020.123976},
}

\bib{Ser1965}{article}{
   author={J., Serrin\text{,}},
   title={Isolated singularities of solutions of quasi-linear equation},
   journal={Acta Math.},
   volume={113},
   date={1965},
%   number={},
   pages={219--240},
%   issn={0022-247X},
%   review={\MR{4080035}},
%   doi={10.1016/j.jmaa.2020.123976},
}

%\bib{Tao_lec}{article}{
%   author={T., Tao\text{,}},
%   title={},
%   journal={Lecture notes 1 for 245C: Interpolation of $L^p$ spaces},
%%   volume={488},
%%   date={2020},
%%   number={1},
%%   pages={123976, 51},
%%   issn={0022-247X},
%%   review={\MR{4080035}},
%%   doi={10.1016/j.jmaa.2020.123976},
%}

\bib{Tak2021}{article}{
   author={J., Takahashi\text{,}},
   title={Existence of solutions with moving singularities for a semilinear heat equation with a critical exponent},
   journal={J. Math. Pures Appl.},
   volume={148},
   date={2021},
   pages={128--149},
}

\bib{Tay2020}{article}{
   author={S., Tayachi\text{,}},
   title={Uniqueness and non-uniqueness of solutions for critical
   Hardy-H\'{e}non parabolic equations},
   journal={J. Math. Anal. Appl.},
   volume={488},
   date={2020},
   number={1},
   pages={123976, 51},
%   issn={0022-247X},
%   review={\MR{4080035}},
%   doi={10.1016/j.jmaa.2020.123976},
}

\bib{TW}{article}{
   author={S., Tayachi\text{,}},
   author={F. B., Weissler\text{,}},
   title={New life-span results for the nonlinear heat equation},
   journal={J. Differential Equations},
   volume={373},
   date={2023},
%   number={},
   pages={564--625},
%   issn={},
%   review={},
%   doi={},
}

\bib{Ter2002}{article}{
   author={E., Terraneo\text{,}},
   title={Non-uniqueness for a critical non-linear heat equation},
   journal={Comm. Partial Differential Equations},
   volume={27},
   date={2002},
   number={1-2},
   pages={185--218},
%   issn={0360-5302},
%   review={\MR{1886959}},
%   doi={10.1081/PDE-120002786},
}

\bib{Tsu2011}{article}{
   author={Y., Tsutsui\text{,}},
   title={The Navier-Stokes equations and weak Herz spaces},
   journal={Adv. Differential Equations},
   volume={16},
   date={2011},
   number={11-12},
   pages={1049--1085},
%   issn={0360-5302},
%   review={\MR{1886959}},
%   doi={10.1081/PDE-120002786},
}

%\bib{Qi1998}{article}{
%   author={Y., Qi\text{,}},
%   title={The critical exponents of parabolic equations and blow-up in ${\bf R}^n$},
%   journal={Proc. Roy. Soc. Edinburgh Sect. A},
%   volume={128},
%   date={1998},
%   number={1},
%   pages={123--136},
%%   issn={0308-2105},
%%   review={\MR{1606357}},
%%   doi={10.1017/S0308210500027190},
%}
%
%\bib{WanHuoHaoGuo2011}{book}{
%   author={B., Wang\text{,}},
%   author={Z., Huo\text{,}},
%   author={C., Hao\text{,}},
%   author={Z., Guo\text{,}},
%   title={Harmonic analysis method for nonlinear evolution equations. I},
%   publisher={World Scientific Publishing Co. Pte. Ltd., Hackensack, NJ},
%   date={2011},
%   pages={xiv+283},
%%   isbn={978-981-4360-73-9},
%%   isbn={981-4360-73-2},
%%   review={\MR{2848761}},
%%   doi={10.1142/9789814360746},
%}

\bib{Wan1993}{article}{
   author={X., Wang\text{,}},
   title={On the Cauchy problem for reaction-diffusion equations},
   journal={Trans. Amer. Math. Soc.},
   volume={337},
   date={1993},
   number={2},
   pages={549--590},
%   issn={0002-9947},
%   review={\MR{1153016}},
%   doi={10.2307/2154232},
}

%
\bib{WWYarxiv}{article}{
   author={Y., Wang\text{,}},
   author={W., Wei\text{,}},
   author={Y., Ye\text{,}},
   title={Gagliardo-Nirenberg inequalities in Lorentz type spaces and energy equality for the Navier-Stokes system},
   journal={arXiv:2106.11212},
%   volume={38},
%   date={2001},
%   number={3},
%   pages={273--291},
%   issn={0273-0979},
%   review={\MR{1824891}},
%   doi={10.1090/S0273-0979-01-00903-X},
}
%

%\bib{Wei1979}{article}{
%   author={F. B., Weissler\text{,}},
%   title={Semilinear evolution equations in Banach spaces},
%   journal={J. Functional Analysis},
%   volume={32},
%   date={1979},
%   number={3},
%   pages={277--296},
%%   issn={0022-1236},
%%   review={\MR{538855}},
%%   doi={10.1016/0022-1236(79)90040-5},
%}

%\bib{WeiWu2022}{article}{
%   author={J., Wei\text{,}},
%   author={K., Wu\text{,}},
%   title={Local behavior of solutions to a fractional equation with isolated singularity and critical Serrin exponent},
%   journal={Discrete Contin. Dyn. Syst.},
%   volume={42},
%   date={2022},
%%   number={3},
%   pages={4031--4050},
%%   issn={0022-1236},
%%   review={\MR{538855}},
%%   doi={10.1016/0022-1236(79)90040-5},
%}


\bib{Wei1980}{article}{
   author={F. B., Weissler\text{,}},
   title={Local existence and nonexistence for semilinear parabolic equations in $L^p$},
   journal={Indiana Univ. Math. J.},
   volume={29},
   date={1980},
%   number={3},
   pages={79--102},
%   issn={0022-1236},
%   review={\MR{538855}},
%   doi={10.1016/0022-1236(79)90040-5},
}

\bib{Wei1981}{article}{
   author={F. B., Weissler\text{,}},
   title={Existence and nonexistence of global solutions for a semilinear heat equation},
   journal={Israel J. Math.},
   volume={38},
   date={1981},
%   number={3},
   pages={29--40},
%   issn={0022-1236},
%   review={\MR{538855}},
%   doi={10.1016/0022-1236(79)90040-5},
}

\bib{Yap1969}{article}{
   author={L. Y. H., Yap\text{,}},
   title={Some remarks on convolution operators and $L(p,q)$ spaces},
   journal={Duke Math. J.},
   volume={36},
   date={1969},
   number={4},
   pages={647--658},
}

%\bib{Yom2022}{article}{
%   author={G. D., Yomgne\text{,}},
%   title={On the generalized parabolic Hardy-H\'enon equation: Existence, blow-up, self-similarity and large-time asymptotic behavior},
%   journal={Differ. Integral Equ.},
%   volume={35},
%   date={2022},
%   number={1},
%   pages={57--88},
%}

 \end{biblist}
\end{bibdiv}
\end{document}